\documentclass[12pt]{amsart}
\usepackage{amsmath,amssymb,amsthm, tikz-cd}
\usepackage[shortlabels]{enumitem}
\usepackage{graphicx}
\usepackage[margin=1in, textwidth=125mm, textheight=185mm]{geometry}
\usepackage[colorlinks=true,linkcolor=blue,urlcolor=blue,pagebackref]{hyperref}

\newtheorem{theorem}{Theorem}[section]
\newtheorem*{theorem*}{Theorem}
\newtheorem{corollary}[theorem]{Corollary}
\newtheorem*{corollary*}{Corollary}
\newtheorem{lemma}[theorem]{Lemma}
\newtheorem{proposition}[theorem]{Proposition}
\newtheorem*{thmdef*}{Theorem/Definition}

\newtheorem{theoremx}{Theorem}
\newtheorem{corollaryx}[theoremx]{Corollary}

\theoremstyle{definition}
\newtheorem{remark}[theorem]{Remark}
\newtheorem{setup}[theorem]{Setup}

\newtheorem{example}[theorem]{Example}
\newtheorem{definition}[theorem]{Definition}

\newtheorem*{definition*}{Definition}

\newtheorem{question}[theorem]{Question} 
\newcommand{\Z}{\mathbb{Z}}

\newcommand{\cC}{\mathcal{C}}
\newcommand{\cJ}{\mathcal{J}}

\newcommand{\cG}{\mathcal{G}}

\newcommand{\cA}{\mathcal{A}}
\newcommand{\cB}{\mathcal{B}}

\newcommand{\lp}{\left(}
\newcommand{\rp}{\right)}
\newcommand{\sq}{\subseteq}
\newcommand{\ds}{\dots}
\newcommand{\cds}{\cdots}
\newcommand{\arrow}{arrow}
\newcommand{\F}{\mathbb{F}}
\newcommand{\Q}{\mathbb{Q}}
\newcommand{\N}{\mathbb{N}}

\newcommand{\fm}{\mathfrak{m}}
\newcommand{\fa}{\mathfrak{a}}
\newcommand{\fb}{\mathfrak{b}}

\newcommand{\fn}{\mathfrak{n}}

\newcommand{\fc}{\mathfrak{c}}

\renewcommand{\t}[1]{\widetilde{#1}}

\newcommand{\ul}[1]{\underline{#1}}

\newcommand{\aae}{\alpha_{<e}}
\newcommand{\aaa}{\alpha_{<a}}
\newcommand{\aaae}{\alpha_{<ea}}
\newcommand{\cts}{\mathrm{Cont}}

\DeclareMathOperator{\Hom}{Hom}
\DeclareMathOperator{\End}{End}

\DeclareMathOperator{\Spec}{Spec}
\DeclareMathOperator{\Frac}{Frac}

\DeclareMathOperator{\id}{id}

\DeclareMathOperator{\Ann}{Ann}

\DeclareMathOperator{\Dim}{Dim}

\DeclareMathOperator{\e}{e}
\DeclareMathOperator{\ord}{ord}

\DeclareMathOperator{\Fun}{Fun}

\DeclareMathOperator{\BSR}{BSR}
\DeclareMathOperator{\Max}{Max}
\DeclareMathOperator{\mult}{mult}

\newcommand{\FF}{\mathbb{F}}
\newcommand{\KK}{\mathbb{K}}
\newcommand{\CC}{\mathbb{C}}
\newcommand{\QQ}{\mathbb{Q}}
\newcommand{\RR}{\mathbb{R}}

\newcommand{\ZZ}{\mathbb{Z}}
\newcommand{\NN}{\mathbb{Z}_{\geq 0}}
\newcommand{\Zp}{\widehat{\mathbb{Z}}_{(p)}}

\newcommand{\simto}{\xrightarrow{\sim}}

\newcommand{\Mustata}{Musta\c{t}\u{a}}

\newcommand{\bfs}{\boldsymbol{f^s}}
\newcommand{\bfa}{\boldsymbol{f^{\alpha}}}
\newcommand{\ut}{\underline{t}}

\DeclareMathOperator{\lct}{lct}
\DeclareMathOperator{\fpt}{fpt}

\newcommand{\bone}{\underline{1}}

\newcommand{\up}[1]{\left\lceil #1 \right\rceil}

\newcommand{\tr}[2]{\left \langle {#1} \right \rangle_{#2}} 

\newcommand{\base}{\operatorname{base}}
\newcommand{\digit}[2]{ {#1}^{(#2)}}

\setlist[enumerate]{itemsep=2pt, topsep=2pt, itemindent=10pt, label=(\roman*)}

\numberwithin{equation}{section}
\renewcommand{\epsilon}{\varepsilon}

\begin{document}
	
\title[Bernstein-Sato for singular rings in char. $p$]{Bernstein-Sato theory for singular rings in positive characteristic}

\author{Jack Jeffries}
\address{University of Nebraska-Lincoln, Lincoln, NE~68588, USA}
\email{jack.jeffries@unl.edu}

\author{Luis N\'u\~nez-Betancourt}
\address{Centro de Investigaci\'on en Matem\'aticas, Guanajuato, Gto., M\'exico}
\email{luisnub@cimat.mx}

\author{Eamon Quinlan-Gallego}
\address{Department of Mathematics, University of Utah, Salt Lake City, UT 84112, USA}
\email{quinlan@math.utah.edu}
\thanks{The first author was partially supported by NSF CAREER Award DMS-2044833. The second author was partially supported by CONACYT Grant 284598 and C\'atedras Marcos Moshinsky. The third author was partially supported by NSF DMS grants 1801697, 1840190 and 1840234.}

\subjclass[2010]{Primary: 14F10, 13N10, 13A35; Secondary:  14B05.}
\keywords{$D$-module, Bernstein--Sato polynomial, Methods in prime characteristic.}

	\maketitle

\begin{abstract} The Bernstein-Sato polynomial is an important invariant of an element or an ideal in a polynomial ring or power series ring of characteristic zero, with interesting connections to various algebraic and topological aspects of the singularities of the vanishing locus.
Work of Musta\c t\u a, later extended by Bitoun and the third author, provides an analogous Bernstein-Sato theory for regular rings of positive characteristic. 

In this paper, we extend this theory to singular ambient rings in positive characteristic. We establish finiteness and rationality results for Bernstein-Sato roots for large classes of singular rings, and relate these roots to other classes of numerical invariants defined via the Frobenius map. We also obtain a number of new results and simplified arguments in the regular case.
\end{abstract}
	
	\setcounter{tocdepth}{1}

	\tableofcontents

\newpage
	
%%%%%%%%%%%%%%%%%%%%%%%%%%%%%%%%%%%%%%%%%%%%%%%%%%%%%%%
\section{Introduction} \label{scn-Intro}
%%%%%%%%%%%%%%%%%%%%%%%%%%%%%%%%%%%%%%%%%%%%%%%%%%%%%%%
\subsection{Background}

Let $R := \CC[x_1, \ds, x_n]$ be a polynomial ring over $\CC$, $f \in R$ be a nonzero polynomial and $D_R$ be the ring of $\CC$-linear differential operators on $R$; that is, $D_R$ is the Weyl algebra over $\CC$. Bernstein \cite{BernsteinPoly} and Sato \cite{SatoPoly}, independently and in different contexts, showed that there is a nonzero polynomial $b_f(s)$ and an element $\xi(s) \in D_R[s]$ satisfying the functional equation
\begin{equation}
\label{eq-functional-eqn}
b_f(s) f^s = \xi(s) \cdot f^{s+1}.
\end{equation}
The monic polynomial $b_f(s)$ of least degree satisfying the equation above  for some operator $\xi(s) \in D_R[s]$ is called the Bernstein-Sato polynomial of $f$. This invariant measures the singularities of the zero-locus of $f$ in very subtle ways.  For example, work of Koll\'ar \cite{Kollar97} and Ein, Lazarsfeld, Smith, and Varolin \cite{ELSV04} gives that the log-canonical threshold $\lct(f)$ of $f$ is the smallest root of $b_f(-s)$, and that every jumping number in the interval $[0,1)$ is a root of $b_f(-s)$. Furthermore, Kashiwara \cite{KashiwaraVfil} and Malgrange \cite{MalgrangeVfil}  proved that the eigenvalues of the monodromy action on the cohomology of the Milnor fiber of $f$ are given by $\exp(2 \pi i \alpha)$ where $\alpha$ ranges through all the roots of the Bernstein-Sato polynomial of $f$.
Kashiwara \cite{KashiwaraRationality}  showed  that the roots of $b_f(s)$ are rational and negative which, combined with the previous result, shows that the monodromy action is quasi-unipotent (see also \cite{BernsteinRationalMalgrange}). 

An alternative characterization of $b_f(s)$ due to Malgrange exhibits $b_f(s)$ as the minimal polynomial for the action of an operator $s$ on a certain $D$-module $N_f$. 
Budur, \Mustata, and Saito constructed an analogue of $s$ and $N_f$ for the case of an arbitrary ideal $\fa \sq R$.
Namely, even though $N_\fa$ is usually not finitely generated, 
there exists  a minimal polynomial of $s$ on $N_\fa$;
the Bernstein-Sato polynomial of $\fa$ is defined as
this polynomial  \cite{BMS2006a}.
One has that   $N_\fa$ splits as a direct sum
$$N_\fa = \bigoplus_{\lambda \in \CC} (N_\fa)_\lambda,$$
where $(N_\fa)_\lambda$ is the $\lambda$-generalized eigenspace. 

One can recover the minimal polynomial $b_\fa(s)$ from this decomposition, since the roots are given by
\begin{equation} \label{eqn-roots-of-BSpoly}
\{\text{ Roots of } b_\fa(s) \ \} = \{\lambda \in \CC: (N_\fa)_\lambda \neq 0 \}
\end{equation}
and the multiplicity of a root $\lambda$ is given by
\begin{equation} \label{eqn-mult-of-BSpoly}
\mult(\lambda, b_\fa(s)) = \min \{k \geq 0 : (s - \lambda)^k (N_\fa)_\lambda = 0 \}.
\end{equation}

An extension of this rich theory has been proposed recently for the case where $R$ is a possibly singular $\CC$-algebra. Whenever $R$ is a direct summand of a polynomial ring over~$\CC$, \`Alvarez-Montaner, Huneke, and the second author \cite{AMHNB} showed that one can find $b_f(s)$ and $\xi(s)$ as in Equation~(\ref{eq-functional-eqn}) and thus define a Bernstein-Sato polynomial for elements $f$ of $R$. We remark that to carry out this construction one must take $D_R$ to be the ring of $\CC$-linear differential operators of $R$ in the sense of Grothendieck. This line of research has continued with explorations into connections with $V$-filtrations, multiplier ideals, and an extension of these constructions for the case of ideals \cite{AMHJNBTW19}. 

Some aspects of the theory have also been developed in positive characteristic. This began with the work of \Mustata, who began this exploration in the case where $R : = \KK[x_1, \ds, x_n]$ is a polynomial ring over a perfect field $\KK$ of characteristic $p > 0$ (or, more generally, a regular $F$-finite ring) and $\fa = (f)$ is a principal ideal. \Mustata's notion was later refined by Bitoun \cite{Bitoun2018} and extended to the case of arbitrary ideals by the third author \cite{QG19}.

The main goal in this paper is to explore the theory of Bernstein-Sato roots in positive characteristic after dropping the regularity assumption on $R$. One can therefore think of this paper as providing a characteristic $p$ counterpart to some of the work on differential operators in singular rings \cite{AMHNB,AMHJNBTW19}.
In order to explain our results, we need to elaborate on this notion of Bernstein-Sato invariants in positive characteristic.

Suppose $R := \KK[x_1, \ds, x_n]$ is a polynomial ring over a perfect field $\KK$ of characteristic $p > 0$ and let $\fa \sq R$ be an ideal. One starts, through mimicking the construction of Budur, \Mustata, and Saito, by defining a $D_R$-module $N_\fa$ associated to $\fa$. In the characteristic $p$ setting the action of the operator $s$ on $N_\fa$ naturally extends to an action of the algebra $C(\Zp, \F_p)$ of continuous functions from the $p$-adics $\Zp$ to $\F_p$. Note that we are explaining this construction using the algebra $C(\Zp, \F_p)$ of continuous functions from the $p$-adics $\Zp$ to $\F_p$, in the style of Bitoun \cite{Bitoun2018}, as opposed to the operators $s_{p^0}, s_{p^1}, \ds $ in the style of \Mustata \ and the third author. See Subsection \ref{SubSec-spi-and-CZpFp} for the equivalence between these two points of view.

Given a $p$-adic integer $\alpha \in \Zp$, we let $\fm_\alpha$ be the maximal ideal of $C(\Zp, \F_p)$ that consists of functions that vanish on $\alpha$, and we let $(N_\fa)_{\alpha} := \Ann_{\fm_\alpha} N_\fa$. A careful analysis of the module $N_\fa$ allows one to show that there is a decomposition
$$N_\fa = \bigoplus_{\alpha \in \Zp} (N_\fa)_{\alpha}$$
for which only finitely many $(N_\fa)_{\alpha}$ are nonzero \cite{Bitoun2018} \cite[Proposition  6.1]{QG19}. A posteriori, we conclude that the $(N_\fa)_{\alpha}$ can also be viewed as quotients of $N_\fa$, namely $(N_\fa)_{\alpha} \cong N_\fa / \fm_\alpha N_\fa$. 

In analogy with the situation over $\CC$, we want to obtain some invariant of $\fa$ from this decomposition. A Bernstein-Sato root of $\fa$ is thus defined to be a $p$-adic integer $\alpha \in \Zp$ such that $(N_\fa)_{(\alpha)}$ is nonzero (cf. (\ref{eqn-roots-of-BSpoly})); we think of these as characteristic $p$ analogues of the roots of the Bernstein-Sato polynomial. In this setting, however, we have $\fm_\alpha = \fm_\alpha^2$ for all $\alpha \in \Zp$ and it is therefore not clear how to associate a multiplicity to each Bernstein-Sato root (cf. (\ref{eqn-mult-of-BSpoly})), and thus there is no notion of Bernstein-Sato polynomial. 

The Bernstein-Sato roots of a nonzero ideal in a regular $F$-finite ring are known to be rational and negative \cite[Corollary 2.4.3]{Bitoun2018} \cite[Theorem  6.7]{QG19}, which gives a characteristic $p$ analogue of Kashiwara's theorem, and they are also known to be intimately linked to the $F$-jumping numbers of $\fa$ \cite[Theorem  2.4.1]{Bitoun2018} \cite[Theorem  6.11]{QG19}.

\subsection{Summary of results}

In this paper, we pose and work with an elementary definition of Bernstein-Sato root  in positive characterstic. Namely, we define a Bernstein-Sato root of an ideal $\fa$ to be a $p$-adic integer that occurs as the $p$-adic limit of a sequence of the form $(\nu_e)$ such that $\fa^{\nu_e}$ is not contained in $\sum_{\phi\in \Hom_{R^{p^e}}(R,R)} \phi(\fa^{\nu_e+1})$ (see Definitions~\ref{def-DiffJump} and~\ref{DefBSrootsPadic}). This elementary notion naturally extends the notion of Bernstein-Sato root in positive characteristic for regular rings described in the previous subsection. Namely, even after dropping the regularity assumption on $R$, we can still build the module $N_\fa$, equip it with a $C(\Zp, \F_p)$-action, and consider the modules $(N_\fa)_\alpha = N_\fa / \fm_\alpha N_\fa$ for all $\alpha \in \Zp$. We have that $\alpha \in \Zp$ is a Bernstein-Sato root of~$\fa$ if and only if the module $(N_\fa)_\alpha$ is nonzero (Theorem \ref{thm-BSRoot-equivalence}). In contrast with the regular case, when $R$ is singular an ideal $\fa \sq R$ may have infinitely many Bernstein-Sato roots (see Example~\ref{example-cone-ell-curve}) and the quotients $N_\fa \to (N_\fa)_\alpha$ a priori need not split.

We begin by isolating the necessary assumptions on $\fa$ so that these pathologies do not occur, and we encapsulate these in the notion of Bernstein-Sato admissible ideals (see Definition~\ref{def-BSadm}). It is implicit in the work of the third author that ideals of regular rings are always Bernstein-Sato admissible \cite[Theorem ~6.1]{QG19}. We show that these ideals also abound in other classes of rings.

\begin{theoremx} [{Theorem  \ref{ThmFFRTAdmissible}, Theorem  \ref{thm:directsummand}}]
	Let $R$ be a noetherian $F$-finite ring, and assume one of the following holds:
	\begin{enumerate}[(a)]
		\item The ring $R$ is graded with finite $F$-representation type.
		\item The ring $R$ is a direct summand of a regular ring.
	\end{enumerate}
	Then every ideal $\fa \sq R$ is Bernstein-Sato admissible.
\end{theoremx}

As mentioned, whenever an ideal $\fa$ is Bernstein-Sato admissible we show that the module $N_\fa$ behaves as in the regular case. 
		
\begin{theoremx}
	Let $R$ be a noetherian $F$-finite ring and $\fa \sq R$ be a Bernstein-Sato admissible ideal. Then:
	\begin{enumerate}[(i)]
		\item \textup{(Theorem  \ref{thm-finiteBSRoots})} The ideal $\fa$ has only finitely many Bernstein-Sato roots. 
		\item \textup{(Corollary \ref{cor-BSadm-Na-splits})} The module $N_\fa$ splits as a direct sum $N_\fa = \bigoplus_{\alpha \in \Zp} (N_\fa)_\alpha$.
	\end{enumerate}
	If in addition $R$ is $F$-split, then
	\begin{enumerate}[(i)]
	\setcounter{enumi}{2}
	
	\item \textup{( Theorems \ref{thm:ratl} \&  \ref{thm:rational}}) All of the Bernstein-Sato roots of $\fa$ are rational and lie in the interval\,\footnote{\label{footnote-1}In fact, the lower bound can be improved by using the analytic spread.}  $[-r, 0]$, where $r$ is the number of generators of $\fa$.
		\end{enumerate}
\end{theoremx}
We remark that, for nonprincipal ideals, the lower bound is new even in the case where $R$ is regular. Combining this with a result of the third author on the behavior of Bernstein-Sato roots of monomial ideals under mod-$p$ reduction \cite[Theorem 3.1]{QG19b}, we are able to give the following characteristic zero result.

\begin{corollaryx}[Corollary \ref{cor-monomial-char0}]
	Let $R = \CC[x_1, \ds, x_n]$ be a polynomial ring over $\CC$ and $\fa \sq R$ be a monomial ideal generated by $r$ elements. If $\lambda$ is a root of the Bernstein-Sato polynomial of $\fa$ then\textsuperscript{\textup{\ref{footnote-1}}} $-r \leq \lambda$.
\end{corollaryx}

To study the action of differential operators on ideals in the ring, we introduce a family of numerical invariants called differential thresholds. The collection of differential thresholds of an ideal contains several of its invariants defined via Frobenius, including all of its jumping numbers, $F$-thresholds, and Cartier thresholds (see Subsection \ref{SubSecDT}). 
This unified approach allow us to obtain properties that were not known in certain cases.
		We show that if $\fa\subseteq R$ is a Bernstein-Sato admissible ideal, then the set of differential thresholds for $\fa$ is a  discrete set of rational numbers (see Theorems \ref{ThmDicretness} and \ref{ThmDTrational}).				
As a consequence, we obtain that the $F$-thresholds of rings with graded  finite $F$-representation type are rational numbers (see Corollary \ref{CorFFRTRational}). This extends previous results obtained for certain ideals in Stanley-Reisner rings \cite[Theorems A \& B]{BC}. We also exhibit a close relation between differential thresholds and Bernstein-Sato roots.

\begin{theoremx}[Theorem  \ref{thm:threshsareroots}]
	 Let $R$ be $F$-split. Let $\fa$ be an ideal with $r$ generators. There is an equality of cosets in $\ZZ_{(p)}/\ZZ$:
	\[ \{ \alpha + \ZZ \ | \ \alpha\in  \ZZ_{(p)}\,\text{Bernstein-Sato root of }\fa\} = \{ -\lambda + \ZZ \ | \ \lambda\in\ZZ_{(p)}\,\text{differential threshold of }\fa \}.\]
\end{theoremx}

We are able to define a module $R_f \bfa$ for any $p$-adic number that is the positive characteristic analogue of the modules $R_f \bfa$ with $\alpha\in\QQ$ in characteristic zero \cite{UliSurvey}. If $R$ is regular and  $\alpha\in \widehat{\ZZ}_{(p)}\cap \QQ_{<0}$, then $R_f \bfa=M_{-\alpha}$ for regular rings, where $M_{-\alpha}$ is the $F$-module introduced in earlier work of Blickle, \Mustata, and Smith to study jumping numbers of principal ideals \cite{BMSm-hyp} (see also \cite{NBP}). 
In Proposition~\ref{prop:BSrootsandalphajumps}, we show that in contrast to the situation in characteristic zero \cite{Saito}, $\alpha$ is a Bernstein-Sato root if and only if $\bfa \notin D_R\cdot f \bfa$.  We also provide a characterization of the simplicity of $R_f \bfa$ in terms of Bernstein-Sato roots and differential thresholds. 

\begin{theoremx}[Theorem  \ref{ThmSimpleRfaBSRDT}]
	Suppose that $R$ is a strongly $F$-regular domain. Let $f \in R$ be  a Bernstein-Sato admissible nonzerodivisor and $\alpha \in \ZZ_{(p)}\cap[-1,0)$. Then the following are equivalent:
	\begin{enumerate}[(a)]
		\item The module $R_f \bfa$ is not simple over $D_R$. 
		\item We have that $\alpha$ is a Bernstein-Sato root of $f$. 
		\item We have that $- \alpha$ is a differential threshold of $f$. 
	\end{enumerate}
\end{theoremx}

Moreover, the finite generation or finite length $R_f \bfa$ as a $D_R$-module  provide information about the distribution of the Bernstein-Sato roots and differential thresholds (see Theorems~\ref{ThmDiscretnessBelow} \& \ref{thm:descthresh}). 

Since in the regular case $M_\alpha$ is an $F$-finite $F$-module, it has finite length as a $D_R$-module \cite{BMSm-hyp} (see also \cite{LyuFMod}). In Theorem \ref{ThmRfaBAlg}, we show that $R_f\bfa$
is a holonomic $D_R$-module for every Bernstein algebra, and so it has finite length as a $D_R$-module.
This is a recently defined class of singular algebras whose $D_R$-modules satisfy the Bernstein inequality
\cite{BernsteinAlgebras}.  We stress that our results regarding $R_f \bfa$ do not use the theory of $F$-modules, which is not available for singular rings.

We point out that we prove some results that are new even in the case where $R$ is regular (e.g.,~Corollary \ref{cor:ell}). In addition, we provide new proofs of previously known theorems (e.g.,~Theorem \ref{thm:rational}).

Axel St\"{a}bler has pointed out to us that Bernstein-Sato polynomials for Cartier modules \cite{Sta1,Sta2} can also be used to give a notion of Bernstein-Sato polynomial for certain singular algebras. Namely, if $R=S/I$ is strongly $F$-regular and $\QQ$-Gorenstein, and $S$ is regular, one may consider $R$ as a Cartier module over the ring $S$, and apply the theory of \textit{ibid.}~to obtain Bernstein-Sato polynomials.
 In contrast, our approach uses the operators on the singular ring itself and is developed for rings that are not necessary strongly $F$-regular. 
In particular, in our approach  an ideal in a strongly $F$-regular ring may have  more roots in the interval $[-1,0]$ than jumping numbers (see Example~\ref{ExToric}).

\subsection{Notation}  We fix a prime number $p$, and $\Zp$ denotes the ring of $p$-adic integers. Unless otherwise stated, all rings have characteristic $p$ and are $F$-finite, meaning that the Frobenius endomorphism is module-finite.  

Given an ideal $\fa$ in a ring $R$, we set $\fa^0 = R$ by convention (even when $\fa = (0)$).

We use multi-index notation: given a tuple of integers $\ul a = (a_1, \ds, a_n) \in \Z^n$ and a tuple of elements $\ul g = (g_1, \ds, g_n) \in S^n$ in a commutative ring $S$, we denote $\ul g^{\ul a} := g_1^{a_1} \cds g_n^{a_n}$. The symbol $\bone$  denotes the tuple $\bone := (1, 1, \ds, 1)$. Recall we have a multi-index binomial theorem: given a commutative ring $S$, tuples $\ul x, \ul y  \in S^n$ and a multi-exponent $\ul a \in (\Z_{\geq 0})^n$ we have
$$(\ul x + \ul y)^{\ul a} = \sum_{0 \leq b_i < a_i} {\ul a \choose \ul b} \ul x^{\ul b} \ul y^{\ul a - \ul b},$$
where ${\ul a \choose \ul b} = \prod_{i = 1}^n {a_i \choose b_i}.$

\subsection*{Acknowledgments} We would like to thank Josep \`Alvarez Montaner, W\'agner Badilla C\'espedes, and Axel St\"abler for comments on an earlier draft of this paper.  We thank the anonymous referee for helpful comments.

	%%%%%%%%%%%%%%%%%%%%%%%%%%%%%%%%%%%%%%%%%%%%%%%%%%%%%%%
	\section{Preliminaries} \
	%%%%%%%%%%%%%%%%%%%%%%%%%%%%%%%%%%%%%%%%%%%%%%%%%%%%%%%
	
	\subsection{Base $p$ and $p$-adic expansions} \label{subscn-padic-basep} 
	
	Fix a prime number $p$.  Let $\Zp$ denote the completion with respect to $(p)$ of $\ZZ_{(p)}$, i.e., the ring of $p$-adic integers.
	Given $\alpha\in \Zp$, there exists a unique sequence of integers $(\alpha_e)_{e\in \NN}$ such that:
	\begin{enumerate}
		\item $0 \leq \alpha_e \leq p-1$, and
		\item $\alpha = \sum_{e \geq 0} p^e \alpha_e$ as a series in $\Zp$.
	\end{enumerate}
	We call $\alpha_e$ the $e$-th $p$-adic digit of $\alpha$. We reserve the notation $\alpha_e$ for this notion. 	We define the $e$-th $p$-adic truncation of $\alpha$ to be the unique integer $n$ with $0\leq n < p^e$ such that $\alpha \equiv n \mod p^e$; equivalently, \[\aae := \alpha_0 + p \alpha_1 + p^2 \alpha_2 + \cdots + p^{e-1} \alpha_{e-1}.  \]

Recall that a $p$-adic integer $\alpha$ is rational if and only if $\alpha\in \ZZ_{(p)}$; this is equivalent to $\alpha$ admitting an eventually periodic sequence of $p$-adic digits. 

A $p$-adic number $\alpha$ has a purely periodic sequence of $p$-adic digits if and only if $\alpha \in \ZZ_{(p)} \cap [-1,0]$. In particular, the sequence of $p$-adic digits of $\alpha$ is periodic of period $e$ if and only if $(1-p^e)\alpha$ is an integer between $0$ and $p^e-1$, and in this case we have $(1-p^e)\alpha =\aae$; our convention is that the period is not necessarily minimal. In particular, if $\alpha \in \ZZ_{(p)} \cap [-1,0]$ with $(1-p^e)\alpha\in \NN$, then $(1-p^{ae})\alpha = \aaae$ for all $a$. Similarly, the sequence of $p$-adic digits of $\alpha$ is eventually periodic of period $e$ if and only if $(1-p^e)\alpha\in \ZZ$. 

We can also extract the $p$-adic truncations of an arbitrary element $\alpha\in \Z_{(p)}$. For our purposes, it suffices to determine for any such $\alpha$ an infinite sequence of $p$-adic truncations. 

\begin{lemma}\label{lem:expn}
	Let $\alpha \in \ZZ_{(p)}$, and let $e\in \ZZ_{>0}$ such that $(p^e-1)\alpha\in \ZZ$. Then, for all $a\gg0$, we have
	\[ \aaae =\begin{cases} (1-p^{ae}) (\alpha- \lceil \alpha \rceil) + \lceil \alpha \rceil &\text{if }\alpha\notin\ZZ_{<0} \\
	p^{ae} + \alpha &\text{if }\alpha\in\ZZ_{<0}.
	\end{cases}\]
\end{lemma}
\begin{proof} 
	The claim is clear when $\alpha \in \ZZ$, so take $\alpha \notin \Z$. We first observe that $(1-p^{ae}) (\alpha- \lceil \alpha \rceil) + \lceil \alpha \rceil \equiv \alpha - \lceil \alpha \rceil + \lceil \alpha \rceil \equiv \alpha$ modulo $p^{ae}$. 
		
	Since $\alpha- \lceil \alpha \rceil > -1$, for $a\gg0$ we have $-(\alpha - \lceil \alpha \rceil) + \frac{\lceil \alpha \rceil}{p^{ae}-1} \leq 1$, so $(1-p^{ae})(\alpha- \lceil \alpha \rceil)  + \lceil \alpha \rceil \leq p^{ae}-1$.
	
	We have $\alpha- \lceil \alpha \rceil < 0$, so for $a\gg0$, we have $-\lceil \alpha \rceil \leq (1-p^{ae})(\alpha - \lceil \alpha \rceil)$ and thus  $(1-p^{ae})(\alpha- \lceil \alpha \rceil)  + \lceil \alpha \rceil \geq 0$. 
		\end{proof}

	Given $\lambda \in (0,1]$, there exists a unique sequence of integers $(\digit{\lambda}{e})_{e \geq 1}$
	satisfying the following conditions.
	\begin{enumerate}
		\item $0 \leq \digit{\lambda}{e} \leq p-1$, 
		\item $\lambda = \sum_{e \geq 1} \frac{\digit{\lambda}{e}}{p^e}$, and 
		\item The sequence $(\digit{\lambda}{e})_{e\geq 1}$ is not eventually zero.
	\end{enumerate}
	We call $\digit{\lambda}{e}$ the $e$-th digit of $\lambda$ base $p$, and we call the expression $\lambda = \sum_{e \geq 1} \frac{\digit{\lambda}{e}}{p^e}$ the non-terminating base $p$ expansion of $\lambda$.  
	By convention, we set $\digit{\lambda}{0} = 0$. 
	We adopt notation analogous to standard decimal notation, writing 
	\[ \lambda = . \ \digit{\lambda}{1} : \digit{\lambda}{2} : \cdots : \digit{\lambda}{e}: \cdots \ \ \ (\base p),\] 
	where colons distinguish between consecutive digits.

	For $e \geq 1$, the $e$-th truncation of $\lambda$ in base $p$ is defined as $\tr{\lambda}{e}  :=  \frac{\digit{\lambda}{1}}{p} + \cdots + \frac{\digit{\lambda}{e}}{p^e}$. Note that $p^e \tr{\lambda}{e}$ is the unique integer $n$ with the property that $\lambda \in (n/p^e, (n+1)/p^e]$, and thus $\tr{\lambda}{e} = \frac{\up{p^e \lambda } - 1}{p^e}$; in particular, for all $e \geq 1$ we have $\tr{\lambda}{e}<\lambda$.
	We define $\tr{\lambda}{\infty} := \lambda$, and make the convention $\tr{\lambda}{0} = 0$. 
	
	A number $\lambda\in (0,1]$ has a purely periodic sequence of base $p$ digits if and only if $\lambda \in \ZZ_{(p)} \cap (0,1]$. In particular, the sequence of base $p$ digits of $\lambda$ is periodic of period $e$ if and only if $(p^e-1)\lambda$ is an integer between $0$ and $p^e-1$, and in this case we have $(p^e-1)\lambda =p^e \tr{\lambda}{e}$; our convention is that period is not necessarily minimal. In particular, if $\lambda \in \ZZ_{(p)} \cap (0,1]$ with $(p^e-1)\lambda\in \NN$, then $(p^{ae}-1)\lambda = \tr{\lambda}{ae}$ for all $a$.

%%%%%%%%%%%%%%%%%%%%%%%%%%%%%%%%%%%%%%%%%%%%%%5	
	\subsection{Methods in prime characteristic} \label{Subsec-PrimeChar} \ 
%%%%%%%%%%%%%%%%%%%%%%%%%%%%%%%%%%%%%%%%%%%%%%5		

\begin{definition}\label{DefBasicCharP}
Suppose that $R$ is a ring of prime characteristic $p$.
\begin{enumerate}[(i)]
\item Given an integer $e \geq 0$, we let $F^e_* R$ be the abelian group $R$ endowed with the $R$-module structure coming from restriction of scalars via the $e$-th iterated Frobenius $F^e: R \to R$. Given an element $f \in R$, we sometimes write it as $F^e_* f$ to emphasize that we view it as an element of $F^e_* R$. With this notation, the $R$-module structure of $F^e_* R$ is given by $g F^e_* f = F^e_* (g^{p^e} f)$ for all $f, g \in R$. 
\item If $R$ is a $\NN$-graded ring, $F^e_* R$  is a  $\frac{1}{p^e}\NN$-graded module over $R$, where $\deg( F^e_*r)=\frac{1}{p^e} \deg(r).$
\item We say that $R$ is $F$-finite if $ F^e_* R$ is a finitely generated $R$-module for some $e\geq 1$ (equivalently, for every $e\geq 1$).
\end{enumerate}
\end{definition}

A perfect field is $F$-finite. If $R$ is $F$-finite then the polynomial ring $R[x]$, the power series ring $R [[x]]$, all quotients of $R$ and all localizations of $R$ are also $F$-finite. This means that most rings that arise when doing algebraic geometry over a perfect field are $F$-finite. 

\begin{definition}
Suppose that $R$ is a ring of prime characteristic $p$. 
	\begin{enumerate}[(i)]
\item We say that $R$ is $F$-split if the Frobenius map splits or, equivalently, if the $R$-module 
$ F_*R$ has a nonzero free summand.
\item We say that $R$ is $F$-pure if the Frobenius map is pure. Specifically,  the map
$M\to M\otimes_R  F_*R$ is injective for every $R$-module $M$.
\item Assume that $R$ is a domain. We say that $R$ is strongly $F$-regular if for every nonzero $r\in R$
there exists $e\in \Z_{\geq 0}$ such that the $R$-module homomorphism $\varphi:R\to F^e_* R$
given by $1\mapsto F^e_* r$ splits.
\end{enumerate}
\end{definition}

\begin{remark}\label{RemFpureFsplit} 
Suppose that $R$ is a ring of prime characteristic $p$. 
	\begin{enumerate}[(i)]
\item  If $R$ is an $F$-finite ring  \cite[Corollary $5.3$]{HochsterRoberts}  or a complete local ring \cite[Lemma~1.2]{Fedder}, $R$ is $F$-pure if and only $R$ is $F$-split.
\item In the definitions of $F$-finite, $F$-pure, and $F$-split, the conditions on $F_* R$ can be replaced by 
$F^e_* R$ for some $e\geq 1$, or by $F^e_* R$ for every $e\geq 1$.
\end{enumerate}
\end{remark}

\begin{definition} \label{FrobeniusCartier}
Let $R$ be an  $F$-finite ring and $e \geq 0$ be an integer.
 An additive map $\phi:R\to R$ is a $p^{-e}$-linear map if
$\phi(r^{p^e} f)=r\phi(f)$ for all $r, f \in R$. We denote by $\cC^e_R$ the set of all
$p^{-e}$-linear maps. Then, we have  $\cC^e_R =
\Hom_R(F^e_* R,R)$. Given an ideal $\fa \sq R$ we denote by $\cC^e_R \cdot \fa$ the ideal $\cC^e_R \cdot \fa = (\phi(f) \ | \ \phi \in \cC^e_R, f \in \fa)$. 
\end{definition}

Test ideals were introduced by Hochster and Huneke, and they are  a fundamental tool in the theory of tight closure   \cite{HH90,HoHu2,HoHu2}.
Hara and Yoshida \cite{HY03} extended the notion of test ideals, $\tau_R(\fa^\lambda)$,  to pairs $(R,\fa^\lambda)$, where $ \fa \subseteq R$ is an
ideal and $\lambda \in \mathbb{R}$.
One can approach the theory of test ideals using  Cartier operators   \cite{BMSm2008,BMSm-hyp,TestQGor,BB-CartierMod,Bli13}.
We  now give the definition in terms of Cartier operators for strongly $F$-regular rings  \cite{TTFFRT}.
 
\begin{definition}
Let  $R$ be  a  $F$-finite  strongly $F$-regular ring. Let $\fa\subseteq R$
be an ideal, and $\lambda\in \RR_{>0}.$ The test
ideal of  the pair $(\fa,\lambda)$ is
defined by
$$
\tau_R(\fa^{\lambda})=\bigcup_{e\in\NN}\cC^e_R \cdot \fa^{\lceil p^e \lambda \rceil}.
$$
\end{definition}

The notion of test ideal discussed here is sometimes called the big test ideal.

We note that the chain of ideals $\{ \cC^e_R \cdot \fa^{\lceil p^e \lambda
\rceil}\}$   is increasing, and so,
$\tau_R(\fa^{\lambda})=\cC^e_R \cdot \fa^{\lceil p^e \lambda \rceil}$ for
$e\gg 0$, because $R$ is noetherian.

We now recall well-known properties of test ideals. 
We refer to the work done specifically for strongly $F$-regular rings \cite{TTFFRT}.
For a more  general approach,  we refer to Blickle's work on this subject \cite{Bli13}.

\begin{proposition}[{\cite[Lemma 4.5]{TTFFRT}}]\label{PropBasics}
Let $R$ be an strongly $F$-regular $F$-finite ring, $\fa,\fb \subseteq R$ ideals, and $\lambda,\lambda'\in\RR_{>0}.$
Then,
\begin{enumerate}
\item If $\fa\subseteq \fb,$ then $\tau(\fa^\lambda)\subseteq \tau_R(\fb^{\lambda})$.
\item If $\lambda<\lambda',$ then $\tau(\fa^{\lambda'})\subseteq \tau(\fa^{\lambda})$.
\item There exists $\epsilon>0$ such that $\tau_R(\fa^{\lambda})= \tau_R(\fa^{\lambda'})$
if $\lambda'\in [\lambda,\lambda+\epsilon)$.
\end{enumerate}
\end{proposition}

Every ideal $\fa\subseteq R$ is  associated to a family of test ideals $\tau(\fa^{\lambda})$
parameterized by real numbers $\lambda \in \mathbb{R}_{>0}$ which forms a decreasing nested chain of ideals as $\lambda$ increases.

\begin{definition}
Let $R$ be an $F$-finite strongly $F$-regular ring and let $\fa\subseteq R$ be an ideal. A real number $\lambda\geq 0$ is an $F$-jumping number
of $\fa$ if
$$
\tau_R(\fa^{\lambda})\neq \tau_R(\fa^{\lambda-\epsilon})
$$
for every $\epsilon>0.$
\end{definition}

%%%%%%%%%%%%%%%%%%%%%%%%%%%%%%%%%%%%%%%%%%%%%%5	
	\subsection{Basics of differential operators} \label{Subsec-BAsicDiffOps} \ 
%%%%%%%%%%%%%%%%%%%%%%%%%%%%%%%%%%%%%%%%%%%%%%5

In this section we briefly recall the basic notions  on the theory of rings
of differential operators  introduced  by Grothendieck \cite[\S 16.8]{EGAIV}.

Let $R$ be an $\KK$-algebra, where $\KK$ is a field.
 The ring of $\KK$-linear  differential
operators of $R$ is the subring $D_{R|\KK}\subseteq \Hom_{\ZZ}(R,R)$
whose elements are characterized inductively as follows.  The differential
operators of order zero are $D^0_{R|\KK}= \Hom_R(R,R)$.
A linear map
$\delta\in \Hom_\KK(R,R)$ is an operator of order less than or equal
to $\ell$ if $\delta r-r\delta$ is an operator of order less
than or equal to $\ell-1.$ We write $D^{\ell}_{R|\KK}$ for the collection of differential operators of order at most $\ell$. We define $D_{R|\KK}=\bigcup_{\ell\in\NN}D^{\ell}_{R|\KK}$, which is a ring with composition as the multiplication.

%Let $\mu: R\otimes_\KK R\to R$ be the multiplication map and $\Delta_{R|\KK}$ its kernel.
%The module of principal parts of $R$ is defined as 
%$$\mathcal{P}^n_{R | \KK}=\frac{R\otimes_\KK R }{\Delta^{n+1}_{R|\KK}}.$$
%The differential operators of order $n$ are represented by the module of principal parts. Specifically, there is a natural $R$-module isomorphism  $D^n_{R|\KK} \cong \Hom_R(\mathcal{P}^{n}_{R|\KK},R)$. 
%We recall that $\End_\KK(R, R)$ has an structure of  $R \otimes_\KK R$-module, where $(r_1\otimes r_2 \cdot \delta)  (f)=r_1\delta (r_2 f)$. Then, $D^n_{R|\KK}=\Ann_{\End_\KK(R, R)}     \Delta^{n+1}_{R|\KK}.$

\begin{example}
Let $R$ be either the polynomial ring $\KK[x_1,\dots,x_n] $ or the formal power
series ring $\KK[[x_1,\dots,x_n]]$ with coefficients in a ring $\KK$. The
ring of $\KK$-linear differential operators is:
$$D_{R|\KK}= R \left \langle \hskip 2mm \frac{1}{t!}\frac{d^t}{dx_i^t} \hskip 2mm | \hskip 2mm i=1,\dots,n; \hskip 2mm t\in \NN \right\rangle ,$$
that is, the free $R$-module generated by the differential operators $\frac{1}{t!}\frac{d^t}{dx_i^t}$.
We recall that $\frac{1}{t!}\frac{d^t}{dx_i^t}$ acts on the monomials of $R$ by
$$
\frac{1}{t!}\frac{d^t}{dx_i^t} \cdot x_1^{\alpha_1}\cdots x_n^{\alpha_n}=
\begin{cases}
\binom{\alpha_i}{t} x_1^{\alpha_1}\cdots x^{\alpha_i-t}_i\cdots  x_n^{\alpha_n} & \alpha_i< t\\
0 &  \alpha_i\geq  t.
\end{cases}
$$
Furthermore, if $\KK$ is a field of characteristic zero, we have
 $$D_{R|\KK}= R \left \langle  \frac{d}{dx_1},\dots, \frac{d}{dx_n} \right\rangle.$$
\end{example}

\begin{example}
Let $S= \KK[x_1,\dots,x_d]$ be a polynomial ring over a field $\KK$ of
characteristic zero and, given an ideal $I\subseteq S$, set $R=S/I$.
Then,  the ring of $\KK$-linear
differential operators of $R$ is characterized in terms of the differential operators in $S$  \cite[Theorem 5.13]{MCR}. Specifically,  $$D_{R|\KK}=
\frac{\{ \delta \in D_{S|\KK} \hskip 2mm | \hskip 2mm \delta (I)
\subseteq I \}}{ID_{S|\KK}}$$
\end{example}

Let $R$ be an $F$-finite ring (not necessarily regular). We denote by $D_R$ the ring of $\F_p$-linear differential operators on $R$. In this context, we have that
$$D_R = \bigcup_{e = 0}^\infty D^{(e)}_R$$
where $D^{(e)}_R = \End_{R^{p^e}}(R)$ and, if $\KK$ is a perfect field contained in $R$, then the ring $D_{R| { \KK}}$ of {$\KK$}-linear differential operators on $R$ agrees with $D_R$ \cite{Yek,SVdB97}. Given an integer $e \geq 0$, we call $D^{(e)}_R$ the ring of differential operators of level $e$. We note that for any $F$-finite ring, the formation of $D^{(e)}_R$ commutes with localization. Additionally, if $R$ is $F$-finite and local, then the formation of $D^{(e)}_R$ commutes with completion. Both of these facts follow from description of $D_R$ in terms of the level filtration above and the behavior of Hom under flat base change.

We will also use a result of Smith that states that, whenever $R$ is an $F$-split domain, $R$ is simple as a $D_R$-module if and only if $R$ is strongly $F$-regular \cite{Smi95}.

%%%%%%%%%%%%%%%%%%%%%%%%%%%%%%%%%%%%%%%%%%%%%%5	
	\subsection{Differential operators and $V$-filtrations} \label{Subsec-DiffOps_VFilt} \ 
%%%%%%%%%%%%%%%%%%%%%%%%%%%%%%%%%%%%%%%%%%%%%%5

	We introduce  a few facts about the relationship between $D_R$ and $D_{R[{\ul{t}}]}$, where $R[\ul{t}] = R[t_1, \ds, t_r]$ is a polynomial ring over $R$. However, these facts are only used in Section~\ref{scn-BSroots-via-V}, where we show that the definition of Bernstein-Sato roots we give in Section~\ref{Sec-BSR} agrees with the definition that one arrives to by considering the $D$-module constructions of Bitoun, \Mustata,  and the third author in the regular case \cite{Bitoun2018, Mustata2009,QG19}. For this reason, we encourage the reader to skip the remaining of this section until they want to read Section~\ref{scn-BSroots-via-V}.
	
	If $\xi \in D_R$ is a differential operator on $R$, then $\xi$ acts on $R[\ul{t}]$ by the formula $\xi \cdot (g \ul{t}^{\ul{k}}) = (\xi \cdot g) \ul{t}^{\ul{k}}$ for $g \in R$ and $\ul{k} \in (\Z_{\geq 0})^r$, and one checks that this exhibits $\xi$ as a differential operator on $R[\ul{t}]$. Similarly, if $\delta \in D_{\F_p[\ul{t}]}$ is a differential operator on $\F_p[\ul{t}]$ then we can think of $\delta$ as a differential operator on $R[\ul{t}]$ via the action $\delta \cdot g \ul{t}^{\ul{k}} = g (\delta \cdot \ul{t}^{\ul{k}})$. These observations yield a ring homomorphism $D_R \otimes_{\F_p} D_{\F_p[t]} \to D_{R[\ul{t}]}$; we observe that it respects the level filtration and it therefore induces maps $D^{(e)}_R \otimes_{\F_p} D^{(e)}_{\F_p[\ul{t}]} \to D^{(e)}_{R[\ul{t}]}$. We want to show that these maps are isomorphisms.

\begin{lemma} \label{lemma-End-and-tensor}
	Let $S$ be a commutative ring, $G$ be a finite free $S$-module, and $M$ be an arbitrary $S$-module. Then the natural map $\End_S(M) \otimes_S \End_S(G) \to \End_S(M \otimes_S G)$ that sends $[\phi \otimes \psi \mapsto [u \otimes v \mapsto \phi(u) \otimes \psi(v)]]$ is an isomorphism. 
\end{lemma}
\begin{proof}
	We have the following natural isomorphisms
	\begin{align*}
	\Hom_S(M, M) \otimes_S \Hom_S(G, G) & \simto \Hom_S(G, G \otimes_S \Hom_S(M, M))  \\
	& \simto \Hom_S(G, \Hom_S(M, M \otimes_S G)) \\
	& \simto \Hom_S(M \otimes_S G, M \otimes_S G),
	\end{align*}
	where the last isomorphism comes from the tensor-hom adjunction. We then check that the composition of these isomorphisms is the morphism given in the statement.
\end{proof}

\begin{lemma} \label{lemma-D_R-addvar}
	Let $R$ be an $F$-finite ring. The morphisms $D^{(e)}_R \otimes_{\F_p} D^{(e)}_{\F_p[\ul{t}]} \to D^{(e)}_{R[\ul{t}]}$ and \\ $D_R \otimes_{\F_p} D_{\F_p[\ul{t}]} \to D_{R[\ul{t}]}$ previously defined are isomorphisms. 
\end{lemma}

\begin{proof}
	Fix an $e \geq 0$. We then have
	\begin{align*}
		\End_{R^{p^e}} (R) \otimes_{\F_p} \End_{\F_p[\ul t]^{p^e}} (\F_p [ \ul t]) & \cong \End_{R^{p^e}} (R) \otimes_{R^{p^e}} R^{p^e} \otimes_{\F_p} \F_p[\ul t]^{p^e} \otimes_{\F_p[\ul t]^{p^e}} \End_{\F_p[\ul t]^{p^e}} (\F_p[\ul t]).
	\end{align*}
	Now note that there is an algebra isomorphism $R^{p^e} \otimes_{\F_p} \F_p[\ul t]^{p^e} \cong R[\ul t]^{p^e}$, and recall that Hom commutes with flat base change whenever the source module is finitely presented. We thus we have 
	\begin{align*}
		\End_{R^{p^e}} (R) & \otimes_{R^{p^e}} R^{p^e} \otimes_{\F_p} \F_p[\ul t]^{p^e} \otimes_{\F_p[\ul t]^{p^e}} \End_{\F_p[\ul t]^{p^e}} (\F_p[\ul t]) \\
			& \cong \big( \End_{R^{p^e}} (R) \otimes_{R^{p^e}} R[\ul t]^{p^e} \big) \otimes_{R[\ul t]^{p^e}} \big( R[\ul t]^{p^e} \otimes_{\F_p[\ul t]^{p^e}} \End_{\F_p[\ul t]^{p^e}} (\F_p[\ul t]) \big) \\
			& \cong \End_{R[\ul t]^{p^e}} (R \otimes_{R^{p^e}} R[\ul t]^{p^e}) \otimes_{R[\ul t]^{p^e}} \End_{R[\ul t]^{p^e}} (R[\ul t]^{p^e} \otimes_{\F_p[\ul t]^{p^e}} \F_p[\ul t]) \\
			& \cong \End_{R[\ul t]^{p^e}} (R \otimes_{R^{p^e}} R[\ul t]^{p^e} \otimes_{\F_p[\ul t]^{p^e}} \F_p[\ul t]) \tag*{(Lemma \ref{lemma-End-and-tensor})} \\
			& \cong \End_{R[\ul t]^{p^e}} (R \otimes_{R^{p^e}} R^{p^e} \otimes_{\F_p} \F_p[\ul t]^{p^e} \otimes_{\F_p[\ul t]^{p^e}} \F_p[\ul t]) \\
			& \cong \End_{R[\ul t]^{p^e}} (R[\ul t]),
	\end{align*}
	and one checks that this composition agrees with the morphism in the statement. The statement for $D_{R[\ul{t}]}$ follows.
\end{proof}

It follows that we identify $D^{(e)}_R$ and $D^{(e)}_{\F_p[\ul{t}]}$ with subrings of $D^{(e)}_{R[\ul{t}]}$; note that they commute with each other. 

Let $I$ denote the ideal $I = (t_1, \ds, t_r) \sq R[\ul{t}]$. For every $e \geq 0$ and $i \in \Z$ we denote
$$V^i D^{(e)}_{R[\ul{t}]} := \{\xi \in D^{(e)}_{R[\ul{t}]} : \xi \cdot I^j \sq I^{j + i} \text{ for all } j \in \Z\},$$
where we adopt the convention that $I^n = R[\ul{t}]$ for all $n \leq 0$. We define $V^i D_{R[\ul{t}]}$ similarly. 

We give $R[\ul{t}]$ the grading that places $R$ in degree zero and gives each variable $t_i$ degree one. Because $R$ is $F$-finite, so is $R[\ul{t}]$, and therefore $D^{(e)}_{R[\ul{t}]} = \End_R(F^e_* R[\ul{t}], F^e_* R[\ul{t}])$ acquires a $\Z$-grading, which also induces a $\Z$-grading on $D_{R[\ut]}$. Given $d \in \Z$ we denote by $R[\ut]_d$ (resp. $(D^{(e)}_{R[\ut]})_d$, $(D_{R[\ut]})_d$) the set of homogeneous elements of $R[\ut]$ (resp. $D^{(e)}_{R[\ut]}$, $D_{R[\ut]}$) of degree $d$. We also denote $R[\ul t]_{\geq d} = \bigoplus_{i = d}^\infty R[\ul t]_i$, $(D^{(e)}_{R[\ul t]})_{\geq d} := \bigoplus_{i = d}^\infty (D^{(e)}_{R[\ul t]})_i$ and $(D_{R[\ul t]})_{\geq d} = \bigoplus_{i = d}^\infty (D_{R[\ul t]})_{i}$. In particular, we have $I^n = R[\ul{t}]_{\geq n}$ for all $n \in \Z$. Note that the previous isomorphisms respect the gradings, and they therefore induce isomorphisms $D^{(e)}_R \otimes_{\F_p} (D^{(e)}_{\F_p[\ut]})_d \simto (D^{(e)}_{R[\ut]})_d$ and $D_R \otimes_{\F_p} (D_{\F_p[\ut]})_d \simto (D_{R[\ut]})_d$.

\begin{lemma} \label{lemma-Vfilt-degree}
	Let $e$ and $i$ be integers with $e \geq 0$. Then:
	\begin{enumerate}
		\item We have $V^i D^{(e)}_{R[\ut]} = (D^{(e)}_{R[\ut]})_{\geq i}$ and $V^i D_{R[\ut]} = (D_{R[\ut]})_{\geq i}$.
		\item If $i \geq 0$, then we also have $V^i D^{(e)}_{R[\ut]} = (D^{(e)}_{R[\ut]})_0 I^i$ and $V^i D_{R[\ut]} = (D_{R[\ut]})_0 I^i$. 
	\end{enumerate}
\end{lemma}
\begin{proof}
	It is enough to prove the claims for $D^{(e)}_{R[\ul{t}]}$, and we begin with (i). The inclusion $V^i D^{(e)}_{R[\ut]} \supseteq (D^{(e)}_{R[\ut]})_{\geq i}$ follows from the fact that $I^j = R[\ul{t}]_{\geq j}$ for all $j \in \Z$. For the other inclusion, suppose that $\xi \in V^i D^{(e)}_{R[\ut]}$, and therefore $\xi \cdot R[\ul{t}]_j \sq R[\ul{t}]_{\geq j + i}$ for all $j \in \Z$. Let $\xi = \sum_{k \in \Z} \xi_k$ where $\xi_k$ is the homogeneous component of degree $k$ for $\xi$, and observe that if $g \in R[\ul{t}]$ is homogeneous of degree $d$, then $\xi_k \cdot g$ is the degree $k + d$ homogeneous component of $\xi \cdot g \in R[\ul{t}]_{\geq i + d}$. We conclude that $\xi_k \cdot g = 0$ whenever $k < i$, which proves the statement.
	
	We now claim that $(D^{(e)}_{R[\ut]})_{i} = (D^{(e)}_{R[\ut]})_0 R[t]_{i}$ for all $i \geq 0$, which together with part (i) gives part (ii). Since $(D^{(e)}_{R[\ut]})_i = D^{(e)}_R \otimes_{\F_p} (D^{(e)}_{\F_p[\ut]})_i$ and $(D^{(e)}_{R[\ut]})_0 = D^{(e)}_R \otimes_{\F_p} (D^{(e)}_{\F_p[\ut]})_0$, we reduce to the case $R = \F_p$. Let $\sigma^{(e)} \in D^{(e)}_{\F_p[\ut]}$ denote the unique operator of level $e$ such that 
	$$\sigma^{(e)} \cdot \ut^{\ul{a}} = \begin{cases} 1 \text{ for } \ul{a} = (p^e-1, \ds, p^e-1) \\ 0 \text{ otherwise} \end{cases}$$
	for all $\ul{a} \in \{0, \ds, p^e-1\}^r$; observe $\sigma^{(e)}$ is homogeneous of degree $-r(p^e-1)$. Then $D^{(e)}_{\F_p[\ut]}$ is spanned over $\F_p$ by the operators of the form $\ut^{\ul{b}} \sigma^{(e)} \ut^{\ul a}$, where $\ul b$ ranges through $(\Z_{\geq 0})^r$ and $\ul a$ ranges through $\{0, \ds, p^e-1\}^r$, and therefore $D^{(e)}_{\F_p[\ut]}$ is spanned by those for which $|\ul b| - r(p^e - 1) + |\ul a| = i$ or, equivalently, $|\ul a | - i = r(p^e - 1) - |\ul b|$. 
	If $\ut^{\ul b} \sigma^{(e)} \ut^{\ul a} \in (D_{R[\ut]})_i$, then there is a multi-exponent $\ul c \in N_0^r$ with $|\ul c| = i$ such that $a_j \geq c_j$ for every $i$,  because $r (p^e - 1) - |\ul b| \geq 0$. We can thus write $\ut^{\ul b} \sigma^{(e)} \ut^{\ul a} = \ut^{\ul b} \sigma^{(e)} \ut^{\ul a - c} \ut^{\ul c}$, which proves the claim.  
\end{proof}

Given an integer $i \geq 0$, we denote by $s_{p^i}$ the unique $R$-linear operator on $R[\ut]$ with the property that
$$s_{p^i} \cdot \ut^{\ul{a}} = (- |\ul{a}| - r)_i  \ \ut^{\ul{a}}$$
for all $\ul{a} \in (\Z_{\geq 0})^r$, where $(-)_i$ denotes $i$-th $p$-adic digit. In the following lemma we aggregate some properties of these operators.
\begin{lemma} \label{lemma-spi-properties}
	We have:
	\begin{enumerate}[(i)]
		\item For all integers $i \geq 0$ and $e > i$, 
		$$s_{p^i} \cdot \ut^{(p^e - 1)\bone - \ul{a}} = |\ul{a}|_i \ut^{(p^e - 1) \bone - \ul a}.$$
		\item For all integers $i \geq 0$, $s_{p^i}$ is in $(D^{(i + 1)}_{R[\ut]})_0$.
		
		\item The operators $s_{p^i}$ commute with each other. 
		
		\item For all integers $i \geq 0$ we have $(s_{p^i})^p = s_{p^i}$ or, equivalently, $\prod_{j = 0}^{p-1} (s_{p^i} - j) = 0$.
		
		\item If $M$ is an $\F_p$-vector space equipped with an action of the operators $s_{p^0}, s_{p^1}, \ds, s_{p^{e-1}}$, then $M$ splits as a sum of multi-eigenspaces for the action of these operators; namely, $M = \bigoplus_{\alpha \in \F_p^e} M_\alpha$ where for all $\alpha = (\alpha_0, \ds, \alpha_{e-1}) \in \F_p^e$ we define $M_\alpha := \{u \in M : s_{p^i} \cdot u = \alpha_i u \ \ \forall i = 0, \ds, e-1\}$.
	\end{enumerate}
\end{lemma}

\begin{proof}
	For (i) we simply observe that $(\, -|(p^e - 1)\bone - \ul{a}| - r\,)_i = (- rp^e + |\ul{a}|)_i = |\ul{a}|_i$ (recall that whenever $\alpha \equiv \beta \mod p^i \Z_p$ we have $\alpha_i = \beta_i$). For part (ii), we note that $s_{p^i}$ has degree zero, that it is $R$-linear and that it commutes with multiplication by $t_j^{p^{i + 1}}$ for all $j = 1, \ds, r$. Parts (iii) and (iv) follow because each $s_{p^i}$ is $R$-linear and acts on monomials $\ut^{\ul{a}}$ by an $\F_p$-scalar. Part (v) follows from (iii) and (iv). 
\end{proof}

\begin{remark}
	It is possible to give a formula for the operators $s_{p^i}$ in terms of partial derivatives:
	$$s_{p^i} = - \sum_{|\ul{a}| = p^i} \partial_{t_1}^{[a_1]} t_1^{a_1} \cds \partial_{t_r}^{[a_r]} t_r^{a_r},$$
	where the $\partial^{[a]}$ notation stands for divided power differential operators \cite[Proposition~3.3]{QG19}. We remark that the transpose of these operators already appeared in work of Ma and Zhang \cite{MaZhang14} as higher-order Euler operators. 
\end{remark}

\subsection{The ring of continuous functions from $\Zp$ to $\FF_p$} \label{SubSec-CZpFp} \ 

\begin{definition}
	Given a set $X$ and an integer $e \geq 0$ we denote by $\cts^e(\Zp, X)$ the collection of all functions $\phi: \Zp \to X$ such that $\phi(\alpha) = \phi(\beta)$ whenever $\alpha \equiv \beta \mod p^e$. We denote $\cts(\Zp, X) = \bigcup_{e = 0}^\infty \cts^e(\Zp, X)$, and we call $\cts(\Zp, X)$ the set of continuous functions from $\Zp$ to $X$.
\end{definition}

Note that these are indeed the continuous functions when $X$ is endowed with the discrete topology, which is the only case we consider. When $A$ is a ring, the sets $\cts^e(\Zp, A)$ and $\cts(\Zp, A)$ acquire $A$-algebra structures by pointwise addition and multiplication. 

A function in $\cts^e(\Zp, X)$ is uniquely determined by its values on $\{0, 1, \ds, p^e-1\}$. Consequently, a function in $\cts(\Zp, X)$ is uniquely determined by its values in $\Z_{\geq 0}$ and, given an $\F_p$-vector space $V$, we have canonical isomorphisms
$$\begin{tikzcd} [row sep = tiny]
V \otimes_{\F_p} \cts^e(\Zp, \F_p) \arrow[r, "\sim"] & \cts^e(\Zp, V), \\
V \otimes_{\F_p} \cts(\Zp, \F_p) \arrow[r, "\sim"] & \cts(\Zp, V).
\end{tikzcd}$$

Fix an integer $e \geq 0$, and let us consider the algebra $\cts^e(\Zp, \F_p)$.

\begin{remark}
	The more geometrically-minded reader might like to think of the results in this section via the following remark, which was pointed out to the third author by Bhatt (see \cite[Remark III.2]{QGt} for a more detailed discussion). There is a homeomorphism $\Spec(\cts(\Zp, \F_p)) \cong \Zp$ which becomes an isomorphism of ringed spaces when we equip $\Zp$ with the sheaf of rings $\underline{\F_p}$ associated to $\F_p$. In particular, all local rings of $\cts(\Zp, \F_p)$ are fields, and therefore all $\cts(\Zp, \F_p)$-modules are flat. 
\end{remark}

To every $p$-adic integer $\alpha \in \Zp$ we associate the maximal ideal
$$\fm^{(e)}_\alpha : = \{\varphi \in \cts^e(\Zp, \F_p) \ | \ \varphi(\alpha) = 0 \};$$
note that $\fm^{(e)}_\alpha \equiv \fm^{(e)}_\beta$ whenever $\alpha \equiv \beta \mod p^e$. We have algebra isomorphisms
$$\cts^e(\Zp, \F_p) \cong \Fun(\{0, 1, \ds, p^e - 1\}, \F_p) \cong \F_p \ e_1 \times \cds\ \times \F_p \ e_{p^e -1}$$
where the $e_i$ are orthogonal idempotents. In particular, every $\cts^e(\Zp, \F_p)$-module $M$ splits as $M = \bigoplus_{a = 0}^{p^e - 1} M_{(a)}$, where $M_{(a)} = \Ann_{\fm^{(e)}_a} (M)$.

\begin{remark} \label{rmk-Ce-mod-mAlpha-flat}
	Note that for every $a \in \{0, 1, \ds, p^e-1\}$ the quotient $M / \fm^{(e)}_a$ is naturally identified with the submodule $M_{(a)}$. It follows that if $N \sq M$ is a submodule then the natural map $N / \fm^{(e)}_a \to M / \fm^{(e)}_a$ is injective, which shows that $\cts^e(\Zp, \F_p) / \fm^{(e)}_\alpha$ is a flat $\cts^e(\Zp, \F_p)$-module for every $\alpha \in \Zp$. 
\end{remark}

We now give a presentation of the algebra $\cts^e(\Z_p, \F_p)$. 

Recall that, given a $p$-adic integer $\alpha \in \Zp$, we denote $\alpha_i$ the $i$-th digit in the $p$-adic expansion of $\alpha$ (see Subsection \ref{subscn-padic-basep}). Given a integer $e \geq 0$ we denote by $\sigma_{p^e}: \Zp \to \F_p$ the function $\sigma_{p^e} (\alpha) = \alpha_e$; note that $\sigma_{p^e}$ is in $\cts^{e+1}(\Zp, \F_p)$. The function $\sigma_{p^e}$ can be expressed using binomial coefficients: $\sigma_{p^e}(\alpha) = {\alpha \choose p^e}$. To see this, use Lucas' theorem to observe that whenever $\alpha =n\in \Z_{\geq 0}$, we have $\binom{n}{p^e} \equiv \alpha_e$ modulo $p$, and that therefore $\sigma_{p^e}$ is the unique continuous extension to $\Zp$ of the map $n \mapsto \binom{n}{p^e}$ (considered as a map from $\Z$ to $\F_p$).

\begin{lemma}
	The ring $\cts^e(\Zp, \F_p)$ is generated by the operators $\sigma_{p^0}, \sigma_{p^1}, \ds, \sigma_{p^{e-1}}$ as an $\F_p$-algebra. Moreover, the assignment $x_i \mapsto \sigma_{p^i}$ induces an $\F_p$-algebra isomorphism
	$$\frac{\F_p[x_0, x_1, \ds, x_{e-1}]}{\big( x_i^p - x_i \ | \ i = 0, \ds, e-1 \big)} \cong \cts^e(\Zp, \F_p).$$
\end{lemma}
\begin{proof}
	Every function in $\cts^e(\Zp, \F_p)$ is determined uniquely by its values on ${\{0, 1, \ds, p^e-1\}}$ and, conversely, any $\F_p$-valued function on $\{0, 1, \ds, p^e-1\}$ extends uniquely to an element of $\cts^e(\Zp, \F_p)$. 
	
	By identifying every element of $\{0 , \ds, p^e-1\}$ with its base-$p$ expansion we obtain a bijection ${\{0, \ds, p^e-1\} \cong \F_p^e}$. We therefore have $\F_p$-algebra isomorphisms
	$$\cts^e(\Zp, \F_p) \cong \Fun(\{0, \ds, p^e-1\}, \F_p) \cong \Fun(\F_p^e, \F_p),$$
	and one checks that, under these identifications, the functions $\sigma_{p^e}$ are sent to the coordinate functions on~$\F_p^e$. Since $\F_p^e$ is a finite set every $\F_p$-valued function on it is a polynomial on the coordinate functions. We conclude that $\cts^e(\Zp, \F_p)$ is a quotient of $\F_p[x_0, \ds, x_{e-1}] / {(x_i^p - x_i)}$	and, since both of these algebras have the same number of elements, the result follows.
\end{proof}

We conclude that $\cts(\Zp, \F_p)$ is generated by the operators $\sigma_{p^i}$ $(i \in \Z_{ \geq 0})$, and that we have an algebra isomorphism
$$\frac{\F_p[x_0, x_1, \ds]}{\big( x_i^p - x_i \ | \ i \in \Z_{\geq 0} \big)} \cong \cts(\Zp, \F_p).$$
In particular, given an $F$-finite ring $R$ and an element $f \in R$, we can identify $\cts(\Zp, R_f) = R_f [\sigma_{p^0}, \sigma_{p^1}, \ds, ]$ and $\cts(\Zp, D_R) = D_R [\sigma_{p^0}, \sigma_{p^1}, \ds]$; these are positive characteristic analogues of the objects $R_f[s]$ and $D_R[s]$ of classical Bernstein-Sato theory, while the algebra $\cts(\Zp, \F_p)$ plays the role of the algebra $\CC[s]$ for the operator $s = - \sum_{i = 1}^r \partial_{t_i} t_i$. 

Let us now turn our attention to the algebra $\cts(\Zp, \F_p)$ and its modules. Once again we associate to every $\alpha \in \Zp$ a maximal ideal
$$\fm_\alpha := \{\varphi \in \cts(\Zp, \F_p) \  |\  \varphi(\alpha) = 0 \}.$$

Every maximal ideal of $\cts(\Zp,\FF_p)$ is of the form $\fm_\alpha$ for some $\alpha \in \Zp$ \cite{Bitoun2018}. 

Given a $\cts(\Zp, \F_p)$-module $M$ and a $p$-adic integer $\alpha \in \Zp$, we denote by $M_\alpha$ the quotient
$$M_\alpha : =  M / \fm_\alpha M.$$
If $N \sq M$ is a submodule, $N_\alpha$ is naturally a submodule of $M_\alpha$ by the following result.

\begin{lemma} \label{lemma-CmAlpha-flat}
	The module $\cts(\Zp, \F_p) / \fm_\alpha \cts(\Zp, \F_p)$ is flat over $\cts(\Zp, \F_p)$. 
\end{lemma}
\begin{proof}
	For simplicity of notation, let us denote the algebra $\cts(\Zp, \F_p)$ (resp. $\cts^e(\Z_p, \F_p)$) by $C$ (resp. $C^e$). Note that $C/ \fm_\alpha = \lim\limits_{\to e} C^e/\fm^{(e)}_\alpha$, and that if $N$ is a $C$-module then there is a natural map
	$$\lim\limits_{\to e} (C^e / \fm^{(e)}_\alpha \otimes_{C^e} N) \longrightarrow (C / \fm_\alpha) \otimes_C N,$$
	which we claim is an isomorphism. Indeed, giving an $C$-multilinear map $C/\fm_\alpha \times N \to W$ is equivalent to giving a compatible collection of $C^e$-multilinear maps $C^e / \fm^{(e)}_\alpha \times N \to W$, which shows that both objects have the same universal property.
	
	We know that $C^e / \fm^{(e)}_\alpha$ is flat over $C^e$ (cf.~Remark \ref{rmk-Ce-mod-mAlpha-flat}); since taking limits is an exact operation the result follows.	
\end{proof}

\begin{lemma} \label{lemma-MalphaZero-implies-Mzero}
	Let $M$ be a $\cts(\Zp, \F_p)$-module. If $M_\alpha = 0$ for all $p$-adic integers $\alpha \in \Zp$, then $M = 0$.
\end{lemma}
\begin{proof}
	Fix an integer $e \geq 0$. Given $a \in \{0, \ds, p^e-1\}$ denote by $\chi^{(e)}_a \in \cts^e(\Zp, \F_p)$ the function such that $\chi^{(e)}_a(\beta) = 1$ whenever $\beta \equiv a \mod p^e$ and such that $\chi^{(e)}_a(\beta) = 0$ otherwise. 
We observe that a function $\varphi \in \cts(\Zp, \F_p)$ belongs to $\fm_\alpha$ if and only if $\chi^{(e)}_\alpha \varphi = 0$ for a sufficiently large $e$; indeed, it suffices to take $e$ large enough so that $\varphi \in \cts^e(\Zp, \F_p)$.  We conclude that, given an element $u \in M$ and a $p$-adic integer $\alpha \in \Zp$ there exists some large $e_\alpha$ such that $\chi^{(e_\alpha)}_\alpha u = 0$ or, equivalently, $(1 - \chi^{(e_\alpha)}_\alpha) u = u$. 
	
	The union $\bigcup_{\alpha \in \Zp} (\alpha + p^{e_\alpha} \Zp)$ forms an open cover of $\Zp$ which, by the compactness of $\Zp$, admits a finite subcover $\Zp = \bigcup_{i = 1}^n (\alpha_{(i)} + p^{e_{\alpha_{(i)}}} \Zp)$. We conclude that
	\[u = (1- \chi^{(e_{\alpha_{(1)}})}_{\alpha_{(1)}}) \cds (1- \chi^{(e_{\alpha_{(n)}})}_{\alpha_{(n)}}) u  = 0. \qedhere\]
\end{proof}

\begin{proposition} \label{prop-CzpModules-split}
	Let $M$ be a $\cts(\Zp, \F_p)$-module. Suppose that there are only finitely many $\alpha \in \Zp$ such that $M_\alpha \neq 0$, say $\alpha_{(1)}, \ds, \alpha_{(n)}$. Then the natural map $M \xrightarrow{\sim} \bigoplus_{i = 1}^n  M_{\alpha_{(i)}}$ is an isomorphism that identifies $M_{\alpha_{(i)}}$ with $\Ann_M(\fm_\alpha)$. 
\end{proposition}
\begin{proof}
	Let $K$ (resp. $Q$) be the kernel (resp. cokernel) of the map $M \to \bigoplus_{i = 1}^n M_{\alpha_{(i)}}$. We thus have an exact sequence
	$$0 \to K \to M \to \bigoplus_{i = 1}^n M_{\alpha_{(i)}} \to Q \to 0.$$
	We claim that we have $K_\beta = Q_\beta = 0$  for all $\beta \in \Zp$. 
	
	Indeed, if $\beta \neq \alpha_{(i)}$ for any $i$, then applying the functor $(-)_\beta$ to the exact sequence above yields
	$$0 \to K_\beta \to 0 \to 0 \to Q_\beta \to 0,$$
	by Lemma \ref{lemma-CmAlpha-flat}.
	If $\beta = \alpha_{(i)}$, then we get
	$$0 \to K_\beta \to M_\beta \xrightarrow{\id} M_\beta \to Q_\beta \to 0.$$
	From Lemma \ref{lemma-MalphaZero-implies-Mzero}, we conclude that $K = Q = 0$. 
\end{proof}

Finally, we illustrate how the algebras $\cts^e(\Zp, \F_p)$ and $\cts(\Zp, \F_p)$ arise naturally in the context of differential operators. 

We consider the map
$$\Delta: \cts(\Zp, D_R) \longrightarrow (D_{R[\ut]})_0$$
that sends $\xi \in \cts(\Zp, D_R)$ to the unique operator $\t{\xi}$ on $R[\ut]$ such that 
$${\t{\xi} \cdot f \ut^{\ul{a}}} = {(\xi(-r-|\ul a|) \cdot f) \ut^{\ul{a}}}$$
 for every $f \in R$ and all $\ul{a}\in \NN^{r}$. Note that, for all $e \geq 0$, whenever $\xi \in \cts^e(\Zp, D^{(e)}_R)$ we get that $\Delta(\xi) \in (D^{(e)}_{R[\ut]})_0$. We therefore get an induced map $\Delta^e: \cts^e(\Zp, D^{(e)}_R) \to (D^{(e)}_{R[\ut]})_0$.

\begin{lemma} \label{lemma-ctsfn-diffop-isom}
	The morphism $\Delta:\cts(\Zp, D_R) \longrightarrow (D_{R[\ut]})_0$ previously constructed is injective. Moreover, when $r=1$, $\Delta$ is an isomorphism.
\end{lemma}
\begin{proof}
	We have that $\cts(\Zp, D_R) = D_R \otimes_{\FF_p} \cts(\Zp, \F_p)$ and that $(D_{R[\ut]})_0 = D_R \otimes_{\F_p} (D_{\F_p[\ut]})_0$ (see Lemma \ref{lemma-D_R-addvar}), and the morphism respects these decompositions. It therefore suffices to prove the claims in the case where $R = \F_p$. 
	
	In this case, note that an operator $\delta \in \cts(\Zp, D_{\FF_p})$ is sent to the unique operator $\t{\delta}$ on $\F_p[\ut]$ for which $\t{\delta} \cdot \ut^{\ul{a}} = \delta(-r-|a|)  \ut^{\ul{a}}$. If $\t{\delta}=0$, then $\delta=0$.
	
	Now, let $r=1$. If $\delta$ is a differential operator of degree zero on $\F_p[t]$ then there exists a unique function $\phi_\delta: \Z_{\geq 0} \to \F_p$ such that $\delta \cdot t^a = \phi_\delta(a) t^a$ for all $a \in \Z_{\geq 0}$. If $\delta\in (D^{(e)}_{\FF_p[t]})_0$, then $\delta \cdot t^{a+b p^e} = (\delta \cdot t^a) t^{bp^e}$ for all $b\in\NN$, so $\phi_{\delta}(a)=\phi_{\delta}(a + bp^e)$; therefore, $\phi_{\delta}$ extends to an element of $\cts^e(\Zp, \F_p)$. The assignment $\delta \mapsto \xi_\delta$, where $\xi_\delta(a) = \phi_\delta(1 - a)$ provides a two-sided inverse to the morphism given. 
\end{proof}

\begin{remark}
	It may seem unnatural to take $\xi(-r-|\ul a|)$ in the definition of $\Delta$, as opposed to $\xi(|\ul a|)$. This is a natural consequence of the convention of working with the operator $s_1 = - \sum_{i = 1}^r \partial_{t_i} t_{i}$ in characteristic zero, as opposed to the operator $\sum_{i = 1}^r t_i \partial_{t_i}$. Note that we have $\Delta( \sigma_{p^e}) = s_{p^e}$ for every $e \geq 0$. 
\end{remark}

%%%%%%%%%%%%%%%%%%%%%%%%%%%%%%%%%%%%%%%%%%%%%%%%%%%%%%%%%%%%%
\section{Differential jumps}
%%%%%%%%%%%%%%%%%%%%%%%%%%%%%%%%%%%%%%%%%%%%%%%%%%%%%%%%%%%%%

\begin{definition} \label{def-DiffJump}
	Let $R$ be an $F$-finite ring and $\fa\sq R$ be an ideal. We say that an integer $n \geq 0$ is \emph{differential jump of level $e$} of $\fa$ if the inclusion $D^{(e)}_R \cdot \fa^n \supseteq D^{(e)}_R \cdot \fa^{n+1}$ is proper. We write $\cB^\bullet_\fa(p^e)$ for the collection of all differential jumps of level $e$. 
\end{definition}

We note that $n\in \cB^\bullet_\fa(p^e)$ if and only if $\fa^n \not\sq D^{(e)}_R \cdot \fa^{n+1}$.

\begin{remark}\label{rem:reduce-to-local} \ 
	\begin{enumerate}
		\item	If $W \sq R$ is a multiplicative subset and $\fa \sq R$ is an ideal, then we have \[{D^{(e)}_{W^{-1} R} \cdot \fa W^{-1} R = (D^{(e)}_R \cdot \fa) W^{-1} R};\] cf., \cite[Proposition~2.17]{BJNB19}. If $g_1, \ds, g_k \in R$ are such that $(g_1, \ds, g_k) = R$ and $\Max(R)$ denotes the collection of all maximal ideals of $R$, then we have
		$$\cB^\bullet_\fa(p^e) = \bigcup_{i = 1}^k \cB^\bullet_{\fa R_{g_i}}(p^e) = \bigcup_{ \fm \in \Max(R)} \cB^\bullet_{\fa R_\fm} (p^e).$$
		
		\item	If $(R,\fm)$ is local and $F$-finite and $\fa \sq R$ is an ideal then we have $\widehat{R} \otimes_R (D_R \cdot \fa) \cong D_{\widehat{R}} \cdot (\fa \widehat{R})$ \cite[Proposition  2.24]{BJNB19}. We conclude that ${D^{(e)}_R \cdot \fa^n}/{D^{(e)}_R \cdot \fa^{n+1}}\neq 0$ if and only if 
		\[ 0\neq \widehat{R} \otimes_R \frac{D^{(e)}_R \cdot \fa^n}{D^{(e)}_R \cdot \fa^{n+1}} \cong \frac{\widehat{R} \otimes_R (D^{(e)}_R \cdot \fa^n)}{\widehat{R} \otimes_R (D^{(e)}_R \cdot \fa^{n+1})} \cong \frac{ D^{(e)}_{\widehat{R}} \cdot (\fa\widehat{R})^n}{D^{(e)}_{\widehat{R}} \cdot (\fa\widehat{R})^{n+1}},\]
		and therefore $\cB^\bullet_\fa(p^e) = \cB^\bullet_{\fa \widehat R}(p^e).$
		
		\item	If $R$ is graded with homogeneous maximal ideal $\fm$, and $\fa$ is homogeneous, then ${D^{(e)}_R \cdot \fa^n}/{D^{(e)}_R \cdot \fa^{n+1}}$ is a graded module. Therefore, ${D^{(e)}_R \cdot \fa^n}/{D^{(e)}_R \cdot \fa^{n+1}}\neq 0$ if and only if 
		\[0\neq \left(\frac{D^{(e)}_R \cdot \fa^n}{D^{(e)}_R \cdot \fa^{n+1}}\right)_{\fm}\cong\frac{D^{(e)}_{R_{\fm}} \cdot (\fa{R_{\fm}})^n}{D^{(e)}_{R_{\fm}} \cdot (\fa{R_{\fm}})^{n+1}},\]
		and therefore $\cB^\bullet_\fa(p^e) = \cB^\bullet_{\fa R_{\fm}}(p^e)$. 
	\end{enumerate}
\end{remark}

To compute differential jumps, we can also reduce to the case of an infinite residue field by the following lemma.

\begin{lemma}\label{lem:redu-inf-resfield}
	Let $(R,\fm,K)$ be an $F$-finite local ring and consider the extension $(S, \fn, L)$ given by 
	$$(S, \fn, L) := (R[x]_{\fm R[x]}, \fm R[x]_{\fm R[x]}, K(x)).$$
	Then:
	\begin{enumerate}[(i)]
		\item The extension $(S, \fn, L)$ is faithfully flat, $F$-finite, and local. 
		\item For every ideal $\fa \sq R$ we have $\ell(\fa)=\ell(\fa S)$, where $\ell$ denotes analytic spread.
		\item If $R$ is $F$-split, then so is $S$.
		\item For every integer $e \geq 0$ and every ideal $\fa \sq R$ we have $\cB^\bullet_{\fa}(p^e) = \cB^\bullet_{\fa S}(p^e)$.
	\end{enumerate}
\end{lemma}

\begin{proof}
	The statements in (i), (ii), (iii) are standard \cite[\S 8.5]{SwHu}. For (iv) observe that, by Lemma~\ref{lemma-D_R-addvar}, we have $D^{(e)}_{R[x]} \cdot (\fb R[x]) = (D^{(e)}_{R} \cdot \fb) R[x]$ for any ideal $\fb \sq R$. Then, by Remark~\ref{rem:reduce-to-local}, we have $D^{(e)}_S \cdot (\fb S) = (D^{(e)}_S \cdot \fb) S$. The claim on differential jumps then follows by faithful flatness.
\end{proof}

Differential jumps can also be characterized in terms of $D^{(e)}_R$-ideals.

\begin{definition}
	Let $R$ be an $F$-finite ring. An ideal $\fa$ of $R$ is a {$D^{(e)}_R$-ideal} if it is a $D^{(e)}_R$-submodule of $R$, equivalently,  $D^{(e)}_R \cdot \fa \sq \fa$.
\end{definition}

We record two natural families of $D^{(e)}_R$-ideals.

\begin{lemma}\label{lem:d-ideals}
	Let $\fa \sq R$ be an ideal.
	\begin{enumerate}
		\item The $e$-th Frobenius power $\fa^{[p^e]}$ of $\fa$ is a $D^{(e)}_R$-ideal.
		\item The Cartier preimage $I_e(\fa):=\{f\in R \ | \ \cC^e_R \cdot f \sq \fa\}$ of $\fa$ is a $D^{(e)}_R$-ideal.
	\end{enumerate}
\end{lemma}
\begin{proof} Let  $\delta\in D^{(e)}_R$.

	For (i), let $\fa=(f_1,\dots,f_r)$ and $\sum_i g_i f_i^{p^e}\in \fa^{[p^e]}$. Then $\delta(\sum_i g_i f_i^{p^e})= \sum_i f_i^{p^e} \delta(g_i) \in \fa^{[p^e]}$.
	
	For (ii), given $f\in I_e(\fa)$, and $\psi\in \cC^e_R$, we have $\psi\cdot (\delta \cdot f)=(\psi \circ \delta) \cdot f \in \fa$ since $\psi \circ \delta \in \cC^e_R$. Thus, $\delta \cdot f \in I_e(\fa)$.
\end{proof}

\begin{remark}\label{rem:existjumps} If $\fa=R$ is the unit ideal, then $\cB^\bullet_\fa(p^e)=\varnothing$ for all $e$. Conversely, if $\fa\subsetneqq R$ is a proper ideal, we can take $k$ such that $\fa^k \subseteq \fa^{[p^e]}$, and then $D^{(e)}_R \cdot \fa^{k} \subseteq D^{(e)}_R \cdot \fa^{[p^e]} = \fa^{[p^e]}\neq R = D^{(e)} \cdot \fa^0$ for every $e$, so $\cB^\bullet_\fa(p^e)\neq \varnothing$ for every $e$.
\end{remark}

\begin{lemma} \label{lemma-testSFR3}
	Let $\fa, \fb \sq R$ be two ideals and $e \geq 0$ be an integer. If $D^{(e)}_R \cdot \fa = D^{(e)}_R \cdot \fb$, then $\cC^e_R \cdot \fa = \cC^e_R \cdot \fb$. 
\end{lemma}	
\begin{proof}
	Note that $\cC^e_R \circ D^{(e)}_R = \cC^e_R$. We therefore get $\cC^e_R \cdot \fa = \cC^e_R \cdot (D^{(e)}_R \cdot \fa) = \cC^e_R \cdot (D^{(e)}_R \cdot \fb) = \cC^e_R \cdot \fb$.
\end{proof}

\begin{definition}
	Let $R$ be an $F$-finite ring and fix an ideal $\fa \sq R$. Let $\fb \sq R$ be a proper $D^{(e)}_R$-ideal such that $\fa \sq \sqrt{\fb}$. We define 
	$$\cB^\fb_\fa(p^e) := \max \{ n \geq 0 : D^{(e)}_R \cdot \fa^n \not\sq \fb \}.$$
\end{definition}

Recall that, by convention, we have $\fa^0 = R$ and, since $\fb$ is proper by assumption, the set $\{n \geq 0 : D^{(e)}_R \cdot \fa^n \not \sq \fb \}$ is never empty. Moreover, since $\fa \sq \sqrt \fb$, we may pick some $m > 0$ such that $\fa^m \sq \fb$. If we pick $k > 0$ so that $\fa^k \sq \fa^{\fm[p^e]}$, then $D^{(e)}_R \cdot \fa^k \sq D^{(e)}_R \cdot \fa^{m [p^e]} \sq \fa^{\fm [p^e]} \sq \fb$. We conclude that the maximum above does indeed exist. 

\begin{remark} 
	If $R$ is regular, then\footnote{This is well-known to experts; it follows, for example, by using Frobenius descent \cite{AMBL}.} every $D^{(e)}_R$-ideal $\fb$ can be realized as $\fb = \fc^{[p^e]}$ for some ideal $\fc$. In this case, we have $\cB^{\fb}_\fa(p^e) = \nu^{\fc}_\fa(p^e)$, where $\nu^\fc_\fa(p^e)$ are the $\nu$-invariants introduced by  \Mustata, Takagi, and Watanabe \cite{MTW}. \
\end{remark}

\begin{lemma}\label{lem:thresh}
	If $\fa \sq R$ is an  ideal, then
	\[ \cB^\bullet_\fa(p^e) = \{ \cB^\fb_\fa(p^e) \ | \ \fb \ \text{is a $D^{(e)}_R$-ideal}, \ \fa \sq \sqrt{\fb} \neq (1) \}.\]
\end{lemma}
\begin{proof}
	If $\fb$ is a $D^{(e)}_R$-ideal with $\fa \sq \sqrt{\fb}$, and $n=\cB^\fb_\fa(p^e)$, then $D^{(e)}_R\cdot \fa^n \not\sq\fb$ and $D^{(e)}_R\cdot \fa^{n+1} \sq\fb$, so $D^{(e)}_R\cdot \fa^n \neq D^{(e)}_R\cdot \fa^{n+1}$. Conversely, if $D^{(e)}_R\cdot \fa^n \neq D^{(e)}_R\cdot \fa^{n+1}$, then $\fb=D^{(e)}_R \cdot \fa^{n+1}$ is a $D^{(e)}_R$-ideal, and $n=\cB^\fb_\fa(p^e)$.
\end{proof}

\begin{lemma}\label{lem:De-nested}
	Let $\fa \sq R$ be an ideal and $n, e, a \geq 0$ be integers. If $D^{(e)}_R \cdot \fa^n = D^{(e)}_R \cdot \fa^{n+1}$, then $D^{(e+a)}_R \cdot \fa^n = D^{(e+a)}_R \cdot \fa^{n+1}$. Thus, $\cB^\bullet_\fa(p^e) \supseteq \cB^\bullet_\fa(p^{e+a})$.
\end{lemma}
\begin{proof}
	If $D^{(e)}_R \cdot \fa^n = D^{(e)}_R \cdot \fa^{n+1}$, then \[D^{(e+a)}_R \cdot \fa^n = D^{(e+a)}_R \cdot (D^{(e)}_R \cdot \fa^n) = D^{(e+a)}_R \cdot (D^{(e)}_R \cdot \fa^{n+1}) = D^{(e+a)}_R \cdot \fa^{n+1}.\qedhere\]
\end{proof}

\begin{lemma}\label{lem:frob-powers} Let $\fa \sq R$ be an ideal generated by $r$ elements.	If $n\geq m p^e + (r-1)(p^e-1)$, then $\fa^n=\fa^{n-mp^e} (\fa^{[p^e]})^m$.
\end{lemma}
\begin{proof}
	The statement reduces to the case $m=1$. The containment of the right-hand side in the left is clear. For the other, by the pigeonhole principle, any monomial of degree greater than $r(p^e-1)$ in the generators of $\fa$ must be a multiple of a $p^e$th power of a generator, from which the claim is clear.
\end{proof}

\begin{proposition}\label{prop:subtract-pe}
	Let $R$ be an $F$-finite ring and $\fa\sq R$ be an ideal. If $\fa$ is generated by $r$ elements, $n\geq r(p^e-1) + 1$, and $D^{(e)}_R \cdot \fa^{n-p^e} = D^{(e)}_R \cdot \fa^{n-p^e+1}$, then $D^{(e)}_R \cdot \fa^n = D^{(e)}_R \cdot \fa^{n+1}$. Hence, if $n\in \cB^\bullet_\fa(p^e)$, then $n-p^e\in \cB^\bullet_\fa(p^e)$. 
	
	If $\fa$ is principal and generated by a nonzerodivisor, $n\geq p^e$, and $D^{(e)}_R \cdot \fa^n = D^{(e)}_R \cdot \fa^{n+1}$, then $D^{(e)}_R \cdot \fa^{n-p^e} = D^{(e)}_R \cdot \fa^{n-p^e+1}$.
\end{proposition}
\begin{proof}
	By Lemma~\ref{lem:frob-powers}, we have that $D^{(e)} \cdot \fa^n = D^{(e)} \cdot (\fa^{[p^e]}\fa^{n-p^e}) = \fa^{[p^e]} (D^{(e)} \cdot \fa^{n-p^e})$ and $D^{(e)} \cdot \fa^{n+1}= \fa^{[p^e]} (D^{(e)} \cdot \fa^{n-p^e+1})$ likewise. Then, if $D^{(e)} \cdot \fa^{n-p^e}=D^{(e)} \cdot \fa^{n-p^e+1}$, we must have $D^{(e)} \cdot \fa^n =D^{(e)} \cdot \fa^{n+1}$.
	
	If $f$ is a nonzerodivisor, $\fa=(f)$, and $n\geq p^e$, we have $f^{p^e} D^{(e)}_R \cdot \fa^{n-p^e} = D^{(e)}_R \cdot \fa^{n}$. Then, $D^{(e)}_R \cdot \fa^n = D^{(e)}_R \cdot \fa^{n+1}$ implies $D^{(e)}_R \cdot \fa^{n-p^e} = D^{(e)}_R \cdot \fa^{n-p^e+1}$.
\end{proof}

\begin{lemma}\label{lem:chain-jumps}
	Let $\fa \sq R$ be an ideal and fix integers $n<m$. Then $D^{(e)}_R \cdot \fa^n = D^{(e)}_R \cdot \fa^m$ if and only if $D^{(e)}_R \cdot \fa^j = D^{(e)}_R \cdot \fa^{j+1}$ for all $n\leq j \leq m-1$. Equivalently, $D^{(e)}_R \cdot \fa^n = D^{(e)}_R \cdot \fa^m$ if and only if $[n,m) \cap \cB_\fa^\bullet(p^e) = \varnothing$.
\end{lemma}
\begin{proof}
The statement follows from the chain of ideals
	\[ D^{(e)}_R \cdot \fa^n \supseteq  D^{(e)}_R \cdot \fa^{n+1} \supseteq \cdots \supseteq  D^{(e)}_R \cdot \fa^{m}. \qedhere \]
\end{proof}

\begin{lemma}\label{LemmaEqualityFrobPowers}
	Suppose $R$ is $F$-split. Let $\fb \sq \fa \sq R$ be ideals and $e, a \geq 0$ be integers. Then $D^{(e)}_R \cdot \mathfrak{a} = D^{(e)}_R \cdot \mathfrak{b}$ if and only if $D^{(e+a)}_R \cdot \mathfrak{a}^{[p^{a}]} = D^{(e+a)}_R \cdot  \mathfrak{b}^{[p^{a}]}$.
\end{lemma}
\begin{proof} It suffices to prove the lemma for ${a}=1$.
	Fix a splitting $\sigma$ of the Frobenius. Let $\mathfrak{a}=(f_1,\dots,f_t)$ and $\mathfrak{b}=(g_1,\dots,g_s)$. 
	
	For the forward implication, if $D^{(e)}_R \cdot  \mathfrak{a} = D^{(e)}_R \cdot  \mathfrak{b}$, then  $f_1,\dots,f_t\in D^{(e)}_R\cdot \mathfrak{b}$. Then, $f_i=\sum_j \delta_{ij}(g_j)$ for some $\delta_{ij}\in D^{(e)}_R$. Thus,  $F \circ \delta_{ij} \circ \sigma \in {D^{e+1}_R}$. We have \[ \sum_j F \circ \delta_{ij}  \circ \sigma (g_j^p) = \sum_j F \circ \delta_{ij} (g_j) = F \left(\sum_j \delta_{ij} (g_j) \right) = f_i^p,\]
	and so $D^{(e+1)}_R\cdot \mathfrak{a}^{[p]} = D^{(e+1)}_R\cdot \mathfrak{b}^{[p]}$.
	
	Conversely, if $D^{(e+1)}_R \cdot \mathfrak{a}^{[p]} = D^{(e+1)}_R \cdot \mathfrak{b}^{[p]}$, then  $f_1^p,\dots,f_t^p\in D^{(e+1)}_R \cdot \mathfrak{b}^{[p]}$. Thus,  $f_i^p=\sum_j \delta_{ij}(g_j^p)$ for some $\delta_{ij}\in D^{(e+1)}_R$. Then $\sigma \circ \delta_{ij} \circ F \in D^{(e)}_R$. We have \[ \sum_j \sigma \circ \delta_{ij}  \circ F (g_j) = \sum_j \sigma \circ \delta_{ij} (g_j^p) = \sigma \left(\sum_j \delta_{ij} (g_j^p) \right) = \sigma(f_i^p) = f_i,\]
	and so $D^{(e)}_R \cdot \mathfrak{a} = D^{(e)}_R\cdot \mathfrak{b}$.
\end{proof}

\begin{proposition} \label{prop:higher-jumps}
	Let $R$ be an $F$-finite and $F$-split ring, $\fa$ be an ideal {with  $r$ generators,} and $m \geq n \geq 0$ be integers. Then for all integers $e, a \geq 0$ we have:
	\begin{enumerate}
		\item If $n \in \cB^\bullet_\fa(p^e)$, then $[n p^a, n p^a + r(p^a - 1)] \cap \cB^\bullet_\fa(p^{e + a}) \neq \varnothing$.
		
		\item If $[n - r + 1, m - 1] \cap \cB^\bullet_\fa(p^e) = \varnothing$, then $[np^a - r + 1, m p^a  - 1] \cap \cB^\bullet_\fa(p^{e+a}) = \varnothing$.
	\end{enumerate}
\end{proposition}

\begin{proof}
	For part (i), we consider the chain of ideals
	$$\fa^{np^a} \supseteq \fa^{n[p^a]} \supseteq \fa^{(n+1)[p^a]} \supseteq \fa^{(n+1)p^a + (r-1) (p^a - 1)} = \fa^{np^a + r(p^a - 1) + 1}.$$
	(see Lemma \ref{lem:frob-powers}). Acting with $D^{(e + a)}_R$, we obtain the chain
	$$ D^{(e + a)}_R \cdot \fa^{np^a} \supseteq  D^{(e + a)}_R \cdot \fa^{n[p^a]} \supseteq D^{(e + a)}_R \cdot  \fa^{(n+1)[p^a]} \supseteq D^{(e + a)}_R \cdot \fa^{np^a + r(p^a - 1) + 1}.$$
	By Lemma \ref{LemmaEqualityFrobPowers}, the two ideals in the middle differ, so the two outer ones must also differ. The statement then follows from Lemma \ref{lem:chain-jumps}. 
	
	Part (ii) follows similarly: we consider the chain
	$$\fa^{(n - r + 1)[p^a]} \supseteq \fa^{(n - r + 1)p^a + (r-1) (p^a - 1)} = \fa^{n p^a - r + 1} \supseteq \fa^{m p^a} \supseteq \fa^{m [p^a]},$$
	which gives
	$$ D^{(e + a)}_R \cdot \fa^{(n - r + 1)[p^a]} \supseteq D^{(e + a)}_R \cdot \fa^{n p^a - r + 1} \supseteq D^{(e + a)}_R \cdot \fa^{m p^a} \supseteq D^{(e + a)}_R \cdot \fa^{m [p^a]}.$$
	Lemma \ref{lem:chain-jumps} and Lemma \ref{LemmaEqualityFrobPowers} gives that the two outer ideals are equal, and hence the two in the middle must also be equal. Another application of Lemma \ref{lem:chain-jumps} gives the statement. 		
\end{proof}

%%%%%%%%%%%%%%%%%%%%%%%%%%%%%%%%%%%%%%%%%%%%%%%%%%%%%%%%%%%
\section{Bernstein-Sato roots} \label{Sec-BSR}
%%%%%%%%%%%%%%%%%%%%%%%%%%%%%%%%%%%%%%%%%%%%%%%%%%%%%%%%%%%%

We now begin the study of the first invariant of real interest: the Bernstein-Sato roots of an ideal. These provide a characteristic-$p$ analogue of the roots of the Bernstein-Sato polynomial, although the definition we give below provides no indication of why that would be the case; such an explanation is given later in Section~\ref{scn-BSroots-via-V}. However, using the definition below has many advantages: it does not involve any technicalities beyond those of differential jumps, it provides an easier way of computing Bernstein-Sato roots and it also is more useful for proving statements. 

\subsection{Definition and basic properties} \ 

\begin{definition}\label{DefBSrootsPadic}
	Let $R$ be $F$-finite, and $\fa$ be an ideal. We say that $\alpha\in \Zp$ is a Bernstein-Sato root of $\fa$ if there is a sequence $(\nu_e)_{e = 0}^\infty$ with $\nu_e \in \cB_\fa^\bullet(p^e)$ such that $\alpha$ is the $p$-adic limit of $\nu_e$. We denote by $\BSR(\fa)$ the set of Bernstein-Sato roots of $\fa$. 
\end{definition}

We recall that a sequence of $p$-adic numbers $(\nu_e)$ converges to a $p$-adic number $\alpha$ if and only if for every $m\in \Z_{\geq 0}$ there is some $N\in \Z_{\geq 0}$ such that $p^m \ | \ (\alpha - \nu_e)$ for all $e\geq N$.

\begin{remark} \label{rmk-diffjump-subsequence}
	Note that, given a sequence $(\nu_e)_{e= 0}^\infty \sq \Z_{\geq 0}$, the condition that $\nu_e \in \cB^\bullet_\fa(p^e)$ for every $e \geq 0$ passes to subsequences. Indeed, if $(\nu_{e_i})$ is a subsequence, then $\nu_{e_i} \in \cB^\bullet_\fa(p^{e_i}) \sq \cB^\bullet(p^i)$ for every $i \geq 0$. 
\end{remark}

\begin{proposition}\label{eqroot}
	Let $R$ be $F$-finite, $\fa$ be an ideal with $r$ generators and {$\alpha\in\Zp$} be a $p$-adic integer. The following are equivalent:
	\begin{enumerate}[(a)]
		\item\label{it:a1} $\alpha$ is a Bernstein-Sato root of $\fa$.
	\item\label{it:a2} For all $e \geq 0$ there is some $s_e \in \{0, \ds, r-1\}$ such that $\alpha_{<e} + s_e p^e \in \cB^\bullet_\fa(p^e)$.
	\item\label{it:a3} There is an infinite subset $\{e_j\} \subseteq \NN$ and differential jumps $\nu_j\in \cB^{\bullet}_{\fa}(p^{e_j})$ such that $(\nu_j)$ converges to~$\alpha$.
	\end{enumerate}
\end{proposition}

\begin{proof}
	For \ref{it:a1} implies \ref{it:a2}, let $\alpha$ be a Bernstein-Sato root of $\fa$, and $\nu_e\in \cB_{\fa}^\bullet(p^e)$ such that $\alpha=\lim \nu_e$. For every $a$ there is some $e_a$ such that $p^a | (\alpha-\nu_{j})$ for all $j\geq e_a$; without loss of generality, we can take $e_a\geq a$. Consider the sequence $\eta_a = \nu_{e_a}$. By Lemma~\ref{lem:De-nested}, $\eta_a\in \cB_{\fa}^\bullet(p^a)$, and by construction, $p^a | (\alpha-\nu_{a})$ for all $a$. Then, by Proposition~\ref{prop:subtract-pe}, we may subtract a multiple of $p^a$ from $\eta_a$ to obtain another sequence $\mu_a$ in which $0\leq \mu_a < r p^a$, and $p^a | (\alpha-\mu_{a})$ for all $a$. It follows that $\mu_a = \aaa + s_a p^a \in \cB_{\fa}^\bullet(p^a)$ with $s_a\in\{0,\dots,r-1\}$ as required.
	
	The implication \ref{it:a2} implies \ref{it:a3} is clear.

 For \ref{it:a3} implies \ref{it:a1}, it suffices to see that given a $p$-adically convergent sequence of the form $\nu_{e_a} \in \cB_{\fa}^\bullet(p^{e_a})$ for $e_a$ an infinite increasing sequence of integers that we can extend this to a sequence $\nu_{e}\in \cB_{\fa}^\bullet(p^{e})$ for all $e\in \NN$. This follows from Lemma~\ref{lem:De-nested}.
\end{proof}

 \begin{remark}
  It follows  from the definition and from Remark~\ref{rem:existjumps} that if $\fa=R$, then $\BSR(\fa)=\varnothing$. On the other hand, if $\fa\subsetneqq R$ is a proper ideal, then by Remark~\ref{rem:existjumps} there is a differential jump of level $e$ for every $e$, and by compactness of $\Zp$ condition~\ref{it:a3} of Proposition~\ref{eqroot} holds for some $\alpha$, so $\BSR(\fa)\neq \varnothing$.\end{remark}

In Section~\ref{scn-BSroots-via-V} we show that whether a $p$-adic integer is a Bernstein-Sato root or not is given in terms of the nonvanishing of a certain module, whose construction is compatible with localization. From this it follows that Bernstein-Sato roots are local invariants; however, we give a proof of this fact here that does not require the material of Section~\ref{scn-BSroots-via-V}.

To begin, let $\fa \sq R$ be an ideal and fix an integer $r \geq 0$ such that $\fa$ can be generated by $r$ elements. Given a positive integer $e \geq 0$ and a $p$-adic integer $\alpha$ we  denote by $\cJ_\fa(p^e, \alpha)$ the set
$$\cJ_\fa(p^e, \alpha) : = \{\alpha_{< e} + s p^e \ | \ s = 0, 1, \ds, r-1\} \cap \cB^{\bullet}_\fa(p^e).$$
\begin{lemma} \label{lemma-BSR-are-local-1}
The $p$-adic integer $\alpha$ is not a Bernstein-Sato root of $\fa$ if and only if there is some $e$ large enough so that $\cJ_\fa(p^e, \alpha) = \varnothing$, in which case $\cJ_\fa(p^a, \alpha) = \varnothing$ for all $a \geq e$.
\end{lemma}
\begin{proof}
	The first statement follows from the equivalence of \ref{it:a1} and \ref{it:a2} in Proposition \ref{eqroot}. For the second statement, it is enough to show that whenever $\cJ_\fa(p^e, \alpha) = \varnothing$ then $\cJ_\fa(p^{e+1}, \alpha) = \varnothing$, which we prove by contradiction. 
{Suppose that $\fa$ is generated by $r$ elements.}	
	If we had some $n \in \cJ_\fa(p^{e+1}, \alpha)$, then we have $n \in \cB^\bullet_\fa(p^{e+1})$, and thus $n \in \cB^\bullet_\fa(p^e)$ (see Lemma \ref{lem:De-nested}). From Proposition \ref{prop:subtract-pe} we conclude that there is some integer $k \geq 0$ such that $n - k p^e \in \cB^\bullet_\fa(p^e)$ and $0 \leq n - k p^e < r p^e$. Since $n \equiv \alpha \mod p^{e+1}$, we have $n - kp^e \equiv \alpha \mod p^e$, and therefore $n - kp^e = \alpha_{< e} + sp^e$ for some $s \in \{0, 1, \ds, r-1\}$. We conclude that $n - kp^e \in \cJ_\fa(p^e, \alpha)$, giving the desired contradiction. 
\end{proof}

\begin{lemma} \label{lemma-BSR-are-local-2}
	Suppose $g_1, \ds, g_k \in R$ are such that $(g_1, \ds, g_k) = (1)$. For a fixed $p$-adic integer $\alpha \in \Zp$ and integer $e \geq 0$ the following are equivalent:
	\begin{enumerate}[(a)]
		\item We have $\cJ_\fa(p^e,\alpha) = \varnothing$.
		\item We have $\cJ_{\fa R_{g_i}} (p^e,\alpha) = \varnothing$ for all $i = 1, \ds, k$.
		\item We have $\cJ_{\fa R_\fm}(p^e,\alpha) = \varnothing$ for all maximal ideals $\fm \sq R$. 
	\end{enumerate}
\end{lemma}
\begin{proof}
This 	follows from Remark \ref{rem:reduce-to-local}.
\end{proof}

\begin{lemma} \label{lemma-BSR-are-local-3}
	Let $\fm \sq R$ be a maximal ideal, $\alpha \in \Zp$ be a $p$-adic integer and $e \geq 0$ be an integer. If $\cJ_{\fa R_\fm}(p^e) = \varnothing$, then there exists some $g \in R \setminus \fm$ such that $\cJ_{\fa R_g} (p^e) = \varnothing$. 
\end{lemma}
\begin{proof}
	Let $S$ denote the set $S: = \{\alpha_{<e} + sp^e : s = 0, 1, \ds, r-1\}$. By Remark \ref{rem:reduce-to-local}, $\cJ_{\fa R_\fm}(p^e) = \varnothing$ precisely when $(D^{(e)}_R \cdot \fa^n ) R_\fm = (D^{(e)}_R \cdot \fa^{n+1}) R_\fm$ for all $n \in S$; that is, whenever $\fm$ is not in the support of the module $\bigoplus_{n \in S} D^{(e)}_R \cdot \fa^n / D^{(e)}_R \cdot \fa^{n+1}$. The result then follows from the fact that the support of a finitely-generated module is closed. 
\end{proof}

\begin{proposition} \label{prop-BSR-are-local}
	Let $R$ be a noetherian $F$-finite ring and $\fa \sq R$ be an ideal. Let $g_1, \ds, g_k \in R$ be such that $(g_1, \ds, g_k) = (1)$ and let $\Max(R)$ denote the set of all maximal ideals of $R$. We then have:
	$$\BSR(\fa) = \bigcup_{i = 1}^k \BSR(\fa R_{g_i}) = \bigcup_{\fm \in \Max(R)} \BSR(\fa R_\fm).$$
\end{proposition}

\begin{proof}
	Lemmas \ref{lemma-BSR-are-local-1} and \ref{lemma-BSR-are-local-2} give the first equality, and show that $\BSR(\fa) \supseteq \bigcup_{\fm} \BSR(\fa R_\fm)$. To prove that $\BSR(\fa) \sq \BSR(\fa R_\fm)$, suppose that $\alpha \in \Zp$ is such that $\alpha \notin \BSR(\fa R_\fm)$ for all maximal ideals $\fm$. By Lemma \ref{lemma-BSR-are-local-1} and Lemma \ref{lemma-BSR-are-local-3} we conclude that for every $\fm \in \Max(R)$ there is some integer $e_\fm$ and some element $g_\fm \in R \setminus \fm$ such that $\cJ_{R_{g_\fm}}(p^{e_\fm}, \alpha) = \varnothing$. The elements $(g_\fm | \fm \in \Max(R))$ generate the unit ideal and therefore there is a finite subcollection of them, say $g_1 = g_{\fm_1}, \ds, g_k = g_{\fm_k}$, that still generate the unit ideal. If  $e = \max \{ e_{\fm_1}, \ds, e_{\fm_k} \}$,  then, $ \cJ_{\fa R_{g_i}}(p^e, \alpha) = \varnothing $ for all $i = 1, \ds, k$ by Lemma \ref{lemma-BSR-are-local-1}. Therefore $\cJ_{\fa}(p^e, \alpha) = \varnothing$ by Lemma \ref{lemma-BSR-are-local-2}. We conclude that $\alpha \notin \BSR(\fa)$ by Lemma \ref{lemma-BSR-are-local-1}. 
\end{proof}

\begin{proposition} \label{prop-BSR-homog-and-completion}
Let $R$ be a noetherian $F$-finite ring and $\fa\subseteq R$ be an ideal.
\begin{enumerate}
\item If $R$ is positively graded with homogeneous maximal ideal $\fm$ and $\fa$ is homogeneous, then $\BSR(\fa)=\BSR(\fa R_{\fm})$.
\item If $R$ is local, then $\BSR(\fa)=\BSR(\fa \widehat{R})$.
\end{enumerate}
\end{proposition}
\begin{proof}
Both facts follow from Remark~\ref{rem:reduce-to-local}.
\end{proof}

\subsection{Finiteness and rationality results}  \ 

We now introduce a finiteness condition that has important consequences for Bernstein-Sato roots.

\begin{definition} \label{def-BSadm}
	Let $R$ be an $F$-finite ring and $\fa \sq R$ be an ideal generated by $r$ elements. We say that $\fa$ is Bernstein-Sato admissible if there is a constant $C > 0$ such that
	$$\# \big( \cB^\bullet_\fa(p^e) \cap [0, rp^e) \big) \leq C$$
	for every $e\in \ZZ_{\geq 0}$. We say that $R$ is a Bernstein-Sato admissible ring if all of its ideals are Bernstein-Sato admissible. 
\end{definition}

In our definition of Bernstein-Sato root we only require that the number of differential jumps of level $e$ in the interval $[0, rp^e)$ is bounded, but the subtraction property of differential jumps (Proposition \ref{prop:subtract-pe}) gives a stronger statement.

\begin{proposition} \label{prop-BSadm-iff-linearGrowth}
	Let $R$ be an $F$-finite ring and $\fa \sq R$ be an ideal. The ideal $\fa$ is Bernstein-Sato admissible if and only if there are constants $A, B > 0$ such that for all integers $e, s \geq 0$ we have
	$$\# \bigg( \cB^\bullet_\fa(p^e) \cap [0, s) \bigg) \leq A \frac{s}{p^e} + B.$$
\end{proposition}
\begin{proof}
	Suppose that $\fa$ is generated by $r$ elements. We note that if
	 $\fa$ is Bernstein-Sato admissible, then 
	 there exists 
 $A, B > 0$ such that for all $e, s \geq 0$ we have
$\# \bigg( \cB^\bullet_\fa(p^e) \cap [0, s) \bigg) \leq A \frac{s}{p^e} + B$  by setting $s = rp^e$. To prove the converse statement, let $C > 0$ be a constant as in Definition \ref{def-BSadm}. We observe that for all integers $k \geq 1$ we have $\# \big(\cB^\bullet_\fa(p^e) \cap [(k-1)p^e, kp^e) \big) \leq C$: this follows when $1 \leq k \leq r$, and for $k > r$ it follows from Proposition \ref{prop:subtract-pe}. We conclude:
	\begin{align*}
	\# \bigg( \cB^\bullet_\fa(p^e) \cap [0, s) \bigg) &  \leq   \# \bigg( \bigcup_{k = 1}^{\lceil s/p^e \rceil} \cB^\bullet_\fa(p^e) \cap [(k-1)p^e, kp^e) \bigg) \\
		& \leq \left\lceil \frac{s}{p^e} \right\rceil C \\
		& \leq \left(\frac{s}{p^e} + 1 \right) C = C \frac{s}{p^e} + C. \qedhere
	\end{align*}
\end{proof}

\begin{corollary}
	Whether $\fa$ is Bernstein-Sato admissible or not does not depend on the choice of $r$.
\end{corollary}

\begin{theorem} \label{thm-finiteBSRoots}
	Let $R$ be an $F$-finite ring and $\fa \sq R$ a Bernstein-Sato admissible ideal. Then $\fa$ has finitely many Bernstein-Sato roots.
\end{theorem}
\begin{proof} 
	Pick an integer $C >0$ such that $\#\big( \cB^\bullet_\fa(p^e) \cap [0, rp^e) \big) \leq C$ for all $e \geq 0$. We claim that there are at most $C$ Bernstein-Sato roots, and we prove it by contradiction. Suppose that $\{\alpha_1, \ds, \alpha_{C+1}\}$ are distinct Bernstein-Sato roots of $\fa$, and choose $N$ large enough so that $\alpha_i \not \equiv \alpha_j \mod p^N$ for all $i \neq j$. By Proposition \ref{eqroot} we know there is some $e$ large enough and $\nu_1, \ds \nu_{C+1} \in \cB^\bullet_\fa(p^e)$, with $0 \leq \nu_i < r p^e$, such that $\nu_i \equiv \alpha_i \mod p^N$, and therefore $\nu_1, \ds, \nu_{C+1}$ are distinct differential jumps. This gives the desired contradiction.
\end{proof}

\begin{lemma} \label{lemma-BSroot-dynamics}
	Let $R$ be $F$-split, $\fa \sq R$ be an ideal generated by $r$ elements, $\alpha \in \Zp$ be a Bernstein-Sato root of $\fa$ and $a \geq 0$ be an integer. There exists some $i \in \{0, 1, \ds, r(p^a - 1)\}$ such that $p^a \alpha + i$ is a Bernstein-Sato root of $\fa$. 
\end{lemma}
\begin{proof}
	Pick a sequence $(\nu_e)$ such that $\nu_e \in \cB^\bullet_\fa(p^e)$ whose $p$-adic limit is $\alpha$. By Proposition~\ref{prop:higher-jumps}, for every $e$ there is some $i_e \in \{0, 1, \ds, r(p^a-1)\}$ such that $p^a \nu_e + i_e \in \cB^\bullet_\fa(p^{e+a})$. Since $\{0, 1, \ds, r(p^a - 1)\}$ is a finite set, there is some $i \in \{0, 1, \ds, r(p^a - 1)\}$ and an increasing sequence $(e_j)$ such that $p^a \nu_{e_j} + i \in \cB^\bullet_\fa(p^{e_j+a})$. The $p$-adic limit of $p^a \nu_{e_j} + i$ is $p^a \alpha + i$, and the result follows from Proposition \ref{eqroot}. 
\end{proof}

We recall that an ideal $J$ is a reduction of an ideal $I$ with reduction number $n$ if $JI^n=I^{n+1}$ and $JI^{n-1} \neq JI^n$.

\begin{lemma} \label{lemma-reduction}
	Let $R$ be an $F$-finite local ring, $\fa \sq R$ be an ideal and $\fb \sq \fa$ be a reduction of $\fa$ with reduction number $k$. Then:
	\begin{enumerate}[(i)]
		\item For all integers $e \geq 0$ we have $\cB^\bullet_\fa(p^e) \sq \bigcup_{i = 0}^k \cB^\bullet_\fb (p^e) + i$ and $\cB^\bullet_\fb(p^e) \sq \bigcup_{i = 0}^k \cB^\bullet_\fa(p^e) - i$. 
		\item The ideal $\fa$ is Bernstein-Sato admissible if and only if $\fb$ is Bernstein-Sato admissible.
		\item We have $\BSR(\fa) \sq \bigcup_{i = 0}^k \BSR(\fb) + i$ and $\BSR(\fb) \sq \bigcup_{i = 0}^k \BSR(\fa) - i$. 
	\end{enumerate}
\end{lemma}

\begin{proof}
	Part (i) follows by considering the following chains of ideals:
	$$D^{(e)}_R \cdot \fb^{n-k} \supseteq D^{(e)}_R \cdot  \fa^n \supseteq D^{(e)}_R \cdot \fa^{n+1} \supseteq D^{(e)}_R \cdot \fb^{n+1},$$
	$$D^{(e)}_R \cdot \fa^n \supseteq D^{(e)}_R \cdot  \fb^n \supseteq D^{(e)}_R \cdot  \fb^{n+1} \supseteq D^{(e)}_R \cdot \fa^{n + k + 1},$$
	and applying Lemma \ref{lem:chain-jumps}. 
	
	Let us now prove part (ii); we use the alternative characterization of Bernstein-Sato admissibility given in Proposition \ref{prop-BSadm-iff-linearGrowth}. Suppose that $\fa$ is Bernstein-Sato admissible and pick constants $A_\fa$, $B_\fa$ such that $\# \big( \cB^\bullet_\fa(p^e) \cap [0,s) \big) \leq A_\fa (s/p^e) + B_\fa$ for all $e, s \geq 0$. By applying part (i) we conclude that
	\begin{align*}
	\# \bigg( \cB^\bullet_\fb(p^e) \cap [0,s) \bigg) & \leq \# \bigg( \bigcup_{i = 0}^k \cB^\bullet_\fa(p^e) \cap [0, s + k) - i \bigg) \\
	& \leq k \left(A_\fa \frac{s + k}{p^e} + B_\fa \right) \\
	& \leq k A_\fa \frac{s}{p^e} + k^2 A_\fa + B_\fa.
	\end{align*}
	For the other direction, suppose $\fb$ is Bernstein-Sato admissible and choose constants $A_\fb$, $B_\fb$ similarly. Then
	\begin{align*}
	\# \bigg( \cB^\bullet_\fa(p^e) \cap [0,s) \bigg) & \leq \# \bigg( \bigcup_{i = 0}^k \cB^\bullet_\fb(p^e) \cap [0, s) + i \bigg) \\
		& \leq k \left( A_\fb \frac{s}{p^e} + B_\fb \right) \\
		& = kA_\fb \frac{s}{p^e} + k B_\fb.
	\end{align*}
	
	We now tackle part (iii). Suppose that $\alpha \in \BSR(\fa)$, and choose a sequence $(\nu_e)$ with $\nu_e \in \cB^\bullet_\fa(p^e)$ such that $\alpha$ is the $p$-adic limit of $\nu_e$. By part (i), for every $e \geq 0$ there is some $i_e \in \{0, 1, \ds, k\}$ such that $\nu_e - i_e \in \cB^\bullet_\fb(p^e)$. We conclude there is some $i \in \{0, 1, \ds, k\}$ and a subsequence $(\nu_{e_j})$ such that $\nu_{e_j} - i \in \cB^\bullet_\fb(p^{e_j})$. The $p$-adic limit of this subsequence is $\alpha - i$ which, by Remark \ref{rmk-diffjump-subsequence}, is a Bernstein-Sato root of $\fb$. Therefore, $\alpha \in \BSR(\fb) + i$. The other statement follows similarly.
\end{proof}

\begin{theorem}\label{thm:ratl}
	Let $R$ be an $F$-finite $F$-split ring. Let $\fa$ be a Bernstein-Sato admissible ideal. Then every Bernstein-Sato root of $\fa$ is rational.
\end{theorem}
\begin{proof}
	Recall that $\BSR(\fa)$ denotes the set of Bernstein-Sato roots of $\fa$, and let $\t{\BSR}(\fa) \sq \Zp / \Z$ be its image under the quotient map $\Zp \to \Zp/ \Z$; in other words, $\t{\BSR}(\fa) = \{\alpha + \Z  \ | \ \alpha \in \BSR(\fa) \}.$ Note that, by Lemma~\ref{lemma-BSroot-dynamics}, $\t{\BSR}(\fa)$ is closed under multiplication by~$p$. 
	
	Let $\alpha \in \Zp$ be a Bernstein-Sato root. Since $\t{\BSR}(\fa)$ is a finite set, there exist some $n < m$ such that $p^n \alpha \equiv p^m \alpha \mod \Z$; that is, there exists some $c \in \Z$ such that $p^n \alpha = p^m \alpha + c$. It follows that $\alpha = c/(p^n (p^{m-n} - 1))$ and thus $\alpha$ is rational (note that, a posteriori, we know that $p^n$ must divide $c$).
\end{proof}

\begin{lemma}\label{lem:seq}
		Let $R$ be an $F$-finite $F$-split ring. Let $\fa$ be an $r$-generated ideal of $R$. Let $\alpha\in \ZZ_{(p)}$ be a Bernstein-Sato root of $\fa$.
		\begin{enumerate}
			\item If $\alpha >0$, then there exists an increasing sequence $\{n_j\} \subset \NN$ such that $\alpha +n_j$ is a Bernstein-Sato root of $\fa$ for each $j$.
			\item If $\alpha <-r$, then there exists an increasing sequence $\{n_j\} \subset \NN$ such that $\alpha -n_j$ is a Bernstein-Sato root of $\fa$ for each $j$.
				\end{enumerate}
\end{lemma}
\begin{proof}
	For the first part, it suffices to show that there is some positive integer $n$ such that $\alpha+n$ is a Bernstein-Sato root. We can write $\displaystyle \alpha=\frac{a}{1-p^e}+b$ with $a,b\in \NN$ such that $0\leq a < p^e-1$ and $b>0$. By Lemma~\ref{lemma-BSroot-dynamics}, there is some $i\in \{0,\dots,r(p^e-1)\}$ such that $p^e\alpha+i$ is a root. We have $p^e\alpha+i = \alpha + (p^e-1)b -a +i$. Since $a<p^e-1<(p^e-1)b$, the claim follows.
	
	Likewise, for the second part, it suffices to show that there is some negative integer $n$ such that $\alpha+n$ is a Bernstein-Sato root. We can write $\displaystyle \alpha=\frac{a}{p^e-1}-r-b$ with $a,b\in \NN$ such that $0\leq a < p^e-1$ and $b>0$. Then $p^e \alpha+i=-(p^e-1) r +a - (p^e-1)b + \alpha +i$ is a root for some $i\in \{0,\dots,r(p^e-1)\}$. We have $-(p^e-1) r +a - (p^e-1)b + \alpha +i \leq a - (p^e-1)b + \alpha \leq \alpha$, so we are done.
\end{proof}

\begin{theorem}\label{thm:rational}
	Let $R$ be an $F$-finite $F$-split ring. Let $\fa$ be a Bernstein-Sato admissible ideal with $r$ generators. Then every Bernstein-Sato root of $R$ lies in the interval $[-r,0]$.
\end{theorem}

\begin{proof}
	 Since $\fa$ is Bernstein-Sato admissible, the set of roots is finite. The bounds on the roots then follow from Lemma~\ref{lem:seq}.
\end{proof}

\begin{corollary}\label{cor:ell}
	Let $R$ be a local $F$-finite $F$-split ring. Let $\fa$ be a Bernstein-Sato admissible ideal with analytic spread $\ell$. Then every Bernstein-Sato root of $R$ lies in the interval $[-\ell,0]$.
\end{corollary}

\begin{proof}
	By Remark~\ref{rem:reduce-to-local} and Lemma~\ref{lem:redu-inf-resfield}, the statement reduces to the case where $R$ is local with infinite residue field. In this case, there exists a reduction of $\fa$ that is generated by at most $\ell$ elements. The result then follows from Theorem~\ref{thm:rational} and Lemma \ref{lemma-reduction}. 
\end{proof}

\begin{corollary} \label{cor-monomial-char0}
	Let $R = \CC[x_1, \ds, x_n]$ be a polynomial ring over $\CC$, $\fm$ be the maximal ideal $\fm = (x_1, \ds, x_n)$ and $\fa \sq \CC[x_1, \ds, x_n]$ be a monomial ideal. If $\lambda \in \Q$ is a root of the Bernstein-Sato polynomial of $\fa$ then $- \ell(\fa R_\fm) \leq \lambda$.
\end{corollary}
\begin{proof}
	Pick a large prime $p$ and let $\bar R = \F_p[x_1, \ds, x_n]$ be a polynomial ring over $\F_p$, $\bar \fm \sq \bar R$ denote the maximal ideal $\bar \fm = (x_1, \ds, x_n)$ and $\bar \fa \sq \bar R$ denote the mod-$p$ reduction of $\fa$. 
	
	We may pick $p$ large enough so that $\lambda$ is a Bernstein-Sato root of $\bar \fa$ \cite[Theorem 3.1]{QG19b}. Moreover, since the construction of the fibre cone of $\fa$ with respect to $\fm$ is compatible with mod-$p$ reduction, we can further enlarge $p$ to assume\footnote{In fact, by a result of Singla, the analytic spread of a monomial ideal $\fa \sq \KK[x_1, \ds, x_n]_{(x_1, \ds, x_n)}$ depends only on the Newton polytope of $\fa$ \cite[Cor. 4.10]{Singla07} (see also \cite{Bivia03}), and therefore $\ell(\fa R_\fm) = \ell(\bar \fa \bar R_{\bar \fm})$ for any $p$.}
 that $\ell(\fa R_\fm) = \ell(\bar \fa \bar R_{\bar \fm})$.	Since $\fa$ is homogeneous, we conclude that $\lambda$ is a Bernstein-Sato root of $\bar \fa \bar R_{\bar \fm}$ (Proposition \ref{prop-BSR-homog-and-completion}), and we conclude that $- \ell(\bar \fa \bar R_{\bar \fm}) \leq \lambda$ from Corollary \ref{cor:ell}.
\end{proof}

In characteristic zero, whenever $\fa \sq \CC[x_1, \ds, x_n]$ is a nonzero ideal, all the roots of the Bernstein-Sato polynomial of $\fa$ are strictly negative. Since we have only shown that the Bernstein-Sato roots are nonpositive, the question of whether zero can be a Bernstein-Sato root arises. We can answer it for principal ideals as follows.

\begin{proposition}\label{prop:d-simple}
	Let $R$ be an $F$-finite ring. The following are equivalent:
	\begin{enumerate}[(a)]
		\item The ring $R$ is simple as a $D_R$-module.
		\item For all nonzero $f \in R$ we have $0 \notin \BSR(f)$. 
	\end{enumerate}
	Moreover, if these hold then the nilradical $\sqrt 0$ of $R$ is a prime ideal. In particular, if $R$ is reduced then it must be a domain.
\end{proposition}
\begin{proof}
	Suppose that $R$ is simple as a $D_R$-module. Given a nonzero $f \in R$, we have $D_R \cdot f = R$ and therefore there is some $e$ large enough so that $D^{(e)}_R \cdot f = R = D^{(e)}_R \cdot f^0$. We conclude that $0 \notin \cB^\bullet_f(p^e)$ and, by Proposition \ref{eqroot}, we conclude that $0 \notin \BSR(f)$. 
	
	Conversely, suppose that $0 \in \BSR(f)$ for some nonzero $f \in R$. By Proposition \ref{eqroot}, for all $e \geq 0$ we have $0 \in \cB^\bullet_f(p^e)$ and thus $D^{(e)}_R \cdot f \neq D^{(e)}_R \cdot f^0 = R$. We conclude that $D_R \cdot f \neq R$, and hence $R$ is not simple as a $D_R$-module.
	
	For the last statement, suppose that $f, g \in R$ are such that $fg \in \sqrt 0$ and $f \notin \sqrt 0$. Then the collection $H^0_g(R)$ of $g$-torsion elements of $R$ is a $D_R$-submodule of $R$ which contains a power of $f$, and is therefore nonzero. If $R$ is $D_R$-simple, we conclude that $1 \in H^0_g(R)$ and thus some power of $g$ is zero, i.e. $g \in \sqrt 0$. 
\end{proof}

\begin{question}
	Let $R$ be an $F$-finite ring that is simple as a $D_R$-module. Do we have $0 \notin \BSR(\fa)$ for all nonzero ideals $\fa \sq R$?
\end{question}

	%%%%%%%%%%%%%%%%%%%%%%%%%%%%%%%%%%%%%%%%%%%%%%%%%%%%%%%%%%%%%%%%%%%%
\section{Differential thresholds} \label{SubSecDT}
%%%%%%%%%%%%%%%%%%%%%%%%%%%%%%%%%%%%%%%%%%%%%%%%%%%%%%%%%%%%%%%%%%%%

In this section we introduce the other key numerical invariant of this paper: differential thresholds.  These are related to $F$-jumping numbers, $F$-thresholds, Cartier thresholds, and Bernstein-Sato roots. 

\subsection{Definition and basic properties} \

\begin{definition}
	Let $R$ be an $F$-finite  ring, and $\fa\subseteq R$ an ideal.
	We say that $\lambda\in\RR_{\geq 0}$ is a differential threshold if there exists a sequence of elements
	$\nu_e\in \cB^\bullet_\fa(p^e)$ such that $\lambda=\lim\limits_{e\to\infty}\frac{\nu_e}{p^e}$, where the limit is taken in the usual Euclidean topology.
\end{definition}

When $R$ is $F$-split, it turns out that every differential threshold can be realized as a limit in a nice way.

\begin{proposition}\label{prop:equivsthresh}
	Let $R$ be an $F$-finite $F$-split ring and $\lambda\in \RR_{> 0}$. Let $\fa$ be an $r$ generated ideal. The following are equivalent:
	\begin{enumerate}[(a)]
		\item\label{it:l1} $\lambda$ is a differential threshold of $\fa$.
		\item\label{it:l2} For all $e\geq 0$, there is a differential jump of level $e$ for $\fa$ in the interval $[p^e \lambda-r, p^e \lambda]$.
		\item\label{it:l3} There is an infinite set $\{e_j\} \subseteq \ZZ_{>0}$ and differential jumps $\nu_j\in \cB_{\fa}^\bullet(p^{e_j})$ such that $(\nu_j/p^{e_j})$ converges to $\lambda$.
	\end{enumerate}
\end{proposition}

\begin{proof}
	We start by showing that \ref{it:l1} implies \ref{it:l2} by contraposition. 
Suppose that ${[p^e \lambda - r, p^e \lambda] }\cap \cB^\bullet_\fa(p^e) = \varnothing$ for some $e \geq 0$. Since every differential jump is an integer, we get that $[\lceil p^e \lambda \rceil - r, \lfloor p^e \lambda \rfloor ] \cap \cB^\bullet_\fa(p^e) = \varnothing$. By Proposition \ref{prop:higher-jumps} (ii) we conclude that, for all integers $a \geq 0$, 
	$$\big[p^a \lceil p^e \lambda \rceil - p^a - r + 1, \  p^a \lfloor p^e \lambda \rfloor + p^a \big) \cap \cB^\bullet_\fa(p^{e+a}) = \varnothing$$
	and therefore 
	$$\bigg[\frac{\lceil p^e \lambda \rceil - 1}{p^e} - \frac{r-1}{p^{e+a}}, \frac{\lfloor p^e \lambda \rfloor + 1}{p^e}\bigg) \cap \frac{1}{p^{e+a}} \cB^\bullet_\fa(p^{e+a}) = \varnothing,$$
	and thus 
	$$\bigg(\frac{\lceil p^e \lambda \rceil - 1}{p^e} , \frac{\lfloor p^e \lambda \rfloor + 1}{p^e} \bigg) \cap \frac{1}{p^{e+a}} \cB^\bullet_\fa(p^{e+a}) = \varnothing.$$
	By considering the cases $p^e \lambda \in \Z$ and $p^e \lambda \notin \Z$ separately, we observe that 
	$$\lambda \in \bigg(\frac{\lceil p^e \lambda \rceil - 1}{p^e} , \frac{\lfloor p^e \lambda \rfloor + 1}{p^e} \bigg)$$
	and therefore $\lambda$ cannot be a differential threshold.

	The implication \ref{it:l2} implies \ref{it:l3} is clear.
	
	To show that \ref{it:l3} implies \ref{it:l1}, let $(e_j)$ be an infinite increasing sequence of integers; we need to show that we can extend $(\nu_{e_j}/p^{e_j})$ to a convergent sequence $(\nu_i / p^i)$, $i\in \NN$. By Proposition~\ref{prop:higher-jumps}, for $\nu_{j}\in \cB_{\fa}^{\bullet}(p^{e_j})$ and for $a>e_j$ there is some $\nu_{a,j} \in \cB_{\fa}^{\bullet}(p^{a})\cap [\nu_{j}p^{a-{e_j}},(\nu_{j}+r)p^{a-{e_j}}]$. For $i=e_j$, take $\nu_i=\nu_{j}$, and for $e_j<i<e_{j+1}$, take $\nu_i = \nu_{i,j}$. If $a,b\geq e_j$, and $u,v$ are such that $e_u \leq a < e_{u+1}$ and $e_v \leq b < e_{v+1}$, then
	\[ \left| \frac{\nu_a}{p^a} - \frac{\nu_b}{p^b} \right| \leq \left|\frac{\nu_a}{p^a} - \frac{\nu_{u}}{p^{e_u}} \right| + \left|\frac{\nu_{u}}{p^{e_u}} - \frac{\nu_{v}}{p^{e_v}}\right| + \left| \frac{\nu_{b}}{p^{b}} - \frac{\nu_{v}}{p^{e_v}} \right| \leq \frac{2r}{p^{e_j}} + \left|\frac{\nu_{u}}{p^{e_u}} - \frac{\nu_{v}}{p^{e_v}}\right|.  \]
	Since $(\nu_{j}/p^{e^j})$ is a Cauchy sequence, the right-hand side in the previous equation tends to zero as $j\to \infty$, and hence $(\nu_{i}/p^{i})$ is Cauchy, as required.
\end{proof}

\begin{remark} It follows from definition and from Remark~\ref{rem:existjumps} that if $\fa=R$,  then $\fa$ has no differential jumps. Conversely, if $\fa\subsetneqq R$ is a proper ideal with $r$ generators, then by Remark~\ref{rem:existjumps} and Proposition~\ref{prop:subtract-pe}  there is a differential jump of level $e$ in the interval $[0,rp^e]$ for every $e$, so by compactness of $[0,r]$ there is a differential jump for $\fa$. \end{remark}

\begin{proposition}\label{prop:nojumpnothreshold}
	Let $R$ be a noetherian $F$-split ring of characteristic $p$, $\fa \sq R$ be an ideal generated by $r$ elements and $k,l \geq 0$ be integers with $l - k  \geq r-1$. If $[k, l) \cap \cB^\bullet_\fa(p^e) = \varnothing$ for some $e \geq 0$, then there are no differential thresholds of $\fa$ in the interval $\big( \ (k + r - 1)/ p^e, \ l / p^e \ \big)$.
\end{proposition}
\begin{proof}
	If $[k,l) \cap \cB^\bullet_\fa(p^e) = \varnothing$ then, by Proposition \ref{prop:higher-jumps}, we have
	$$\big [k p^a + (r-1)(p^a -1), l p^a \big) \cap \cB^\bullet_\fa(p^{e+a}) = \varnothing$$
	for all $a \geq e$, and thus
	$$\bigg[ \frac{k + r - 1}{p^e} - \frac{r-1}{p^{e+a}}, \frac{l}{p^e} \bigg) \cap \frac{1}{p^{e+a}} \cB^\bullet_\fa(p^{e+a}) = \varnothing.$$
	The result follows. 
\end{proof}

\begin{remark}\label{rem:zerothresh} If $R$ is $F$-split, and $\fa$ is an $r$-generated ideal, we have that $0$ is a differential threshold if and only if $0$ is differential jump of level $e$ for every $e$. If $0\notin \cB_{\fa}^\bullet(p^e)$, then $D_R^{(e)} \cdot \fa =R$, so $D_R^{(e+a)} \cdot \fa^{[p^a]} =R$, and hence $D_R^{(e+a)} \cdot \fa^{p^a} =R$ for all $a$, so there are no thresholds in the interval $[0,1/p^e)$. The other implication follows from the definition of differential threshold. 
	\end{remark}

If $R$ is $F$-split, we can also find differential thresholds that are close to differential jumps.

\begin{lemma}\label{lem:thresh-nearby}
	Let $R$ be an $F$-finite $F$-split ring, and $\fa$ be an ideal with at most $r$ generators. If $n\in \cB_{\fa}^\bullet(p^e)$, then there is a differential threshold $\lambda$ for $\fa$ in the interval $\displaystyle \left[\frac{n}{p^e},\frac{n + r}{p^{e}}\right]$.
\end{lemma}
\begin{proof}
	Let $\nu_e : = n$. Applying Proposition \ref{prop:higher-jumps} inductively we build a sequence $(\nu_a)_{a \geq e}$  with $\nu_a\in \cB_{\fa}^\bullet(p^a)$ such that $p \nu_a \leq \nu_{a + 1} \leq p \nu_a + r(p-1)$, and thus 
	$$\frac{\nu_a}{p^a} \leq \frac{\nu_{a+1}}{p^{a+1}} \leq \frac{\nu_a}{p^a} + \frac{r(p-1)}{p^{a+1}}.$$
	In particular, the sequence $(\nu_a/p^a)$ is increasing. We claim it is bounded by $(n + r )/p^e$; indeed, for all $b \geq 0$ we have
	\begin{align*}
	\frac{\nu_{e + b}}{p^{e+b}} & \leq \frac{n}{p^e} + \frac{r(p-1)}{p^{e+1}} + \cds + \frac{r(p-1)}{p^{e+b}} \\
	& \leq \frac{n}{p^e} + \frac{r(p-1)}{p^{e+1}} \bigg( 1+ \frac{1}{p} + \frac{1}{p^2} + \cds \bigg) \\
	& = \frac{n}{p^e} + \frac{r(p-1)}{p^{e+1}} \frac{1}{1 - \frac{1}{p}} \\
	& = \frac{n + r}{p^e}.
	\end{align*}
	Thus, the sequence $(\nu_a/p^a)$ converges to a value in the interval $[\frac{n}{p^e},\frac{n + r}{p^{e}}]$.
\end{proof}

\begin{remark}\label{RemDiffJumps}
	Let $R$ be an $F$-finite  ring, and $\fa\subseteq R$ an ideal. By Lemma~\ref{lem:thresh},
	we have that $\lambda$ is a differential threshold if and only if there exists a sequence of $D^{(e)}_R$-ideals $J_e$ such that 
	$\lambda=\lim\limits_{e\to \infty}\frac{ \max\{ n\in\NN \; | \; \fa^n\not\subseteq J_e\}}{p^e}$.
\end{remark}

We now provide several properties of differential thresholds. We first show that the set formed by them remains the same after taking integral closure. Then, we show a version of Skoda's Theorem  and a $p$-fractal property.

\begin{proposition}\label{PropDTIntgralC}
	Let $R$ be an $F$-finite  ring, and $\fa,\fb\subseteq R$ be ideals with the same integral closure.
	Then, $\fa$ and $\fb$ have the same differential thresholds.
\end{proposition}	
\begin{proof}

	Since $\fa$ and $\fb$ have the same integral closure, there exists an integer $a$ such that
	$\fa^n\subseteq \fb^{n+a} \subseteq \fa^{n+2a}$
	for every $n\in\NN$.
	Then, 
	$$\fa^n\subseteq \fb^{n+a} \subseteq \fa^{n+2a}\subseteq \fb^{n+3a} \subseteq \fa^{n+4a}.$$
	If $D^{(e)}_R \cdot \fa ^n=D^{(e)}_R \cdot \fa^{n+4a},$ then 
	$D^{(e)}_R \cdot \fb ^{n+a}=D^{(e)}_R \fb ^{n+3a}.$
	As a consequence, if
	$D^{(e)}_R \cdot \fb ^{n+a}\neq D^{(e)}_R \cdot  \fb ^{n+3a},$ then
	$D^{(e)}_R \cdot \fa ^n\neq D^{(e)}_R \cdot  \fa^{n+4a}.$
	
	Let $\lambda$ be a differential threshold of $\fb$, and $\nu_e$ differential jumps of level $e$ for $\fb$ such that $\lim\limits_{e\to \infty}\frac{\nu_e}{p^e}=\lambda$. 
	It suffices to show that $\lambda$ is  a differential threshold of $\fa$, as the role of $\fa$ and $\fb$ are interchangeable.
Since $\bigcap_{e\in \ZZ_{>0}} I_e(\fm)$ is a prime ideal \cite{AE}, we have that $\fa\subseteq \bigcap_{e\in \ZZ_{>0}} I_e(\fm)$ if and only if $\fb\subseteq \bigcap_{e\in \ZZ_{>0}} I_e(\fm).$
Then,  $\fpt(\fa)=0$ if an only if $\fpt(\fb)=0$. 
	We can assume that  $\lambda$ is positive, we have that $\nu_e>a$ for $e\gg 0$.
	We have that 
	$D^{(e)} \cdot \fb ^{\nu_e}\neq D^{(e)} \cdot \fb ^{\nu_e+2a}.$
	As a consequence, 
	$D^{(e)} \cdot \fa ^{\nu_e-a}\neq D^{(e)} \cdot \fa^{n+3a}$
	for $e\gg 0$.
	Then, there exists a differential jump of level $e$ for $\fa$, $w_e$ in 
	$\{ \nu_e-a, \nu_e-a+1,\ldots, \nu_e+3a\}$
	for $e\gg 0$.
	We have that
	$\lim\limits_{e\to \infty}\frac{w_e}{p^e}=\lambda$.
	Then, $\lambda$ is  a differential threshold of $\fa$.
\end{proof}

\begin{proposition}\label{PropSkoda}
	Let $R$ be an $F$-finite  ring, and $\fa\subseteq R$ an ideal generated by $r$ elements.
	If $ \lambda> r$ is a differential threshold of $\fa$, then  $\lambda-1$ is a differential threshold of $\fa$.
	If $\fa$ is principal generated by a nonzerodivisor, the converse is true.
\end{proposition}
\begin{proof}
	Let $\nu_e\in \cB^\bullet_\fa(p^e)$ be such that $\lim\limits_{e\to\infty}\frac{\nu_e}{p^e}=\lambda$.
	For $e\gg0$, we have $\nu_e> r p^e$ so, $\nu_e-p^e\in \cB^\bullet_\fa(p^e)$ by Proposition~\ref{prop:subtract-pe}.	Since $\lim\limits_{e\to\infty}\frac{\nu_e-p^e}{p^e}=\lambda-1$.  We conclude that $\lambda-1$ is a differential threshold.
	
	Likewise, if  $\fa=(f)$, where $f\in R$ is a nonzerodivisor,
	let $u_e\in \cB^\bullet_\fa(p^e) $ be such that $\lim\limits_{e\to\infty}\frac{u_e}{p^e}=\lambda-1$. By Proposition~\ref{prop:subtract-pe}, $u_e+p^e\in \cB^\bullet_f(p^e) $.
	Since $\lim\limits_{e\to\infty}\frac{u_e+p^e}{p^e}=\lambda$, we conclude that $\lambda$  
	is a differential threshold.
\end{proof}

\begin{corollary}\label{CorSkodaLocal}
	Let $(R,\fm,\KK)$ be a local  $F$-finite  ring,  $\fa\subseteq \fm$ an ideal, and $\ell$ be its analytic spread. 
	If $ \lambda> \ell$ is a differential threshold of $\fa$, then  $\lambda-1$ is a differential threshold of $\fa$.
\end{corollary}
\begin{proof}
	By Lemma \ref{lem:redu-inf-resfield}, we may assume that $\KK$ is infinite.
	Then, there exists an ideal  $\fb$ generated by $\ell$ elements with the same integral closure of $\fa$.
	Then,
	the result follows from Propositions \ref{PropDTIntgralC} and \ref{PropSkoda}.
\end{proof}

\subsection{Differential thresholds and numerical $F$-invariants} \ 

We now start comparing differential thresholds with other numerical invariants in prime characteristic.

\begin{definition}
	Let $R$ be a   $F$-finite  ring. Let $\fa,\fb\subseteq R$ be proper  ideals such that $\fa\subseteq \sqrt{\fb}$.
		\begin{enumerate}[(i)]
\item 		The $F$-threshold of $\fa$ in $\fb$ \cite{MTW,HMTW,DSNBP} is defined by
	$$
	c^\fb(\fa)= \lim\limits_{e\to\infty} \frac{ \max\{n\in\NN\; | \; \fa^n \not\subseteq \fb^{[e]} \} }{p^e}.
	$$
	
\item 	If $R$ is $F$-split,  the Cartier threshold of $\fa$ in $\fb$ \cite{DSHNBW} is defined by
	$$
	\operatorname{ct}^\fb(\fa)= \lim\limits_{e\to\infty} \frac{ \max\{n\in\NN \; | \; \cC^e_R\cdot \fa^n \not\subseteq \fb \} }{p^e}.
	$$
	\end{enumerate}
\end{definition}

\begin{proposition}\label{PropDTFinv}
	Let $R$ be a   $F$-finite  ring, and $\fa\subseteq R$ be a proper ideal. Then:
		\begin{enumerate}[(i)]	 
\item 	Every $F$-threshold of $\fa$ is a differential threshold of $\fa$.
\item 	If $R$ is $F$-split, then  every Cartier threshold of $\fa$ is a differential threshold of $\fa$.
\item   If $R$ is strongly $F$-regular, then every $F$-jumping number of  $\fa$ is a differential threshold of $\fa$.
\item If $R$ is regular, the set of $F$-jumping numbers of $\fa$ and the set of differential thresholds of $\fa$ agree.
\end{enumerate}
\end{proposition}
\begin{proof}
	The first  claim follows from Remark~\ref{RemDiffJumps} and Lemma~\ref{lem:d-ideals}.

	The claim about Cartier thresholds   follows from Remark~\ref{RemDiffJumps} and Lemma~\ref{lem:d-ideals}, since $\max\{n\; | \; \cC^e_R\cdot \fa^n \not\subseteq \fb \}=\max\{n \;|\; \fa^n \not\sq I_e(\fb)\}$.

We now focus on the third statement.  If $D^{(e)}_R \cdot \fa= D^{(e)}_R \cdot \fb$, then $\cC^e_R \cdot \fa=\cC^e_R \cdot \fb$  (see Lemma~\ref{lemma-testSFR3}).
Then the set of differential jumps of level $e$ contains the set \[{\cA(p^e)=\{n\in \NN \; | \; \cC^e_R \cdot  \fa^n\neq \cC^e_R \cdot \fa^{n+1} \}}.\]
It suffices to show that every jumping number is a limit of elements in  $\frac{1}{p^e}\cA(p^e)$.
We recall that in a strongly $F$-regular ring  we have that $\tau(\fa^\lambda)=\bigcup_{e\in\NN} \cC^e_R \cdot \fa^{\lceil p^e\lambda\rceil}$ \cite[Proposition~4.4]{TTFFRT}.
We set $a$ such that $\tau(\fa^\lambda)=\cC^a \cdot \fa^{\lceil p^a\lambda\rceil}$.
Let $\nu_e=\max\{ n\; | \; \cC^e_R\cdot \fa^n \not \sq \tau(\fa^{\lambda}) \}\in \cA(p^e)$.
We now show that $\lim\limits_{e\to\infty}\frac{\nu_e}{p^e}=\lambda$.
Set $\epsilon>0$. We pick $b$ such that $|\lambda-\frac{\lceil p^e \lambda\rceil }{p^e} |>\frac{\epsilon}{2}$ for $e\geq b$.
We set $\alpha=\lambda-\frac{\epsilon}{2}$ and  $s\in\NN$ such that $\tau(\fa^{\alpha})=\cC^{s}\cdot \fa^{\lceil p^s \alpha \rceil} $.
Then,
$$
\cC^{e} \cdot \fa^{\lceil p^e \alpha \rceil} =\tau(\fa^{\alpha})\neq \tau(\fa^{\lambda})=\cC^{e} \cdot \fa^{\lceil p^e \lambda \rceil}
$$
for $e\geq \max\{a,s\}$.
Then,
$\lceil p^s \alpha \rceil \leq \nu_e\leq  \lceil p^s \lambda \rceil $ for $e\geq \max\{a,s\}$. Thus,
$$\alpha\leq \frac{\lceil p^e \alpha \rceil}{p^e}\leq \nu_e\leq \frac{\lceil p^e \lambda \rceil}{p^e}\leq \lambda+\frac{\epsilon}{2}$$ for $e\geq \max\{a,b,s\}$.
Hence, $\lim\limits_{e\to\infty}\frac{\nu_e}{p^e}=\lambda$.

We now focus on the last claim. Since $R$ is a regular $F$-finite ring, $\cC^e_R  \cdot\fa=\cC^e_R \cdot\fb$ if and only if $D^{(e)}_R\cdot\fa= D^{(e)}_R\cdot\fb$ \cite[Lemma 3.1]{AMBL}.
By Proposition \ref{PropDTFinv}, it suffices to show that every differential threshold is an $F$-jumping number.
Then, the set of differential jumps of level $e$ coincides with the set $\{n\in \NN \; | \; \cC^e_R\cdot  \fa^n\neq \cC^e_R\cdot \fa^{n+1} \}$.
We recall that $\tau(\fa^\lambda)=\bigcup_{e\in\NN} \cC^e_R\cdot \fa^{\lceil p^e\lambda\rceil}$ \cite[Definition 2.9]{BMSm2008}.
Let $\lambda=\lim\limits_{e\to\infty}\frac{\nu_e}{p^e}$ with $\nu_e \in \cA(p^e)$.
There exists $\epsilon>0$ such that $\tau(\fa^{\lambda})=\cC^e_R\cdot \fa^k $ for every $\lambda<\frac{k}{p^e}<\lambda+\epsilon$ \cite[Proposition~2.14]{BMSm2008}.
Then, $\frac{\nu_e}{p^e}\leq \lambda$ for $e\gg 0$.
Set $a$ such that $\frac{1}{p^e}>\frac{\epsilon}{2}$, $\frac{\nu_e}{p^e}\leq \lambda$   and $\tau(\fa^{\lambda})=\cC^e_R\cdot \fa^{\lceil p^e \lambda\rceil} $ for $e\geq a$.
If
$\nu_e+1\leq  p^e \lambda$, then 
$\nu_e+1\leq \lceil p^e \lambda\rceil$ and
$  \cC^e_R\cdot  \fa^{\nu_e+1}\supseteq  \cC^e_R\cdot \fa^{\lceil p^e \lambda\rceil} =\tau(\fa^{\lambda})$.
If $\nu_e+1>  p^e \lambda\leq \nu_e$, then $ \cC^e_R\cdot  \fa^{\nu_e+1}=\tau(\fa^{\lambda})$.
We have that $ \tau(\fa^{\frac{\nu_e}{p^e}})\supseteq \cC^e_R\cdot  \fa^{\nu_e}\supsetneqq  \cC^e_R\cdot  \fa^{\nu_e+1}\supseteq \tau(\fa^{\lambda})$.
We conclude that $\lambda$ is an $F$-jumping number.
\end{proof}

\begin{proposition}
	Let $(R,\fm,K)$ be an  $F$-finite $F$-split  ring, and $\fa\subseteq R$ an ideal.
	Then, $\fpt(\fa)$ is the smallest differential threshold of $\fa$.
\end{proposition}
\begin{proof}
	We note that the first jump of level $e$ is 
	$$
	\nu_e=\max\{n\; | \; D^{(e)}_R \cdot\fa^n\neq R\}=\max\{n\; | \; D^{(e)}_R \cdot \fa^n\subseteq  \fm\}.
	$$
	We have that 
	$D^{(e)}_R \cdot\fa^n\subseteq \fm$
	if and only if $\fa^n\subseteq \{f\in R\; | \;  D^{(e)}_R \cdot  f\in \fm\}$.
We have that	$\{f\in R\; | \;  D^{(e)}_R \cdot  f\in \fm\}=I_e(\fm)$ \cite[Proposition 5.10]{BJNB19}.
	Then, $\fpt(\fa)=\lim\limits_{e\to \infty} \frac{\nu_e}{p^e}$ is the smallest differential threshold.
\end{proof}

\subsection{Discreteness and rationality results} \ 

We now show that the set of  differential thresholds is closed under multiplication by $p$. This is known for $F$-thresholds, but not for $F$-jumping numbers outside Gorenstein rings.  Then, this result shows one of the advantages of the unified approach provided by differential thresholds.
 
\begin{lemma}\label{LemmaMultiplicationP}
	Let $R$ be an  $F$-finite $F$-split ring, and $\fa\subseteq R$ an ideal.
	If $\lambda$ is a differential threshold, then $p\lambda$ is also a  differential threshold.
\end{lemma}

\begin{proof}
	Let $r$ be the number of generators of $\fa$.  
	Let $\nu_e \in \cB^\bullet_\fa (p^e)$ be such that $\lim\limits_{e\to\infty} \frac{\nu_e}{p^e}=\lambda$.
	By Proposition~\ref{prop:higher-jumps} and Lemma~\ref{lem:De-nested} there exists some $\omega_e \in [p \nu_e, p(\nu_e + r -1)] \cap \cB^\bullet_\fa(p^e)$.
	We have that
	$$
	p\lambda=\lim\limits_{e\to\infty} \frac{p \nu_e}{p^e}\leq 
	\lim\limits_{e\to\infty} \frac{ \omega_e}{p^e}
	\leq \lim\limits_{e\to\infty} \frac{p \nu_e}{p^e}=p\lambda.
	$$
	Then, $p\lambda$ is a differential threshold.
\end{proof}

In the following results we focus on Bernstein-Sato admissible  ideals. In this case, we show discretness and rationality. 
In Subsection \ref{SecFFRT} we use these results to provide new cases where the $F$-thresholds are rational numbers.

\begin{theorem}\label{ThmDicretness}
	Let $R$ be an $F$-finite ring and $\fa$ an ideal. If $\fa\subseteq R$ is a Bernstein-Sato admissible ideal, then the set of differential thresholds for $\fa$ is discrete.
	If $R$ is $F$-split, then the converse holds.
\end{theorem}
\begin{proof}
Let $r$ the number of generators of $\fa$. We first show that if $\fa$ is Bernstein-Sato admissible, then  the set of differential thresholds is discrete. By Proposition \ref{PropSkoda}, it suffices to show that the set of differential thresholds in $(0,r)$ is finite. Since $\fa$ is a  Bernstein-Sato admissible  ideal, there exists $b\in\NN$ such that
	$\# \big(\cB^\bullet_\fa (p^e) \cap [0, r p^e) \big)\leq b$ for every $e$, and we claim that there are at most $b$ differential thresholds in $(0,r)$.
	
	Suppose, for a contradiction, that $\lambda_1, \ds, \lambda_{b + 1} \in (0,r)$ are distinct differential thresholds of $\fa$. Pick disjoint open intervals $U_1, \ds, U_{b+1} \sq (0,r)$ with $\lambda_i \in U_i$. Then there is some $e$ large enough and $\nu_1, \ds, \nu_{b+1} \in \cB^\bullet_\fa(p^e)$ with $\nu_i / p^e \in U_i$ for every $i$. It follows that $\nu_1, \ds, \nu_{b+1}$ are distinct differential jumps of level $e$ in the interval $(0, r p^e)$, which gives a contradiction. 
	
	Now suppose that $R$ is $F$-split, and assume that the set of differential thresholds is discrete. In particular, there are finitely many differential thresholds in the interval $[0,r]$; let $0\leq \lambda_1<\cdots<\lambda_c\leq r$ be these differential thresholds. To obtain a contradiction, suppose that $\fa$ is not Bernstein-Sato admissible. Then we can choose some $e\in \N$ such that the number of differential jumps of level $e$ is greater than $(r+1)c$. By Lemma~\ref{lem:thresh-nearby}, every differential jump of level $e$ lies in $\bigcup_{i=1}^c [p^e \lambda_i-r,p^e \lambda_i]$. Since there are at most $(r+1)c$ integers in this set, we obtain the desired contradiction.
\end{proof}

\begin{theorem}\label{ThmDTrational}
	Let $R$ be an $F$-finite $F$-split ring, and $\fa\subseteq R$ be a Bernstein-Sato admissible  ideal.
	Then, every differential threshold of $\fa$ is a rational number.
\end{theorem}
\begin{proof}
	Let $\lambda$ be a differential threshold for $\fa$.
	We fix $e_0$ such  that $p^{e_0}\lambda> r$.
	For $e\geq e_0$,
	we take  $\lambda_{e}=p^{e}\lambda-\lfloor p^{e}\lambda \rfloor+r-1$.
	By construction $\lambda_e$ is a differential threshold for every $e$ by Proposition~\ref{PropSkoda}
	and Lemma~\ref{LemmaMultiplicationP}.
	By Theorem~\ref{ThmDicretness}, there exists $e_1<e_2$ such that
	$$
	\lambda_{e_1}
	=p^{e_1}\lambda-\lfloor p^{e_1}\lambda \rfloor+r-1
	=
	p^{e_2}\lambda-\lfloor p^{e_2}\lambda \rfloor+r-1
	=
	\lambda_{e_2}.
	$$ 
	Since $e_2>e_1$, we conclude that
	\[
	\lambda=\frac{\lfloor p^{e_2}\lambda \rfloor-\lfloor p^{e_1}\lambda \rfloor}{p^{e_2}-p^{e_1}}\in\QQ. \qedhere
	\]
\end{proof}

\subsection{Comparison between Bernstein-Sato roots and differential thresholds} \ 

We end this subsection with a comparison between differential thresholds and Bernstein-Sato roots. We note that we do not assume Bernstein-Sato admissibility in this result. 

\begin{theorem}\label{thm:threshsareroots}
	 Let $R$ be $F$-split. Let $\fa$ be an ideal with $r$ generators.
	\begin{enumerate}
		\item If $\alpha\in \Z_{(p)}$ is a Bernstein-Sato root, then there is some differential threshold $\lambda$ of $\fa$ such that \[\alpha - \lceil \alpha \rceil +\lambda \in \begin{cases} \{0,\dots, r-1\} &\text{if }\alpha\notin \ZZ_{<0}\\
		\{1,\dots, r\} &\text{if }\alpha\in \ZZ_{<0}
		.\end{cases}\]
		\item Conversely, if $\lambda\in (\ZZ_{(p)})_{\geq 0}$ is a differential threshold for $\fa$, then there is some Bernstein-Sato root $\alpha$ for $\fa$ such that
		\[\alpha +\lambda - \lfloor \lambda \rfloor \in \begin{cases} \{1-r,2-r,\dots,0\} &\text{if }\lambda\notin \ZZ_{\geq 0}\\
		\{-r,\dots, 0\} &\text{if }\lambda\in \ZZ_{\geq 0}
		.\end{cases}\]
	\end{enumerate}
	
	Thus, there is an equality of cosets in $\ZZ_{(p)}/\ZZ$:
	\[ \{ \alpha + \ZZ \ | \ \alpha\in \BSR(\fa)\cap \ZZ_{(p)}\} = \{ -\lambda + \ZZ \ | \ \lambda\in\ZZ_{(p)} \text{ a differential threshold of }\fa \}.\]
\end{theorem}
\begin{proof}
	We start with (i). 	By Proposition~\ref{eqroot}, for every $a$, there is some $s\in\{0,\dots,r-1\}$ such that $\aaa + sp^a \in \cB_{\fa}^\bullet(p^a)$. Thus, there is an $s\in\{0,\dots,r-1\}$ such that $\displaystyle \aaae + sp^{ae} \in \cB_{\fa}^\bullet(p^{ae})$ for infinitely many~$a$. 
	
	If $\alpha\in \ZZ_{<0}$, then by Lemma~\ref{lem:expn}, we have $\displaystyle \aaae + sp^{ae} = p^{ae} + \alpha + s p^{ae} \in \cB_{\fa}^\bullet(p^{ae})$. It then follows from Proposition~\ref{prop:equivsthresh} that $\displaystyle \lim_{a\to\infty} \frac{(s+1)p^{ae} + \alpha}{p^{ae}} = s+1$ is a differential threshold.

	If $\alpha\notin \ZZ_{<0}$, then by Lemma~\ref{lem:expn}, we have $\aaae = (1-p^{ae})(\alpha-\lceil \alpha \rceil) +\lceil \alpha \rceil$ for $a\gg 0$. It then follows from Proposition~\ref{prop:equivsthresh} that $\displaystyle \lim_{a\to\infty} \frac{(1-p^{ae})(\alpha-\lceil \alpha \rceil) +\lceil \alpha \rceil + s p^{ae}}{p^{ae}} = s-\alpha +\lceil \alpha \rceil$ is a differential threshold.

	For (ii), let $\lambda\in \ZZ_{(p)}$ be a differential  threshold.
	
	For $\lambda=0$, it follows from Remark~\ref{rem:zerothresh} that $0$ is a Bernstein-Sato root.
	
%	 Next, we deal with the case $\lambda \in \ZZ_{> 0}$. Note that $p^a\tr{\lambda}{a}=\lambda-1$ in this case. By  Proposition~\ref{prop:equivsthresh}, for every $a$, there is then some $s_a\in \{-1,0,\dots,r-1\}$ such that $\displaystyle {\lambda-1 - s_a \in \cB^{\bullet}_{\fa}(p^{a})}$. Therefore, there is some $s\in \{-1,0,\dots,r-1\}$ such that ${\lambda-1 - s \in \cB^{\bullet}_{\fa}(p^{a})}$ for infinitely many~$a$. By Proposition~\ref{eqroot}, the $p$-adic limit $\lambda-1 - s$ of this constant sequence is a Bernstein-Sato root, as required.
	 
	 	 Next, we deal with the case $\lambda \in \ZZ_{> 0}$.  By  Proposition~\ref{prop:equivsthresh}, for every $a$, there is some $\nu_a\in \cB^{\bullet}_{\fa}(p^{a})$ such that $p^a \lambda - r \leq \nu_a \leq p^a \lambda$. Writing $\nu_a = p^a\lambda - s_a$, we have that $s_a\in \{ 0, \dots, r\}$ for all $a$ and $p^a\lambda - s_a \in \cB^{\bullet}_{\fa}(p^{a})$. There is some $s\in \{0,\dots, r\}$ such that $s_a= s$  and $p^a\lambda - s\in  \cB^{\bullet}_{\fa}(p^{a})$ for infinitely many values of $a$. It follows from  Proposition~\ref{eqroot} that $-s$, which is the $p$-adic limit of $p^a\lambda - s$, is a Bernstein-Sato root of $\fa$.
	 
	 Finally, suppose that $\lambda\notin \ZZ_{\geq 0}$, and let $e\in \NN$ be such that $(p^e-1)\lambda \in \NN$. Write {$\lambda = \lceil \lambda \rceil -1 + \mu$}, so $\mu\in (0,1)$. We then have $\displaystyle \lceil p^{ae} \lambda \rceil - 1 = p^{ae} \tr{\lambda}{ae} = (p^{ae}-1)\mu + p^{ae} (\lceil \lambda \rceil -1)$ for all $a\in \NN$. By Proposition~\ref{prop:equivsthresh}, for every $a$, there is some $\nu_{ae} \in \cB^{\bullet}_{\fa}(p^{ae})$ such that $\lceil p^{ae} \lambda  \rceil - r \leq \nu_{ae} \leq \lfloor p^{ae} \lambda\rfloor = \lceil p^{ae} \lambda  \rceil -1$. Writing $s_{ae} = \lceil p^{ae} \lambda  \rceil -1 -\nu_{ae}$, one has $s_{ae} \in\{0, 1,\dots, r-1\}$ and 
	 $(p^{ae}-1)\mu + p^{ae} (\lceil \lambda \rceil -1) - s_{ae}\in \cB^{\bullet}_{\fa}(p^{ae})$. Therefore, there is some $s\in \{0,\dots,r-1\}$ such that $(p^{ae}-1)\mu + p^{ae}(\lceil \lambda \rceil -1) - s \in \cB^{\bullet}_{\fa}(p^{ae})$ for infinitely many $a$.  It follows from Proposition~\ref{eqroot} that  $\lim_a (p^{ae}-1)\mu + p^{ae}\lfloor \lambda \rfloor - s = -\mu -s$ is a Bernstein-Sato root of $\fa$.
\end{proof}

\begin{corollary}\label{cor:principalrootsthresholds} Let $R$ be $F$-finite and $F$-split.
	If $\fa=(f)$ is principal and Bernstein-Sato admissible, then the set of Bernstein-Sato roots of $(f)$ is exactly the set of negatives of differential thresholds of $(f)$ in the interval $[0,1] \cap \Z_{(p)}$.
\end{corollary}

\begin{corollary}
	If $R$ is strongly $F$-regular and $\fa$ is an ideal, then there is a containment in $\ZZ_{(p)} / \ZZ$:
	\[ \{ -\lambda +\ZZ \ | \ \lambda \in \ZZ_{(p)} \ \text{is an $F$-jumping number of $\fa$}\, \} \subseteq \{ \alpha +\ZZ \ | \ \alpha \in \ZZ_{(p)} \ \text{is a Bernstein-Sato root of $\fa$}\, \}.  \]
\end{corollary}

%%%%%%%%%%%%%%%%%%%%%%%%%%%%%%%%%%%%%%%%%%%%%%%%%%%%%%%%%%%%%%
	%%%%%%%%%%%%%%%%%%%%%%%%%%%%%%%%%%%%%%%%%%%%%%%%%%%%%%%%%%%%%%%%%%%%%%%	
\section{Classes of Bernstein-Sato admissible rings}
%%%%%%%%%%%%%%%%%%%%%%%%%%%%%%%%%%%%%%%%%%%%%%%%%%%%%%%%%%%%%%%%%%%%%%%

\subsection{Rings with finite $F$-representation type}\label{SecFFRT} \ 

In this section we prove that every ideal in a graded $\KK$-algebra with finite $F$-representation type is Bernstein-Sato admissible; therefore all ideals have a finite number of Bernstein-Sato roots. We closely mimic the strategy employed by Takagi and Takahashi \cite{TTFFRT} in their proof of discreteness of $F$-jumping numbers.

\begin{definition}[{\cite{TTFFRT}}]
	Let $R=\bigoplus_{n\in \NN} R_n$ be finitely generated  $\NN$-graded $\KK$-algebra over a field $R_0=\KK$.
	We say that a  $R$ has  finite $F$-representation type if there exist a finite set of finitely generated graded $R$-modules, $M_1,\ldots,M_\ell$ such that for every $e\in\NN$ there exist
	$\alpha_{e,i} \in\NN$ and $\theta^{(e)}_{i,j}\in \QQ_{\leq 0}$ such that 
	$$F^e_* R\cong \bigoplus^\ell_{i=1}  \bigoplus^{\alpha_{e,i}}_{j=1}   M_i (\theta^{(e)}_{i,j}),$$
where the grading on $F^e_* R$ is as in Definition \ref{DefBasicCharP}(ii).
	We say that  $M_1,\ldots,M_\ell$ are the  finite $F$-representation type factors of $R$. 
\end{definition}

\begin{definition}[{\cite{TTFFRT}}]
	Let $R=\bigoplus_{n\in \NN} R_n$ be finitely generated  graded $\KK$-algebra with $R_0=\KK$,
	and $\fa\subseteq R$.
	Let $M$ be a finitely generated  $R$-module.
	For $e\in\NN$, we set
	$$
	I_e(\fa,M)=\Hom_R(F^e_* R,M) \cdot F^e_*  \fa.
	$$
\end{definition}

\begin{lemma}\label{LemmaGenDeg}
	Let $R=\bigoplus_{n\in \NN} R_n$ be finitely generated  graded $\KK$-algebra with $R_0=\KK$,
	$\fa\subseteq R$ be an ideal and $M$ be a graded $R$-module.
	Suppose that $R$ has  finite $F$-representation type, and that $\fa$ is generated in degree less or equal to $N$.
	Then, there exists an integer $C\in\NN$ such that $I_e(\fa,M)$ is generated in degree less or equal to $\lfloor C+\frac{N}{p^e}\rfloor$ for every integer $e \geq 0$.
\end{lemma}

\begin{proof}
	Pick $C > 0$ large enough so that, for all $i = 1, \ds, \ell$, the module $\Hom_R(M_i , M)$ is generated in degrees $\leq C$. This implies that for all $i = 1, \ds, \ell$ and $\theta \in \Q$ the module $\Hom_R(M_i(\theta), M) = \Hom_R(M_i, M)(- \theta)$ is generated in degrees $\leq C + \theta$ and, since all $\theta^{(e)}_{i,j}$ are negative, we conclude that the module
	$$\Hom_R(F^e_* R, M) = \bigoplus_{i = 1}^\ell \bigoplus_{j = 1}^{\alpha_{e,i}} \Hom_R(M_i(\theta^{(e)}_{i,j}), M)$$
	is generated in degrees $\leq C$.
	
	We consider the module $\Hom_R(F^e_* R, M) \otimes F^e_* R$ with the induced grading, and we note that its submodule $\Hom_R(F^e_* R, M) \otimes F^e_* \fa$ is generated in degrees $\leq C + N/p^e$. Note that $I_e(\fa, M)$ is the image of $\Hom_R(F^e_* R, M) \otimes F^e_*  \fa$ under the evaluation morphism $\Phi: \Hom_R(F^e_* R, M) \otimes_R F^e_* R \to M$ and, since $\Phi$ is homogeneous of degree zero, the result follows. 
\end{proof}

\begin{theorem}\label{ThmFFRTAdmissible}
	Let $R=\bigoplus_{n\in \NN} R_n$ be finitely generated  graded $\KK$-algebra with $R_0=\KK$.
	If $R$ has finite $F$-representation type then $R$ is a Bernstein-Sato admissible ring. 
\end{theorem}
\begin{proof}
	Let $N\in\NN$ be such that $\fa$ is generated in degree at most $N$.
	Let $r$ be the number of generators of $\fa$, and fix $m\leq rp^e$.
	By Lemma \ref{LemmaGenDeg}, 
	there exists $C_i$ such that 
	$\Hom(F^e_* R,M_i) \cdot F^e_* \fa^m$
	is generated in degree at most $\frac{Nm}{p^e}+C_i\leq Nr+C_i$.
	Let $\beta_i:=\dim_\KK [M_i]_{\leq Nr+C_i}$.
	Then, 
	$\{I_e(\fa^m,M_i) \; | \; m=0,\ldots, rp^e  \}$ has at most $\beta_i$ elements.
	We note that $I_e(\fa^m,M_i)=I_e(\fa^m,M_i(-\gamma))$ for every $\gamma\in \QQ$. Then, 
	\begin{align*}
	D^{(e)}_R\cdot  \fa^m  &= \Hom(F^e_* R,F^e_* R)  \cdot F^e_*  \fa^m \\
	&= \Hom\left( F^e_* R, \bigoplus^\ell_{i=1}  \bigoplus_j   M^{\alpha^{(e)}_{i,j}}_i (\theta^{(e)}_{i,j}) \right) \cdot F^e_*  \fa^m \\
	&=
	\bigoplus^\ell_{i=1}  \bigoplus_j  \Hom\left(  F^e_* R,  M^{\alpha^{(e)}_{i,j}}_i (\theta^{(e)}_{i,j})  \right) \cdot F^e_* \fa^m 
	\end{align*}
	and so
	$\{ D^{(e)}_R \cdot  \fa^m   | \; m=0,\ldots, rp^e  \}$  has at most $\beta_1\cdots \beta_\ell$ elements.
\end{proof}

As a consequence of the previous theorem, we obtain a new case where the differential thresholds satisfy rationality and discreteness. In particular, these recover and extends previous results known for Stanley-Reisner rings \cite[Theorems A \& B]{BC}.

\begin{corollary}\label{CorFFRTRational}
	Let $R=\bigoplus_{n\in \NN} R_n$ be finitely generated  graded $\KK$-algebra with $R_0=\KK$.
	Suppose that $R$ has  finite $F$-representation type.
	Then,  the sets
	$$
	\{c^\fb(\fa) \; | \;  \fa\subseteq \sqrt{\fb} \neq (1)\}
\quad \text{and} \quad 
	\{	\operatorname{ct}^\fb (\fa) \; | \;  \fa\subseteq \sqrt{\fb} \neq (1)\}
	$$
consist of rational numbers and do not have any accumulation points. 
\end{corollary}
\begin{proof}
Since every ideal is Bernstein-Sato admissible by Theorem \ref{ThmFFRTAdmissible}, the claim follows from 
Proposition \ref{PropDTFinv} and Theorems \ref{ThmDicretness} \& \ref{ThmDTrational}.
\end{proof}

%%%%%%%%%%%%%%%%%%%%%%%%%%%%%%%%%%%%%%%%%%%%%%%%%%%%%%%%%%%%%%%%%%%%%%%%%%%%%
\subsection{Direct summands} \ 
%%%%%%%%%%%%%%%%%%%%%%%%%%%%%%%%%%%%%%%%%%%%%%%%%%%%%%%%%%%%%%%%%%%%%%%%%%%%%

Next, we provide a second class of possibly singular rings for which all ideals are Bernstein-Sato admissible: the class of direct summands of regular rings. Passing to direct summands behaves especially well in the case of level-differentially extensible direct summands, a notion introduced by Brenner together with the first two authors \cite{BJNB19}. Let us give the definition in the case of $F$-finite rings.
\begin{definition}
	An extension $R \sq S$ of $F$-finite rings is called level-differentially extensible if for every integer $e \geq 0$ and every $\delta \in D^{(e)}_R$ there exists some $\t \delta \in D^{(e)}_S$ such that $\delta = \t \delta |_R$. 
\end{definition}

We note that large classes of invariant rings can be realized as level-differentially extensible direct summands of polynomial rings (see \cite[Section~6]{BJNB19} for more details).

\begin{theorem}\label{thm:directsummand}
	Let $R \sq S$ be a split extension of $F$-finite rings and $\fa \sq R$ be an ideal. Then:
	\begin{enumerate}[(i)]
		\item\label{item:jumpds} For all integers $e\geq 0$, $\cB^\bullet_\fa(p^e) \sq \cB^\bullet_{\fa S} (p^e)$. 
		\item\label{item:admds} If $\fa S \sq S$ is a Bernstein-Sato admissible ideal, then so is $\fa \sq R$. 
		\item\label{item:rootds} Every Bernstein-Sato  root of $\fa$ is a Bernstein-Sato root of $\fa S$. 
		{ \item\label{item:threshds} Every differential threshold of $\fa$ is a differential threshold of $\fa S$.}
	\end{enumerate}
	Assume furthermore that the extension $R \sq S$ is level-differentially extensible. Then:
	\begin{enumerate}[resume]
		\item\label{item:jumpde} For all integers $e \geq 0$, $\cB^\bullet_\fa(p^e) = \cB^\bullet_{\fa S}(p^e)$.
		\item\label{item:rootde} The Bernstein-Sato roots of $\fa$ and the Bernstein-Sato roots of $\fa S$ coincide.
		{ \item\label{item:threshde} The differential thresholds of $\fa$ and the differential thresholds of $\fa S$ coincide.}
	\end{enumerate}
\end{theorem}
\begin{proof}
	Let us start with \ref{item:jumpds}. Suppose that $n \notin \cB^\bullet_{\fa S}(p^e)$; that is, $D^{(e)}_S \cdot \fa^n = D^{(e)}_S \cdot \fa^{n+1}$, and we claim that $D^{(e)}_R \cdot \fa^n = D^{(e)}_R \cdot \fa^{n+1}$. We observe it suffices to prove that $\fa^n \sq D^{(e)}_R \cdot \fa^{n+1}$. To prove this, suppose that $f \in \fa^n$. We know that there exist differential operators $\xi_i \in D^{(e)}_S$ and elements $g_i \in \fa^{n+1}$ such that $f = \sum_i \xi_i \cdot g_i$. Applying a splitting $\beta: S \to R$ to this equation, we obtain $f = \sum_i (\beta \circ \xi_i) \cdot g_i$. Since $\beta \circ \xi_i |_R \in D^{(e)}_R$, we conclude that $f \in D^{(e)}_R \cdot \fa^{n+1}$. This proves the claim, and thus \ref{item:jumpds} is proven. 
	
	Statement \ref{item:admds} follows from \ref{item:jumpds}: every bound for $\# \big( \cB_{\fa S}^\bullet (p^e) \cap [0, rp^e) \big)$ is a bound for $\# \big( \cB_\fa^\bullet(p^e) \cap [0, rp^e) \big)$. Statement \ref{item:rootds} follows from \ref{item:jumpds} together with \ref{item:admds}. { Statement \ref{item:threshds} follows from \ref{item:jumpds}.}
	
	We now assume that the extension $R \sq S$ is level-differentially extensible, and prove \ref{item:jumpde}. Suppose that $n \notin \cB^\bullet_\fa(p^e)$; that is, $D^{(e)}_R \cdot \fa^n = D^{(e)}_R \cdot \fa^{n+1}$, and we claim that $D^{(e)}_S \cdot \fa^n = D^{(e)}_S \cdot \fa^{n+1}$ and, once again, we observe that it suffices to show that $\fa^n \sq D^{(e)}_S \cdot \fa^{n+1}$. We thus let $f \in \fa^n$. We know that there exist $\xi_i \in D^{(e)}_R$ and $g_i \in \fa^{n+1}$ such that $f = \sum_i \xi_i \cdot g_i$. For every $i$, let $\t{\xi_i} \in D^{(e)}_S$ be a lift of $\xi_i$. We conclude that $f = \sum_i \t{\xi_i} \cdot g_i$ and therefore $f \in D^{(e)}_S \cdot \fa^{n+1}$. This proves the claim, and thus \ref{item:jumpde} is proven.
	
	Statements \ref{item:rootde} { and \ref{item:threshde}} follow from \ref{item:jumpde}. 
\end{proof}

\begin{corollary}
	Suppose that $S$ is a ring in which every ideal is Bernstein-Sato admissible (e.g. $S$ is regular, or graded with finite $F$-representation type), and that $R$ is a direct summand of $S$. Then every ideal of $R$ is Bernstein-Sato admissible.
\end{corollary}

%%%%%%%%%%%%%%%%%%%%%%%%%%%%%%%%%%%%%%%%%%%%%%%%%%%%%%%%%%%
\section{Bernstein-Sato roots via the Malgrange construction} \label{scn-BSroots-via-V}
%%%%%%%%%%%%%%%%%%%%%%%%%%%%%%%%%%%%%%%%%%%%%%%%%%%%%%%%%%%

\subsection{Bernstein-Sato roots and $N_\fa$} \ 

Let $R$ be an $F$-finite ring and let $\fa \sq R$ be an ideal. The definition of the Bernstein-Sato roots of $\fa$ given in Definition \ref{DefBSrootsPadic} has three important advantages: it is nontechnical, as it only relies on the relatively simple notion of differential jumps; one can use it to prove things, as illustrated in Section \ref{Sec-BSR} and, finally, it is also more convenient for computing Bernstein-Sato roots (see Section \ref{Sec-examples}). However, it also has a serious drawback: it is not clear how this notion of Bernstein-Sato roots is related to the classical notion of the Bernstein-Sato polynomial. Our goal in this subsection is to explain how one arrives at Definition \ref{DefBSrootsPadic} from the point of view of the classical theory. 

Fix generators $\fa = (f_1, \ds, f_r)$ for $\fa$, and for every integer $e \geq 0$ we consider the module
$$H^e_\fa := \frac{R[\ut]}{(\ul{f} - \ut)^{p^e \bone}} \delta_{p^e} = \frac{R[\ut]}{(f_1 - t_1)^{p^e} \cds (f_r - t_r)^{p^e}} \delta_{p^e}$$
where $\delta_{p^e}$ is just a formal symbol. 
In particular, $H^e_\fa$ is the quotient of $R[\ut]$ by a $D^{(e)}_{R[t]}$-ideal and therefore it is a $D^{(e)}_{R[t]}$-module itself. Note that, as an $R$-module, we have a decomposition
$$H^e_\fa = \bigoplus_{\ul{a} \in \{0, \ds, p^e - 1\}^r} R \ (\ul{f} - \ut)^{\ul{a}} \ \delta_{p^e}$$

We let $\phi^e: H^e_\fa \to H^{e+1}_\fa$ be the map induced by multiplication by $(\ul{f} - \ut)^{p^e(p - 1) \bone}$, which is $D^{(e)}_{R[\ut]}$-linear. We let $H_\fa$ be the direct limit
$$H_\fa := \lim_{\to} (H^0_\fa \xrightarrow{\phi^0} H^1_\fa \xrightarrow{\phi^1} H^2_\fa \to \cds),$$
which acquires the structure of a $D_{R[\ut]}$-module. We note that the maps $\phi^e$ are injective, and therefore each $H^e_\fa$ is isomorphic to its image in $H_\fa$. From this point onwards, we identify each module $H^e_\fa$ with its image in $H_\fa$, and we think of every element of $H^e_\fa$ as an element of $H_\fa$. For example, for all $e \geq 0$ we have
$$\delta_1 = (\ul{f} - \ut)^{(p^e -1 )\bone} \delta_{p^e}.$$

\begin{lemma}
	There is an isomorphism 
	$$H_\fa \cong H^r_{(f_1 - t_1, \ds, f_r - t_r)} R[\ut]$$
	of $D_{R[\ut]}$-modules.
\end{lemma}
\begin{proof}
	We let $L$ denote the local cohomology module on the right hand side, which we construct via the \v{C}ech complex on the given generators. We have an $R$-module decomposition
	$$L = \bigoplus_{\ul{a} \in (\Z_{>0})^r} R \delta'_{\ul{a}}$$
	where $\delta'_{\ul{a}}$ denotes the class of $(\ul{f} - \ut)^{- \ul{a}}$. 
	
	We let $\psi^e : H^e_\fa \to L$ be the unique $R$-linear map with $\psi^e ((\ul{f} - \ut)^{\ul{a}} \delta_{p^e}) = \delta'_{p^e \bone - \ul{a}}$, which gives an isomorphism of $H^e$ onto the submodule $\bigoplus_{\ul{a} \in \{1, \ds, p^e\}^r} R \delta'_{\ul{a}}$ of $L$. One immediately checks that the $\psi^e$ are compatible as $e$ changes, and that they give an $R$-module isomorphism $\psi: H_\fa \xrightarrow{\sim} L$.
	
	It remains to check that this isomorphism is $D_{R[\ut]}$-linear; i.e. that it is $D^{(e)}_{R[\ut]}$-linear for every $e$. Since $H^i_\fa$ is a $D^{(e)}_{R[\ut]}$-submodule of $H_\fa$ for every $i \geq e$, the $D^{(e)}_{R[\ut]}$-module structure on $H_\fa$ is uniquely determined by the fact that $\xi \cdot (g \delta_{p^i}) = (\xi \cdot g) \delta_{p^i}$ for all $\xi \in D^{(e)}_{R[\ut]}$, all $g \in R[\ut]$ and all $i \geq e$. Now recall that $\delta'_{p^i}$ is the class of $(\ul{f} - \ut)^{-p^i \bone} = (f_1 - t_1)^{-p^i} \cds (f_r - t_r)^{-p^i}$, which is a $p^i$-power. It follows that for every $\xi \in D^{(e)}_R$ and every $i \geq e$, $\xi$ commutes with multiplication by $(\ul{f} - \ut)^{-p^i \bone}$ in the localization $R[\ut]_{(f_1 - t_1)\cds(f_r - t_r)}$. Therefore, in $L$ we have $\xi \cdot (g \delta'_{p^i}) = (\xi \cdot g) \delta'_{p^i}$ for all $\xi \in D^{(e)}_{R[\ut]}$, all $g \in R[\ut]$ and all $i \geq e$. 
\end{proof}

It follows from the proof that the isomorphism we have constructed identifies $\delta_{p^0} = \delta_1$ with the class of $(\ul{f} - \ut)^{- \bone}$, when the local cohomology module is viewed via the \v{C}ech complex. 

The characteristic zero theory leads us to consider the module
$$N_\fa := \frac{V^0 D_{R[\ut]} \cdot \delta_1}{V^1 D_{R[\ut]} \cdot \delta_1}.$$
In the following lemma, we give a description that is more useful for our purposes. Recall that we denote by $(D_{R[\ut]})_0$ the differential operators of degree zero, with the grading induced by $\deg t_i = 1$. 

\begin{lemma}
	We have
	$$N_\fa = \frac{(D_{R[\ut]})_0 \cdot \delta_1}{ (D_{R[\ut]})_0 \cdot \fa \delta_1}.$$
\end{lemma}

\begin{proof}
	Let $I$ denote the ideal $I = (t_1, \ds, t_r)$. Since $(f_i - t_i) \delta_1 = 0$, we have that $I \delta_1 = \fa \delta_1$. By using Lemma \ref{lemma-Vfilt-degree} we get
	\begin{align*}
	V^0 D_{R[\ut]} \cdot \delta_1 & = \sum_{n = 0}^\infty (D_{R[\ut]})_0 I^n \cdot \delta_1 = \sum_{n = 0}^\infty (D_{R[\ut]})_0 \cdot \fa^n \delta_1 \\
	& = (D_{R[\ut]})_0 \cdot \delta_1,
	\end{align*}
	and similarly
	\begin{align*}
	V^1 D_{R[\ut]} \cdot \delta_1 & = \sum_{n = 1}^\infty (D_{R[\ut]})_0 I^n \cdot \delta_1 = \sum_{n = 1}^\infty (D_{R[\ut]})_0 \cdot \fa^n \delta_1 \\
	& = (D_{R[\ut]})_0 \cdot \fa \delta_1. \qedhere
	\end{align*}
\end{proof}

In particular, we conclude that $N_\fa$ is a $(D_{R[\ul t]})_0$-module. In characteristic zero, the Bernstein-Sato polynomial of $\fa$ is the minimal polynomial for the action of $s := - \partial_t t$ on the module $N_\fa$ (the existence of such a minimal polynomial being far from clear). In particular, the module $(N_\fa)$ splits as a direct sum of generalized eigenspaces $N_\fa = \bigoplus_{\lambda \in \CC} (N_\fa)_\lambda$, and the roots of $b_\fa(s)$ are precisely the $\lambda \in \CC$ for which $(N_\fa)_\lambda$ is nonzero. 

In characteristic $p>0$, we view the algebra $\cts(\Zp, \F_p)$ as a subalgebra of $(D_{R[\ul t]})_0$ by using the map $\Delta$ from Subsection \ref{SubSec-CZpFp}, and this subalgebra plays the role of $\CC[s]$. Given a $p$-adic integer $\alpha \in \Zp$, we define
$$(N_\fa)_\alpha := N_\fa / \fm_\alpha N_\fa,$$
and our result is as follows. 

\begin{theorem} \label{thm-BSRoot-equivalence}
	Let $R$ be an $F$-finite ring, $\fa \sq R$ be an ideal and $\alpha \in \Zp$ be a $p$-adic integer. Let $N_\fa$ be the module defined above by using a choice of generators for $\fa$. Then $\alpha$ is a Bernstein-Sato root of $\fa$ if and only if $(N_\fa)_\alpha$ is nonzero.
\end{theorem}

\begin{corollary} \label{cor-BSadm-Na-splits}
	Suppose that $\fa$ is Bernstein-Sato admissible with Bernstein-Sato roots $\{\alpha_1, \ds, \alpha_s\}$. Then we have a decomposition $N_\fa = \bigoplus_{i = 1}^s (N_\fa)_{\alpha_i}.$
\end{corollary}

\begin{proof}
	Follows from Theorem \ref{thm-finiteBSRoots} and Proposition \ref{prop-CzpModules-split}.
\end{proof}

We begin working towards the proof of Theorem \ref{thm-BSRoot-equivalence}. The idea is to write $(N_\fa)_\alpha$ as a direct limit $(N_\fa)_\alpha = \lim_{\to e} (N^e_\fa)_\alpha$, to understand the nonvanishing of the $(N^e_\fa)_\alpha$ for a fixed $e$, and to then analyze the effect of taking the direct limit.

We note that $N_\fa = \lim\limits_{\to e} N^e_\fa$, where 
$$N^e_\fa := \frac{(D^{(e)}_{R[\ut]})_0 \cdot \delta_1}{(D^{(e)}_{R[\ut]})_0 \cdot \fa \delta_1},$$
and that both $(D^{(e)}_{R[\ut]})_0 \cdot \delta_1$ and $(D^{(e)}_{R[\ut]})_0 \cdot \fa \delta_1$ are $(D^{(e)}_{R[\ut]})_0$-submodules of $H^e_\fa$. By viewing $\cts^e(\Zp, \F_p)$ as a subalgebra of $(D^{(e)}_{R[\ul t]})_0$, we conclude that $N^e_\fa$ is a $\cts^e(\Zp, \F_p)$-module. 

Given a $p$-adic integer $\alpha \in \Zp$ we let $\fm^{(e)}_\alpha = \fm_\alpha \cap \cts^e(\Zp, \F_p)$ and 
$$(N^e_\fa)_\alpha := N^e_\fa / \fm^{(e)}_\alpha N^e_\fa.$$
In particular, we have $(N_\fa)_\alpha = \lim\limits_{\to e} (N^e_\fa)_\alpha$. 

\begin{lemma} \label{lemma-binomialcoeffs}
	Let $\ul{x} = (x_1, \ds, x_r)$ and $\ul{y} = (y_1, \ds, y_r)$ be two sets of variables and $e \geq 0$ be an integer. In the ring $\F_p[\ul{x}, \ul{y}]$ we have
	$$(\ul{x} - \ul{y})^{(p^e - 1) \bone} = \sum_{\ul{b}}  \ul{x}^{\ul{b}} \ \ul{y}^{(p^e -1 ) \bone - \ul{b}}$$
	where the sum takes place over all multi-exponents $\ul{b} \in \Z_{\geq 0}^r$ with $0 \leq b_i < p^e$. 
\end{lemma}
\begin{proof}
	In the case where $r = 1$, this follows by observing that 
	$$(x - y)(y^{p^e - 1} + y^{p^e - 2} x + \cds + y x^{p^e - 2} + x^{p^e - 1}) = x^{p^e} - y^{p^e} = (x - y)^{p^e},$$
	together with the fact that $\F_p[x,y]$ is a domain. 
	
	For the general case, first note that the multi-index binomial theorem states that the claim in the lemma is equivalent to the statement that ${(p^e - 1)\bone \choose \ul{b}} (-1)^{r(p^e - 1) - |\ul{b}|} \equiv 1 \mod p$, which in turn is equivalent to the claim that ${(p^e - 1) \bone \choose \ul{b}} \equiv (-1)^{|\ul{b}|} \mod p$. This latter statement follows directly from the $r = 1$ case. 	
\end{proof}

Given an integer $e \geq 0$ and a multi-exponent $\ul{a} \in \{0, \ds, p^e - 1\}^r$ we define $Q^e_{\ul{a}}$ to be the following element of $H^e_\fa$:
$$Q^e_{\ul{a}} := \ut^{(p^e - 1) \bone - \ul{a}} \ \delta_{p^e}.$$
Note we have an $R$-module decomposition
$$H^e_{\fa} = \bigoplus_{\ul{a} \in \{0, \ds, p^e -1\}^r} R \ Q^e_{\ul{a}}.$$ 

Recall that we identified $D^{(e)}_R$ and $D^{(e)}_{\F_p[\ut]}$ with subrings of $D^{(e)}_{R[\ut]}$, and that with this identification we have $(D^{(e)}_{R[\ut]})_0 = D^{(e)}_R \otimes_{\F_p} (D^{(e)}_{\F_p[\ut]})_0$ (see Lemma \ref{lemma-D_R-addvar}).

\begin{lemma} \label{lemma-deg0-acts-on-delta}
	For every integer $e \geq 0$ we have the following equality of submodules of $H^e_\fa$:
	$$(D^{(e)}_{\F_p[\ut]})_0 \cdot \delta_1 = \bigoplus_{0 \leq a_i < p^e} \fa^{|\ul{a}|} \ Q^e_{\ul{a}}.$$
\end{lemma}
\begin{proof}
	We begin by noting that, by Lemma \ref{lemma-binomialcoeffs}, we have
	$$\delta_1 =(\ul{f} - \ut)^{(p^e-1) \bone} \ \delta_{p^e} = \sum_{0 \leq a_i < p^e} \ul{f}^{\ul{a}} \ Q^e_{\ul{a}}.$$
	Given multi-exponents $\ul{b}, \ul{c} \in \Z^r$ with $0 \leq b_i < p^e$ and $c_i < p^e$ we denote by $\sigma^{(e)}_{\ul b \to \ul c}$ the unique element of $D^{(e)}_{\F_p[\ut]}$ such that for all $\ul k \in \{0, 1, \ds, p^e -1\}^r$ we have
	$$
	\sigma^{(e)}_{\ul b \to \ul c} \cdot \ut^{\ul k} =  \begin{cases}
	\ut^{(p^e - 1) \bone - \ul c} \text{ if } \ul k = (p^e - 1) \bone - \ul b \\
	0 \text{ otherwise.}	
	\end{cases}$$
	These operators form an $\F_p$-basis for $D^{(e)}_{\F_p[\ut]}$, and since $\sigma^{(e)}_{\ul b \to \ul c}$ is homogeneous of degree $|\ul b| - |\ul c|$, the subcollection for which $|\ul b| = |\ul c|$ is an $\F_p$-basis for $(D^{(e)}_{\F_p[\ut]})_0$. 
	
	Let $\ul b, \ul c \in \Z^r$ be multi-exponents with $0 \leq b_i < p^e$ and $c_i < p^e$ such that $|\ul b| = |\ul c|$. Let $\ul c', \ul c'' \in \Z^r$ be the unique multi-exponents with $0 \leq c_i' < p^e$, $0 \leq c_i''$ and $\ul c = \ul c' - p^e \ul c''$. We then have $\sigma^{(e)}_{\ul b \to \ul c} \cdot Q^{e}_{\ul a} = 0$ for $\ul a \neq \ul b$ and
	\begin{align*}
	\sigma^{(e)}_{\ul b \to \ul c} \cdot Q^e_{\ul b} & = \ul t^{(p^e - 1) \bone - \ul c' + p^e \ul c''} \delta_{p^e} \\
	& = \ul f^{p^e \ul c''} Q^e_{\ul c'},
	\end{align*}
	where in the last equality we use the fact that $(f_i^{p^e} - t_i^{p^e}) = (f_i - t_i)^{p^e} \delta_{p^e} = 0$. Therefore, $\sigma^{(e)}_{\ul b \to \ul c} \cdot \delta_1 = \ul f^{p^e \ul c'' + \ul b} \  Q^e_{\ul c'}$, and since $p^e |\ul c''| + | \ul b| = |\ul c'|$ we have that $\sigma^{(e)}_{\ul b \to \ul c} \cdot \delta_1 \sq \bigoplus_{0 \leq a_i < p^e} \fa^{|\ul a|} Q^e_{\ul a}.$
	
	For the other inclusion, let $\ul b, \ul a \in \Z^r$ be multi-exponents such that $0 \leq b_i$, $0 \leq a_i < p^e$ and $|\ul b| = |\ul a|$, and we show that $\ul f^{\ul b} \ Q^e_{\ul a} \in (D^{(e)}_{\F_p[\ut]})_0$. Let $\ul b', \ul b'' \in \Z^r$ be the unique multi-exponents such that $0 \leq b'_i < p^e$, $0 \leq b''_i$ and $\ul b = \ul b' + p^e \ul b''$. We then have
	\begin{align*}
	\sigma^{(e)}_{\ul b' \to \ul a - p^e \ul b''} \cdot \delta_1 & = \sigma^{(e)}_{\ul b' \to \ul a - p^e \ul b''} \cdot \ul{f}^{\ul b'} Q^e_{\ul b'} \\
	& = \ul f^{\ul b} Q^e_{\ul a},
	\end{align*}
	and since $|\ul b'| = |\ul a| - p^e \ul b''$, $\sigma^{(e)}_{\ul b' \to \ul a - p^e \ul b''} \in (D^{(e)}_{\F_p[\ut]})_0$, which proves the claim.	
\end{proof}

\begin{proposition} \label{prop:Nestuff}
	Let $R$ be an $F$-finite ring, $\fa \sq R$ be an ideal and $e \geq 0$ be an integer.
	\begin{enumerate}[(i)]
		\item We have a direct sum decomposition
		$$N^e_\fa = \bigoplus_{0 \leq a_i < p^e} \frac{D^{(e)}_R \cdot \fa^{|\ul a|}}{D^{(e)}_R \cdot \fa^{|\ul a| + 1}} \t{Q}^e_{\ul a},$$
		where $\t{Q}^e_{\ul a}$ denotes the image of $Q^e_{\ul a}$ in the quotient.
		
		\item If $\alpha \in \Zp$ is a $p$-adic integer, then $(N^e_\fa)_\alpha$ consists of the summands indexed by those $\ul a$ for which $|\ul a| \equiv \alpha \mod p^e$. 
	\end{enumerate}
\end{proposition}
\begin{proof}
	Lemma \ref{lemma-deg0-acts-on-delta} together with the fact that $(D^{(e)}_{R[\ut]})_0 = D^{(e)}_R \otimes_{\F_p} (D^{(e)}_{\F_p[\ut]})_0$ (see Lemma \ref{lemma-D_R-addvar}) imply that the submodule $(D^{(e)}_{R[t]})_0 \cdot \delta_1$ of $H^e_\fa$ is given by
	$$(D^{(e)}_{R[t]})_0 \cdot \delta_1 = \bigoplus_{0 \leq a_i < p^e} (D^{(e)}_R \cdot \fa^{|\ul a|}) Q^e_{\ul a}.$$
	Similarly, using Lemma \ref{lemma-deg0-acts-on-delta} we observe that $(D^{(e)}_{\F_p[\ut]})_0 \cdot \fa \delta_1 = \fa (D^{(e)}_{\F_p[\ut]})_0 \cdot \delta_1 = \bigoplus_{0 \leq a_i < p^e} \fa^{|\ul a| + 1} Q^e_{\ul a}$, and by once again applying the fact that $(D^{(e)}_{R[\ut]})_0 = D^{(e)}_R \otimes_{\F_p} (D^{(e)}_{\F_p[\ut]})_0$ we conclude that 
	$$(D^{(e)}_{R[t]})_0 \cdot \fa \delta_1 = \bigoplus_{0 \leq a_i < p^e} (D^{(e)}_R \cdot \fa^{|\ul a| + 1}) Q^e_{\ul a},$$
	and part (i) follows. 
	
	For part (ii), recall that the action of $\cts^e(\Zp, \F_p)$ on $N^e_\fa$ comes via the map $\Delta$ defined in Subsection \ref{SubSec-CZpFp}. An easy computation yields that a function $\varphi \in \cts^e(\Zp, \F_p)$ acts on $Q^e_{\ul a}$ by the scalar $\varphi(|\ul a|)$, and therefore
	\[ \fm^{(e)}_\alpha Q^e_{\ul a} = \begin{cases} 0 \text{ if } |\ul a| \equiv \alpha \mod p^e \\ \F_p Q^{e}_{\ul a} \text{ otherwise.} \end{cases} \qedhere \]
\end{proof}
\begin{corollary} \label{cor-NeAlpha-and-diffJump}
	The module $(N^e_\fa)_\alpha$ is a direct sum of the modules from the list 
	$$\bigg\{ \frac{D^{(e)}_R \cdot \fa^{n}}{D^{(e)}_R \cdot \fa^{n+1}} \  \bigg| \  0 \leq n \leq r(p^e-1) \text{ and } n \equiv \alpha \mod p^e  \bigg\},$$
	and every module from the list appears in the decomposition. 
\end{corollary}
\begin{proof}
	The result follows from Proposition \ref{prop:Nestuff}(ii) , together with the observation that 
	\[\{|\ul a| \;|\; 0 \leq a_i < p^e \text{ and } |\ul a| \equiv \alpha \mod p^e \} = \{0 \leq n \leq r(p^e-1)  \;|\; n \equiv \alpha \mod p^e\}. \qedhere\]
\end{proof}

\begin{proposition} \label{prop-vanishNeAlpha-equivalences}
	Let $R$ be an $F$-finite ring, $\fa \sq R$ be an ideal, $e \geq 0$ be an integer and $\alpha \in \Zp$ be a $p$-adic integer. The following are equivalent:
	\begin{enumerate}[(a)]
		\item The module $(N^e_\fa)_\alpha$ is nonzero.
		\item The image of $\delta_1$ in $(N^e_\fa)_\alpha$ is nonzero.
		\item There is a differential jump $n \in \cB^\bullet_\fa(p^e)$ with $n \equiv \alpha \mod p^e$. 
	\end{enumerate}
\end{proposition}
\begin{proof}
We note that (b) implies (a).  To observe that (a) implies (b), note that the subalgebra $\cts^e(\Zp, \F_p)$ of $(D^{(e)}_{R[\ul t]})_0$ is central and therefore $(N^e_\fa)_\alpha$ is a cyclic left $(D^{(e)}_{R[\ul t]})_0$-module generated by $\delta_1$.
	
	By Corollary \ref{cor-NeAlpha-and-diffJump}, if $(N^e_\fa)_\alpha$ is nonzero, we have $D^{(e)}_R \cdot \fa^n \neq D^{(e)}_R \cdot \fa^{n+1}$ for some $n$ with $0 \leq n \leq r(p^e-1)$ and $n \equiv \alpha \mod p^e$. We conclude that  (a) implies (c). To show that (c) implies (a), suppose that we are given a differential jump $n$ as in part (c). By Proposition~\ref{prop:subtract-pe} we can subtract $p^e$ enough times to assume that $0 \leq n \leq r(p^e - 1)$, and the result follows by applying Corollary \ref{cor-NeAlpha-and-diffJump} once again.
\end{proof}

\begin{corollary} \label{cor-zeroRemainsZero}
	Suppose that $(N^e_\fa)_\alpha = 0$ for some $e \geq 0$. Then $(N^i_\fa)_\alpha = 0$ for all $i \geq e$.
\end{corollary}

\begin{proof}
	[Proof of Theorem \ref{thm-BSRoot-equivalence}]
	
	Suppose that $\alpha \in \Zp$ is a Bernstein-Sato root of $\fa$; that is, there is a sequence $(\nu_e) \sq \Z_{\geq 0}$ such that $\nu_e \in \cB^\bullet_\fa(p^e)$ and so that $\alpha$ is the $p$-adic limit of $\nu_e$. By passing to a subsequence (see Remark \ref{rmk-diffjump-subsequence}) we may assume that $\nu_e \equiv \alpha \mod p^e$. By Proposition~\ref{prop-vanishNeAlpha-equivalences}, the image of $\delta_1$ in $(N^e_\fa)_\alpha$ is nonzero for every integer $e \geq 0$. We conclude that the image of $\delta_1$ in $(N_\fa)_\alpha$ is nonzero, and thus $(N_\fa)_\alpha \neq 0$.
	
	For the other direction, suppose that $(N_\fa)_\alpha$ is nonzero. By Corollary \ref{cor-zeroRemainsZero}, we must have that $(N^e_\fa)_\alpha$ is nonzero for every $e \geq 0$. By Proposition \ref{prop-vanishNeAlpha-equivalences} we conclude that for every $e \geq 0$ there is a differential jump $\nu_e \in \cB^\bullet_\fa(p^e)$ with $\nu_e \equiv \alpha \mod p^e$. Since $\alpha$ is the $p$-adic limit of the sequence $(\nu_e)$, $\alpha$ is a Bernstein-Sato root of $\fa$.	
\end{proof}

\subsection{The operators $s_{p^i}$ and the algebra $\cts(\Zp, \F_p)$.} \label{SubSec-spi-and-CZpFp} \

We give a few remarks about why the result on the previous subsection establishes that the definition of Bernstein-Sato root given in Section \ref{Sec-BSR} is a natural extension of the notion in previous work on this subject \cite{QG19}.

Given a $(D_{R[\ul t]})_0$-module $M$ and a $p$-adic integer $\alpha$, we define
$$M_{(\alpha)} : = \{u \in M : s_{p^i} \cdot u = \alpha_i u \},$$
where the operators $s_{p^i}$ are as given in Subsection \ref{Subsec-DiffOps_VFilt} (note that this module was previously  denoted as $M_\alpha$  \cite{QG19}). Recall we also have the definition
$$M_\alpha := M / \fm_\alpha M $$
where $\fm_\alpha \sq \cts(\Zp, \F_p)$ is the ideal of functions that vanish at $\alpha$. 

The module $N_\fa$ is a $(D_{R[\ul t]})_0$- module. Our approach is to think of $N_\fa$ as a $\cts(\Zp, \F_p)$-module by restriction of scalars, and $\alpha \in \Zp$ is a Bernstein-Sato root of $\fa$ precisely when $(N_\fa)_\alpha$ is nonzero (Theorem \ref{thm-BSRoot-equivalence}). We recall that  in earlier work $\alpha$ is defined to be a Bernstein-Sato root whenever $(N_\fa)_{(\alpha)}$ is nonzero \cite{QG19}. A priori these two constructions are different, but the following proposition tells us that they agree whenever the module $N_\fa$ splits nicely; this is the case whenever $\fa$ is a Bernstein-Sato admissible ideal, and therefore the two definitions agree when the ring $R$ is regular (which is the only case considered in the third author's work \cite{QG19}).

\begin{proposition}
	The following are equivalent for a left $(D_{R[\ul t]})_0$-module $M$. 
	\begin{enumerate}[(a)]
		\item We have $\# \{\alpha \in \Zp \ | \ M_{(\alpha)} \neq 0\} < \infty$ and  $M = \bigoplus_{\alpha \in \Zp} M_{(\alpha)}$.
		\item We have $\#\{\alpha \in \Zp \; |\; M_\alpha \neq 0\} < \infty$, and the natural map $\psi: M \to \bigoplus_{\alpha \in \Zp} M_\alpha$ is an isomorphism.
		
	\end{enumerate}
	If these hold then for all $\alpha \in \Zp$, we have $\psi (M_{(\alpha)}) = M_\alpha$. In particular, \[\{\alpha \in \Zp \; |\;  M_{(\alpha)} \neq 0 \} = \{\alpha \in \Zp \; |\;  M_\alpha \neq 0\}.\] 
\end{proposition}

\begin{proof}
	With the notation from Subsection \ref{Subsec-DiffOps_VFilt}, we have that the ideal $\fm_\alpha$ is generated by
	$$\fm_\alpha = (\sigma_{p^i} - \alpha_i \ | \ i \in \Z_{\geq 0} )$$
	and, since $\Delta(\sigma_{p^i}) = s_{p^i}$, we see that $M_{(\alpha)} = \Ann_M (\fm_\alpha)$ for any $(D_{R[\ul t]})_0$-module $M$ and any $\alpha \in \Zp$. 
	
	Suppose that $M$ splits as in part (a). Then for all $\beta \in \Zp$ we have
	$$M_\beta = \big( \bigoplus_{\alpha \in \Zp} M_{(\alpha)} \big)_\beta \cong M_{(\beta)},$$
	and the composition $M \to M_\beta \cong M_{(\beta)}$ is the projection map, which proves (b).
	
	Suppose now that $M$ splits as in part (b). The natural map $M \to \bigoplus_{\alpha} M_\alpha$ is $\cts(\Zp, \F_p)$-linear, and thus for all $\beta \in \Zp$ we have
	$$M_{(\beta)} \cong \big( \bigoplus_{\alpha \in \Zp} M_\alpha \big)_{(\beta)} = M_\beta,$$
	and the composition $M_\beta \cong M_{(\beta)} \to M$ is the inclusion map induced by the direct sum decomposition, which proves (a). 
	
	The last statement follows from the proof. 
\end{proof}

\subsection{Alternative characterization of $N_f$} \ 

	For every $e \geq 0$ let $H^e_f := R_f[t]/(f-t)^{p^e} \delta_{p^e}$, where $\delta_{p^e}$ is a formal symbol. We have maps $H^e_f \to H^{e+1}_f$ given by multiplication by $(f-t)^{p^e(p-1)} = (f^{p^e} - t^{p^e})^{p-1}$, and the limit
$$H_f := \lim_{\rightarrow} (H^0_f \longrightarrow H^1_f \longrightarrow \cds )$$
can be identified with $R_f[t]_{f-t}/R_f[t]$, whereby $\delta_{p^e}$ gets identified with the class of $(f-t)^{-p^e}$. Note that $H^e_f$ has a $D^{(e)}_{R[t]}$-module structure, and the map $H^e_f \to H^{e+1}_f$ is $D^{(e)}_{R[t]}$-linear. This gives the limit $H_f$ a $D_{R[t]}$-module structure, and the isomorphism $H_f \cong R_f[t]_{f-t}/R_f[t]$ is $D_{R[t]}$-linear.

We get an action of the algebra $\cts^e(\Zp, \F_p)$ on $H^e_f$ by restriction of scalars through $\Delta^e$, and an action of the algebra $\cts(\Zp, \F_p)$ on $H_f$  by restriction of scalars through $\Delta$.

\begin{proposition} \label{prop-CtsRf-to-Hf-iso}
	For every $e \geq 0$ there is a unique additive isomorphism 
	$$\Phi^e: \cts^e(\Zp, R_f) \simto H^e_f$$
	that identifies $1$ with $\delta_1$ and that is linear over $R_f$ and over $\cts^e(\Zp, \F_p)$. These isomorphisms glue to give an isomorphism 
	$$\Phi: \cts(\Zp, R_f) \simto H_f$$
	that is linear over $R_f$ and over $\cts(\Zp, \F_p)$. 
\end{proposition}

\begin{proof}
	Fix an integer $e \geq 0$. Given $a \in \{0, \ds, p^e-1\}$ denote by $\chi^{(e)}_a \in \cts^e(\Zp, \F_p)$ the function such that $\chi^{(e)}_a(\beta) = 1$ whenever $\beta \equiv a \mod p^e$ and such that $\chi^{(e)}_a(\beta) = 0$ otherwise. 
	
	We claim that $\Delta(\chi^{(e)}_a)  \in (D^{(e)}_{R[t]})_0$ is the unique $R$-linear operator with the property that, for all $b \in \{0, 1, \ds, p^e-1\}$, $\Delta(\chi^{(e)}_a) \cdot t^{p^e -1 - b} = t^{p^e - 1- b}$ whenever $b = a$ and such that $\Delta(\chi^{(e)}_a) \cdot t^{p^e - 1 - b} = 0$ otherwise. Indeed, by definition $\Delta(\chi^{(e)}_a)$ acts on $t^{p^e - 1- b}$ by the scalar $\chi^{(e)}_a(- 1 - (p^e - 1 - b)) = \chi^{(e)}_a(b - p^e)$, which is $1$ when $b = a$ and $0$ when $b \neq a$. 
	
	We now observe that, given an $e \geq 0$, we have the following equalities in $\cts^e(\Zp, R_f)$ and $H^e_f$ respectively:
	\begin{align*}
	1 & = \sum_{a = 0}^{p^e-1} \chi^{(e)}_a \\
	\delta_1 & = (f - t)^{p^e - 1} \delta_{p^e} = \sum_{ a= 0}^{p^e - 1}  f^a t^{p^e - 1- a} \delta_{p^e}.
	\end{align*}
	(For the last equality, see Lemma \ref{lemma-binomialcoeffs}). It follows that any map $\Phi^{(e)}: \cts^e(\Zp, R_f) \to H^e_f$ with $\Phi^e(1) = \delta_1$ that respects the $R$ and $\cts^e(\Zp, \F_p)$-actions must send $\Phi^e(\chi^{(e)}_a) = T^{e}_a$ where $T^{e}_a := f^a t^{p^e - 1- a} \delta_{p^e}$. Since $\cts^e(\Zp, R_f) = \bigoplus_{ a = 0}^{p^e -1 } R_f \chi^{(e)}_a$ and $H^e_f = \bigoplus_{a = 0}^{p^e - 1} R_f T^e_a$ such a map $\Phi^e$ exists and is indeed an isomorphism. 
	
	Only the claim regarding the gluing remains, which follows from the following identities in $\cts^e(\Zp, R_f)$ and $H^e_f$:
	\begin{align*}
	\chi^{(e)}_a & = \sum_{c = 0}^{p-1} \chi^{(e)}_{a + cp^e} \\
	T^e_a & = f^a t^{p^e - 1- a} (f - t)^{p^e(p-1)} \delta_{p^{e+1}} \\
	& = \sum_{c = 0}^{p-1} f^{a + cp^e} t^{p^{e+1}- 1- (a + cp^e)},
	\end{align*}
	 where we use Lemma \ref{lemma-binomialcoeffs} once again in the last equality.
\end{proof}

By restriction of scalars along $\Delta: \cts(\Zp, D_R) \simto (D_{R[t]})_0$ we view $H_f$ as a $\cts(\Zp, D_R)$-module, and we then transfer this structure along $\Phi: \cts(\Zp, R_f) \simto H_f$ to endow $\cts(\Zp, R_f)$ with a $\cts(\Zp, D_R)$-module structure. The module $\cts(\Zp, R_f)$, viewed as a $\cts(\Zp, D_R)$-module in this way, is denoted by $\cts(\Zp, R_f) \bfs$ and an element $\phi \in \cts(\Zp, R_f)$ is written as $\phi \bfs$ when we want to emphasize that we view it as an element of $\cts(\Zp, R_f) \bfs$. Namely, for $\xi\in \cts(\Zp, D_R)$ and $\phi \bfs\in \cts(\Zp, R_f) \bfs$, we define $\xi \cdot \phi \bfs := \Phi^{-1}(\Delta(\xi) \cdot \Phi(\beta))\bfs$.

With this notation, it follows from Proposition \ref{prop-CtsRf-to-Hf-iso} that we have isomorphisms
\[
\cts(\Zp, D_R) \cdot \bfs  \cong (D_{R[t]})_0 \cdot \delta \quad \text{and} \quad
\frac{\cts(\Zp, D_R) \cdot \bfs}{\cts(\Zp, D_R) \cdot f \bfs}  \cong N_f.\]

Our next goal is to describe the $\cts(\Zp, D_R)$-module structure of $\cts(\Zp, R_f) \bfs$ more explicitly. Note that 
\begin{equation}\label{obs-conj}
a\equiv b  \mod \, p^e \ \text{and} \ \delta\in D^{(e)}_R \quad \text{implies that} \quad f^{-a}\delta f^a = f^{-b} \delta f^b. 
\end{equation} Thus, given a $p$-adic integer $\alpha \in \Zp$ and an operator $\delta \in D^{(e)}_R$, the operators
$$f^{-\aaa} \delta f^{\aaa}$$
are equal for all $a \geq e$. This construction defines a map
$$\Upsilon_{\alpha, f}: D_R \to D_{R_f} \qquad \Upsilon_{\alpha, f} (\delta) = f^{-\aae} \delta f^{\aae} \ \ (e \gg 0).$$
Note that $\Upsilon_{\alpha, f}$ respects the level filtration, and it therefore induces maps $\Upsilon^e_{\alpha, f} : D^{(e)}_R \to D^{(e)}_{R_f}$. 

\begin{lemma}
	For all $\xi \in \cts(\Zp, D_R)$ and $\phi \bfs \in \cts(\Zp, R_f) \bfs$ we have
	$$(\xi \cdot \phi \bfs)(\alpha) = \Upsilon_{\alpha, f}(\xi(\alpha)) \cdot \phi(\alpha).$$
\end{lemma}
\begin{proof}
	We retain the notation from the proof of Proposition \ref{prop-CtsRf-to-Hf-iso}. Let us also temporarily denote by $\xi \star \phi \in \cts(\Zp, R_f)$ the function $\alpha\mapsto \Upsilon_{\alpha, f}(\xi(\alpha)) \cdot \phi(\alpha)$ as before. We need to show that in $H_f$ we have the equality $\Delta(\xi) \Phi(\phi) = \Phi(\xi \star \phi)$ and, since the operation $\star$ is bilinear, it suffices to prove it for $\xi = \delta \chi^{(e)}_a$ and $\phi = g \chi^{(e)}_b$ for some $\delta \in D_R$, $g \in R_f$ and $a,b  \in \{0, \ds, p^e-1\}$, where we retain the notation of the proof of Proposition~\ref{prop-CtsRf-to-Hf-iso}.
	
	First note that for $a \neq b$ we have $\Delta(\chi^{(e)}_a) T^e_b = 0$ and therefore $\Delta(\delta \chi^{(e)}_a) \Phi(g \chi^{(e)}_b) = \delta \Delta(\chi^{(e)}_a) \cdot g T^e_b = 0$, and that $(\delta \xi^{(e)}_a) \star (g \xi^{(e)}_b) = 0$. We may thus assume that $a = b$, in which case we first observe that $(\delta \chi^{(e)}_a) \star (g \chi^{(e)}_a) = f^{-a} (\delta \cdot g f^a) \chi^{(e)}_a$, and we then compute:
	\begin{align*}
	\Delta(\xi \chi^{(e)}_a) \Phi(g \xi^{(e)}_a) & = \delta \Delta(\chi^{(e)}_a) \cdot g T^e_a \\
	& = \delta \cdot g T^e_a \\
	& = (\delta \cdot g f^a) t^{p^e -1 - a} \delta_{p^e} \\
	& = f^{-a} (\delta \cdot g f^a) T^e_a \\
	& = \Phi\big(f^{-a} (\delta \cdot g f^a ) \chi^{(e)}_a \big) \\
	& = \Phi((\delta \chi^{(e)}_a)\star(g \chi^{(e)}_a)). \qedhere
	\end{align*}
\end{proof}

Given a $p$-adic integer $\alpha \in \Zp$, evaluation at $\alpha$ defines a surjective $R_f$-module homomorphism $\cts(\Zp, R_f) \bfs \to R_f$ whose kernel is $\fm_\alpha \cts(\Zp, R_f) \bfs$, and therefore we get an isomorphism
$$\displaystyle \frac{\cts(\Zp, R_f) \bfs}{\fm_\alpha \cts(\Zp, R_f) \bfs} \xrightarrow{\sim} R_f,$$
along which we can transfer the $D_R$-module structure of the left hand side to $R_f$. The module $R_f$, equipped with this exotic $D_R$-module structure, is denoted $R_f \bfa$, where once again $\bfa$ is a formal symbol. 

We describe the $D_R$-module structure more explicitly: given $\delta \in D_R$ and $g \in R_f$ we have
\begin{equation}\label{eq:actionfa}\delta \cdot (g \bfa) = (\Upsilon_{f, \alpha}(\delta) \cdot g) \bfa = f^{-a} (\delta \cdot f^a g) \bfa,
\end{equation}
where $a \in \Z$ is an integer that $p$-adically approximates $\alpha$; more precisely, if $\delta$ has level $e$, then we require that $p^e$ divides $\alpha - a$.

If  $\alpha \in \Z_{(p)}$, then there exist $b > 0$ such that $\alpha (p^b - 1) \in \Z$. Then, $\alpha(p^{eb} - 1) \in \Z$ for all integers $e > 0$. If $\delta \in D^{(e)}_R$, then 
$$\delta \cdot (g \bfa) = f^{\alpha(p^{eb} - 1)} (\delta \cdot f^{-\alpha(p^{eb} - 1)} g) \bfa$$
for all $g \in R_f$.
This shows that $R_f \bfa$ agrees with the $D_R$-module $M_{- \alpha}$ as introduced by Blickle, \Mustata, and Smith \cite{BMSm-hyp} and further studied by the second author and P\'erez \cite{NBP}.

The following lemma justifies the notation $R_f \bfa$.

\begin{lemma}\label{lem:malpha}
	Let $R$ be an $F$-finite ring, $f \in R$. Then:
	\begin{enumerate}[(i)]
		\item\label{item:mal1} For all $\alpha \in \Zp$ the $R_f$-module isomorphism $R_f \boldsymbol{f^{\alpha + 1}} \simto R_f \bfa$ that sends $\boldsymbol{f^{\alpha + 1}} \mapsto f \bfa$ is $D_R$-linear.
		\item\label{item:mal2} For all $n \in \Z$ the $R_f$-linear map $R_f \boldsymbol{f^n} \simto R_f$ that sends $\boldsymbol{f^n} \mapsto f^n$ is $D_R$-linear.
		\item\label{item:mal3} Let $h = f^n$ for some $n \in \Z_{(\geq 0)}$. Then for all $\alpha \in \Zp$ the $R_f$-module isomorphism $R_h \boldsymbol{h^\alpha} \simto R_f \boldsymbol{f^{n \alpha}}$ that sends $\boldsymbol{h^\alpha} \mapsto \boldsymbol{f^{n\alpha}}$ is $D_R$-linear. 
		\item\label{item:mal4} Suppose $R$ is a domain and that for some $h \in \Frac (R)$, some $m \in \Z$ and some $k \in \Z \smallsetminus p\Z$ we have an equality $h^k = f^m$ in $\Frac(R)$. Then the $R_f$-module homomorphism $R_f \boldsymbol{f^{m/k}} \to \Frac(R)$ that sends $f^{m/k} \mapsto h$ is $D_R$-linear.
	\end{enumerate}
\end{lemma}
\begin{proof}
	Fix an operator $\delta \in D_R$ of level $e$, an element $g \in R_f$, $\alpha \in \Zp$ and an integer $a \in \Z$ such that $p^e$ divides $\alpha - a$. Part \ref{item:mal1} follows because the given morphism sends
	\begin{align*}
	\delta \cdot g \boldsymbol{f^{\alpha + 1}} & = f^{-(a + 1)} (\delta \cdot f^{a+1} g) \boldsymbol{f^{\alpha + 1}} \\ 
	& \mapsto f^{-a} (\delta \cdot f^{a+1} g) \bfa \\
	& = \delta \cdot (gf \bfa). 
	\end{align*}
	Similarly, in part \ref{item:mal2} we have
	\begin{align*}
	\delta \cdot g \boldsymbol{f^n} & = f^{-n} (\delta \cdot g f^n) \boldsymbol{f^n}  \\
	& \mapsto f^{-n} (\delta \cdot gf^n) f^n \\
	& = \delta \cdot g f^n.
	\end{align*}
	For part \ref{item:mal3} we note that $p^e$ divides $n \alpha - na$, and we compute:
	\begin{align*}
	\delta \cdot g \boldsymbol{h^\alpha} & = h^{-a} (\delta \cdot g h^a) \boldsymbol{h^\alpha} \\
	& \mapsto f^{-na} (\delta \cdot f^{na} g) \boldsymbol{f^{n \alpha}} \\
	& = \delta \cdot (g \boldsymbol{f^{na}}).
	\end{align*}
	In order to prove part \ref{item:mal4}, we may replace $m$ and $k$ by $nm$ and $nk$ respectively, and we may therefore assume that $k = p^b - 1$. We  take our approximation to $m/k$ to be $a = - \frac{m}{k} (p^{eb}- 1)$; since $(p^{eb} - 1)/k$ is an integer, so is $a$ and, moreover, $f^a = h^{-(p^{eb} - 1)}$.  Note then that the morphism sends
	\begin{align*}
	\delta \cdot g \boldsymbol{f^{m/k}} & = f^{-a} (\delta \cdot f^a g) \boldsymbol{f^{m/k}} \\
	& \mapsto f^{-a} (\delta \cdot f^a g) h \\
	& = h^{p^{eb}} (\delta \cdot h^{-p^{eb} + 1} g) \\
	& = \delta \cdot hg. \qedhere
	\end{align*}
\end{proof}

\begin{proposition} \label{prop:BSrootsandalphajumps}
	Let $R$ be an $F$-finite ring, $f \in R$ be a nonzerodivisor and $\alpha \in \Zp$ be a $p$-adic integer. 
	\begin{enumerate}[(i)]
		\item We have a $D_R$-module isomorphism
		$$\frac{D_R \cdot \bfa}{D_R \cdot f \bfa} \cong (N_f)_\alpha.$$
		\item We have $\bfa \notin D_R \cdot f \bfa$ if and only if $\alpha$ is a Bernstein-Sato root of $f$.
	\end{enumerate}
\end{proposition}
\begin{proof}
	Recall that, given a $\cts(\Zp, \F_p)$-module $M$, we denote by $M_\alpha$ the quotient $M_\alpha = M / \fm_\alpha M$; the functor $(-)_\alpha$ is exact by Lemma \ref{lemma-CmAlpha-flat}. Therefore $(\cts(\Zp, D_R) \cdot \bfs)_\alpha$ is isomorphic to its image in $R_f \bfa$, which is $D_R \cdot \bfa$; similarly, we have a natural isomorphism $(\cts(\Zp, D_R) \cdot f \bfa)_\alpha \cong D_R \cdot f \bfa$. We conclude that
	\begin{align*}
	(N_f)_\alpha & = \lp \frac{\cts(\Zp, D_R) \cdot \bfs}{\cts(\Zp, D_R) \cdot f \bfs} \rp_\alpha \\
	& = \frac{(\cts(\Zp, D_R) \cdot \bfs )_\alpha}{(\cts(\Zp, D_R) \cdot f \bfs )_\alpha} \\
	& \cong \frac{D_R \cdot \bfa}{D_R \cdot f \bfa},
	\end{align*}
	which gives (i). Statement (ii) follows from (i) together with Theorem \ref{thm-BSRoot-equivalence}. 
\end{proof}

We note that the analogue of statement (ii) in the previous proposition does not hold in characteristic zero \cite{Saito}.

Using Proposition \ref{prop:BSrootsandalphajumps}, we show that positive Bernstein-Sato roots abound in rings with certain bad singularities. We recall that a domain $R$ is seminormal whenever, for all $a \in \Frac(R)$ such that $a^2, a^3 \in R$, we have $a \in R$ \cite{Swan80}.

\begin{proposition}\label{prop:seminormal}
	Let $R$ be an $F$-finite domain. If $R$ is not seminormal then, for every $n\in \ZZ_{\geq 2}\smallsetminus p \ZZ$, there is some $f \in R$ such that $1/n$ is a Bernstein-Sato root of $f$.
\end{proposition}
\begin{proof}
	Since $R$ is not seminormal, we may pick some $a \in \Frac(R) \setminus R$ such that $a^2, a^3 \in R$, and therefore $a^k \in R$ for all $k \geq 2$. Let $f = a^n$; we then have $f,fa \in R$ and $a \notin R$. By  Lemma~\ref{lem:malpha}~\ref{item:mal4}, the $R_f$-module homomorphism $R_f \boldsymbol{f^{1/n}} \to \Frac(R)$ that sends $f^{1/n} \mapsto a$ is a $D_R$-linear embedding. Since $D_R \cdot fa \subseteq R$, we have $a\notin D_R \cdot fa$, so $\boldsymbol{f^{1/n}} \notin D_R \cdot f \boldsymbol{f^{1/n}}$. By Proposition~\ref{prop:BSrootsandalphajumps}, we have that $1/n$ is a Bernstein-Sato root of $f$.
\end{proof}

\begin{corollary}
	Let $R$ be an $F$-finite domain and suppose that all Bernstein-Sato roots of all elements of $R$ are nonpositive. Then $R$ is seminormal.
\end{corollary}

\section{$D$-module structure of $R_f f^\alpha$}

\subsection{Bernstein-Sato roots, differential thresholds, and $R_f \bfa$} \

Proposition \ref{prop:BSrootsandalphajumps} tells us that we can characterize the Bernstein-Sato roots of nonzerodivisor $f \in R$ in terms of the modules $R_f \bfa$. In this section we explore how different properties of the modules $R_f \bfa$ reflect on the Bernstein-Sato roots and the differential thresholds of $f$. 		

\begin{theorem} \label{ThmDiscretnessBelow}
	Let $R$ be an $F$-finite $F$-split ring, $f \in R$ be a nonzerodivisor and $\alpha \in \Z_{(p)}$. The following are equivalent:
	\begin{enumerate}[(a)]
		\item We have $R_f \bfa = D_R \cdot f^{-\lfloor \alpha + 1 \rfloor } \bfa$.
		\item The module $R_f \bfa$ is finitely generated over $D_R$.
		\item We have $\BSR(f) \cap \{\alpha - \lfloor \alpha + 1 \rfloor  -1, \alpha -\lfloor \alpha + 1 \rfloor  - 2, \ds \} = \varnothing$.
		\item The set $\BSR(f) \cap \{\alpha-1, \alpha-2, \ds \}$ is finite.
		\item There is some $\varepsilon > 0$ such that the interval $(\lfloor \alpha + 1 \rfloor  - \alpha - \varepsilon, \lfloor \alpha + 1 \rfloor  - \alpha)$ contains no differential thresholds of $f$.
	\end{enumerate}
\end{theorem}
\begin{proof}
	Recall that there is a $D_R$-module isomorphism $R_f \bfa \cong R_f \boldsymbol{f^{- \lfloor \alpha + 1 \rfloor  + \alpha}}$ which identifies $f^{- \lfloor \alpha + 1 \rfloor } \bfa$ with $\boldsymbol{f^{-\lfloor \alpha + 1 \rfloor  + \alpha}}$ (Lemma \ref{lem:malpha}). In particular, $R_f \bfa$ is finitely generated over $D_R$ if and only if $R_f \boldsymbol{f^{- \lfloor \alpha + 1 \rfloor  + \alpha}}$ is finitely generated over $D_R$. We can therefore replace $\alpha$ with $\alpha - \lfloor \alpha + 1 \rfloor $ to assume that $\alpha \in [-1, 0)$.
	
	We first show that (a) is equivalent to (c). Note that $R_f \bfa = \bigcup_{k = 0}^\infty D_R \cdot f^{-k} \bfa$, and therefore  we have $R_f \bfa = D_R \cdot \bfa$ if and only if every inclusion in the chain
	$$D_R \cdot \bfa \sq D_R \cdot f^{-1} \bfa \sq D_R \cdot f^{-1} \bfa \sq \cds $$
	is an equality. By Lemma \ref{lem:malpha}, for each integer $k \geq 0$ we have compatible $D_R$-module isomorphisms $D_R \cdot f^{-k} \bfa \cong D_R \cdot \boldsymbol{f^{\alpha - k}}$. By Proposition \ref{prop:BSrootsandalphajumps} we conclude that, for all $k \geq 1$, $D_R \cdot f^{-k + 1} \bfa = D_R \cdot f^{-k} \bfa$ if and only if $\alpha - k$ is not a Bernstein-Sato root of $f$. 
	
	That (b) is equivalent to (d) is proved similarly: now we observe that $R_f \bfa$ is finitely generated over $D_R$ if and only if the chain above stabilizes, which happens precisely when only finitely many of the inclusions are strict.
	
	Statement (c) implies (d) trivially. To see that (d) implies (c), suppose that $\BSR(f) \cap \{\alpha-1, \alpha-2, \ds, \}$ is nonempty; that is, suppose that there is a Bernstein-Sato root of $f$ of the form $\alpha - k $ for some integer $k \geq 1$. Since $\alpha - k < -1$, we get that $\BSR(f) \cap \{\alpha-k, \alpha-k-1, \ds, \}$ must be infinite by Lemma \ref{lem:seq}.
	
	We thus have that (a), (b), (c), and (d) are equivalent. We now show that (a) implies (e). Fix some $a \in \Z_{> 0}$ such that $\alpha(p^a - 1) \in \Z$. By assumption, there is an operator $\xi \in D_R$ such that $\xi \cdot \bfa = f^{\alpha(p^a - 1)} \bfa$. If we pick $i$ large enough so that $\xi \in D^{((i+1)a)}_R$ then we have 
	$$\xi \cdot \bfa = f^{\alpha(p^{(i+1)a} - 1)} \xi(f^{-\alpha(p^{(i+1)a} - 1)}), $$
	and we conclude that
	$$\xi\big(f^{-\alpha(p^{(i+1)a} - 1)} \big) = f^{\alpha(p^{ia} - 1) p^a}.$$
	Fix some $e$ such that $\xi \in D^{((e+1)a)}_R$; we conclude that the above holds for all $i \geq e$. 
	
	Fix a splitting $\sigma: F_* R \to R$ of the Frobenius morphism $F$; for all integers $n \geq 0$, its $n$-th iteration $\sigma^n: F^n_* R \to R$ is a splitting of $F^n$. We inductively define operators $\xi_k \in D_R$, for $k \geq 1$, by $\xi_1 = \xi$ and $\xi_k = F^{(k-1)a} \ \xi_1 \ \sigma^{(k-1)a} \ \xi_{k-1}$. For all $k \geq 0$, we have $F^{(k-1)a} \ \xi_1 \ \sigma^{(k-1)a} \in D^{((e+k)a)}_R$ and therefore $\xi_k \in D^{((e+k)a)}_R$ by induction. By using induction on $k$ once again we have that, for all $i \geq e$,
	$$\xi_k \big( f^{-\alpha(p^{(i+k)a} - 1)} \big) = f^{-\alpha(p^{ia}-1) p^{ka}}.$$
	By considering the case $i = e$, we conclude that
	$$\cB^\bullet_f(p^{(e+k)a}) \cap \big[-\alpha(p^{ea}-1) p^{ka},  -\alpha(p^{(e+k)a} - 1)\big) = \varnothing$$
	for all $k \geq 0$. By Proposition \ref{prop:nojumpnothreshold}, we conclude that $f$ has no differential thresholds in the interval
	$$ \bigg( - \alpha + \frac{\alpha}{p^{ea}}, - \alpha + \frac{\alpha}{p^{(e+k)a}} \bigg)$$
	and, since this holds for every $k \geq 0$, statement (e) follows. 
	
	Let us now assume (e), and prove (a). Once again, fix $a \in \Z_{>0}$ such that $\alpha(p^a - 1) \in \Z$. We begin by noting that the sequence $[e \mapsto - \alpha(p^{ea} - 1) / p^{ea}]$ increases to $-\alpha$, and thus there is some $e$ large enough so that the interval
	$$\bigg[ \frac{-\alpha (p^{ea} - 1)}{p^{ea}}, - \alpha \bigg)$$
	contains no differential thresholds of $f$. Observe that, given an integer $k \geq 0$ and an integer 
	$$n \in [- \alpha(p^{ea} - 1)p^{ka}, - \alpha(p^{(e+k)a} - 1) - 1]$$
	we have 
	$$\bigg[ \frac{n}{p^{(e+k)a}} , \frac{n + 1}{p^{(e+k)a}} \bigg] \sq \bigg[ \frac{-\alpha (p^{ea} - 1)}{p^{ea}}, - \alpha \bigg).$$
	From Lemma \ref{lem:thresh-nearby} we conclude that, for all $k \geq 0$,
	$$\cB^\bullet_f(p^{(e+k)a}) \cap \big[ - \alpha(p^{ea} - 1) p^{ka}, - \alpha(p^{(e+k)a} - 1) \big) = \varnothing,$$
	and thus there is some differential operator $\xi_k \in D^{((e+k)a)}_R$ such that 
	$$\xi_k\big(f^{-\alpha(p^{(e+k)a} - 1)}\big) = f^{- \alpha(p^{ea} - 1)p^{ka}},$$
	and hence
	\begin{align*}
	\xi_k \cdot \bfa & = f^{\alpha(p^{(e+k)a} - 1)} \xi_k \big( f^{-\alpha(p^{(e+k)a} - 1)} \big) \bfa \\
	& = f^{\alpha(p^{ka} - 1)} \bfa.
	\end{align*} 
	We conclude that $D_R \cdot \bfa$ contains elements of the form $f^{-t} \bfa$, with $t$ arbitrarily large, and thus $D_R \cdot \bfa = R_f \bfa$. 
\end{proof}

\begin{remark}
	For a non-$F$-split ring $R$, and an arbitrary $p$-adic integer $\alpha \in \Zp$, the equivalence $(b) \iff (d)$ still holds. More generally,  $R_f \bfa$ is finitely generated over $D_R$ if and only if we have $R_f \bfa = D_R \cdot f^{-t} \bfa$ for some $t$ large enough, and we have $R_f \bfa = D_R \cdot f^{-t} \bfa$ if and only if all inclusions in the chain
	$$D_R \cdot f^{-t} \bfa \sq D_R \cdot f^{-t - 1} \bfa \sq D_R \cdot f^{-t-2} \bfa \sq \cdots $$
	are equalities which, by Lemma \ref{lem:malpha} and Proposition \ref{prop:BSrootsandalphajumps}, is in turn equivalent to $\BSR(f) \cap \{\alpha - t - 1, \alpha - t -2, \cdots \} = \varnothing.$
\end{remark}

The following corollary provides an extension of a result of Blickle, \Mustata,  and Smith \cite[Theorem 2.11]{BMSm-hyp}. 

\begin{corollary} \label{cor-BSadm-bfa-generates}
	Let $R$ be an $F$-finite $F$-split ring and $f \in R$ is Bernstein-Sato admissible nonzerodivisor. For all $\alpha \in (\Z_{(p)})_{<0}$ we have $R_f \bfa = D_R \cdot \bfa$.
\end{corollary}

Let $R$ be an $F$-finite ring, $f \in R$ be an element and $\alpha \in \Zp$ be a $p$-adic integer. By $(C_\alpha)$ we denote the following chain of inclusions in the module $R_f \bfa$:
\begin{equation*} \tag*{($C_\alpha$)}
	D_R \cdot f \bfa \supseteq D_R \cdot f^2 \bfa \supseteq D_R \cdot f^3 \bfa \supseteq \cds  .
\end{equation*} 

\begin{theorem} \label{thm:descthresh}
	Let $R$ be an $F$-finite $F$-split ring, $f \in R$ be a nonzerodivisor and $\alpha \in \Z_{(p)}$. The following are equivalent:
	\begin{enumerate}[(a)]
		\item The chain $(C_{- \lceil \alpha \rceil + \alpha})$ is constant.
		\item The chain $(C_\alpha)$ stabilizes.
		\item We have $\BSR(f) \cap \{\alpha - \lceil \alpha \rceil + 1, \alpha - \lceil \alpha \rceil +2, \ds \} = \varnothing$.
		\item The set $\BSR(f) \cap \{\alpha+ 1, \alpha+2, \ds, \}$ is finite. 
		\item There is some $\epsilon > 0$ such that the interval $(\lceil \alpha \rceil - \alpha, \lceil \alpha \rceil - \alpha + \epsilon)$ contains no differential thresholds of $f$. 
	\end{enumerate}
\end{theorem}
\begin{proof}
	Recall that there is an $D_R$-module isomorphism $R_f \bfa \cong R_f \boldsymbol{f^{- \lceil \alpha \rceil + \alpha}}$ which identifies $f^{-\lceil \alpha \rceil} \bfa$ with $\boldsymbol{f^{-\lceil \alpha \rceil + \alpha}}$. We conclude that $(C_\alpha)$ stabilizes if and only if $(C_{- \lceil \alpha \rceil + \alpha})$ stabilizes. We may thus replace $\alpha$ with $-\lceil \alpha \rceil + \alpha$ to assume that $\alpha \in (-1, 0]$. 	
	
	The equivalences of (a), (b), (c), and (d) are proved in the same way as in Theorem~\ref{ThmDiscretnessBelow}. Let us show that (a) implies (e). Fix $a \in \Z_{>0}$ such that $\alpha(p^a - 1) \in \Z$. By assumption, we have $D_R \cdot f \bfa = D_R \cdot f^{\alpha(p^a - 1) + p^a} \bfa$ and therefore there is some differential operator $\xi \in D_R$ such that $\xi \cdot f^{\alpha(p^a - 1) + p^a} \bfa = f \bfa$. If $i$ is large enough so that $\xi \in D_R^{((i+1)a)}$ then we have
	\begin{align*}
	\xi \cdot f^{\alpha(p^a - 1) + p^a} \bfa & = f^{\alpha(p^{(i+1)a} - 1)} \ \xi \big( f^{\alpha(p^a - 1) + p^a} \ f^{ - \alpha(p^{(i+1)a - 1})} \big) \bfa \\
	& = f^{\alpha(p^{(i+1)a} - 1)} \ \xi \big( f^{- \alpha(p^{ia} - 1)p^a + p^a} \big) \bfa,
	\end{align*}
	and we conclude that
	$$ \xi \big( f^{- \alpha(p^{ia} - 1)p^a + p^a} \big) = f^{- \alpha(p^{(i+1)a} - 1) + 1}.$$
	If we fix some $e$ such that $\xi \in D^{((e+1)a)}_R$, we conclude that the above holds for all $i \geq e$.
	
	As before, we consider a splitting $\sigma: F_* R \to R$ of the Frobenius morphism $F$, and we inductively build a sequence of differential operators $\xi_k \in D_R$. We set $\xi_1 = \xi$ and, for all $k > 1$, we let $\xi_k = \xi_1 \ F^a \ \xi_{k-1} \ \sigma^a$. By induction on $k$, we have that $\xi_k \in D^{((e+k)a)}_R$ and that
	$$\xi_k \big( f^{- \alpha(p^{ia} - 1) p^{ka} + p^{ka}} \big) = f^{-\alpha(p^{(i+k) a} - 1) + 1}$$
	for all $k \geq 0$ and all $i \geq e$. 
	
	By considering the case $i = e$, we conclude that for every $k \geq 0$ we have
	$$\cB^\bullet_f(p^{(e+k)a}) \cap \big[ - \alpha(p^{(e+k)a} - 1) + 1, - \alpha(p^{ea} - 1) p^{ka} + p^{ka} \big) $$
	and hence, by Proposition \ref{prop:nojumpnothreshold}, the interval
	$$\bigg( - \alpha + \frac{1 + \alpha}{p^{(e+k)a}}, - \alpha + \frac{1+\alpha}{p^{ea}} \bigg)$$
	contains no differential thresholds of $f$. Since this holds for all $k \geq 0$, statement (e) follows.
	
	Now let us assume (e) and prove (a). Once again we fix $a \in \Z_{>0}$ such that $\alpha(p^a - 1) \in \Z$. Pick some $e$ large enough so that the interval $(- \alpha, - \alpha + \frac{1 + \alpha}{p^{ea}} ]$ contains no differential thresholds of $f$. Note that, for all integers $k \geq 0$ and all integers $n$ with
	$$n \in \big[ - \alpha(p^{(e+k)a} - 1) + 1, - \alpha(p^{ea} - 1) p^{ka} + p^{ka} - 1 \big]$$
	we have
	$$\bigg[ \frac{n}{p^{(e+k)a}}, \frac{n+1}{p^{(e+k)a}} \bigg] \sq \bigg(- \alpha, - \alpha + \frac{1 + \alpha}{p^{ea}} \bigg].$$
	From Lemma \ref{lem:thresh-nearby} we conclude that, for all $k \geq 0$,
	$$\cB^\bullet_f(p^{(e+k)a}) \cap \big[ - \alpha(p^{(e+k)a} - 1) + 1, - \alpha(p^{ea} - 1) p^{ka} + p^{ka}  \big) = \varnothing,$$
	and hence there is some differential operator $\xi_k \in D^{((e+k)a)}_R$ such that 
	$$\xi_k \big( f^{- \alpha(p^{ea} - 1) p^{ka} + p^{ka}} \big) = f^{- \alpha(p^{(e+k)a} - 1) + 1}.$$
	
	We conclude that 
	\begin{align*}
	\xi_k \cdot f^{\alpha(p^{ka}-1) + p^{ka}} \bfa & = f^{\alpha(p^{(e+k)a} - 1)} \ \xi_k \big( f^{- \alpha(p^{(e+k)a} - 1)} \ f^{\alpha(p^{ka}-1) + p^{ka}} \big) \bfa \\
	& = f^{\alpha(p^{(e+k)a} - 1)} \ \xi_k \big( f^{- \alpha(p^{ea} - 1)p^{ka} + p^{ka}}  \big) \bfa \\
	& = f \bfa.
	\end{align*}
	We conclude that $f \bfa \in D_R \cdot f^{\alpha(p^{ka}-1) + p^{ka}} \bfa $ for all $k \geq 0$ and, since $\alpha(p^{ka} - 1) + p^{ka} = p^{ka} (\alpha + 1) - \alpha$ and $\alpha + 1 > 0$, we have $f \bfa \in D_R \cdot f^{t} \bfa $ for all $t \geq 1$. 
\end{proof}

\begin{remark}
	For a non-$F$-split ring $R$ and an arbitrary $p$-adic integer $\alpha$ the equivalence $(b) \iff (d)$ remains true. More generally, for all integers $t$ we have
	$$D_R \cdot f^t \bfa = D_R \cdot f^{t+1} \bfa = D_R \cdot f^{t+2} = \cdots$$
	if and only if $\BSR(f) \cap \{\alpha + t, \alpha + t + 1, \dots  \} = \varnothing$ (see Lemma \ref{lem:malpha} and Proposition \ref{prop:BSrootsandalphajumps}).
\end{remark}

\begin{corollary} \label{cor-BSadm-chainconstant}
	Let $R$ be an $F$-finite $F$-split ring and $f \in R$ be a Bernstein-Sato admissible nonzerodivisor. For all $\alpha \in \Z_{(p)}$ we have
	$$D_R \cdot f^{- \lceil \alpha \rceil + 1} \bfa = D_R \cdot f^{-\lceil \alpha \rceil + 2} \bfa = D_R \cdot f^{- \lceil \alpha \rceil + 3} \bfa = \cds $$
\end{corollary}

\begin{corollary} \label{CorAccChain1}
	Let $R$ be an $F$-finite $F$-split ring, $f \in R$ be a nonzerodivisor and $\alpha \in \Z_{(p)} \smallsetminus \Z$. The following are equivalent:
	\begin{enumerate}[(a)]
		\item The chain
		$$\cds \sq D_R \cdot f^{-1} \bfa \sq D_R \cdot \bfa \sq D_R \cdot f \bfa \sq \cds $$
		stabilizes on both sides.
		
		\item We have $\BSR(f) \cap (\alpha + \Z)\sq \{\alpha - \lceil \alpha \rceil \}$. 
		
		\item There is some $\epsilon > 0$ such that the interval $(\lceil \alpha \rceil - \alpha - \epsilon, \lceil \alpha \rceil - \alpha + \epsilon)$ contains no differential thresholds of $f$.
	\end{enumerate}
\end{corollary}
\begin{proof}
	Follows from Theorems \ref{ThmDiscretnessBelow} and \ref{thm:descthresh}, together with the observation that $\lfloor \alpha + 1 \rfloor = \lceil \alpha \rceil$ whenever $\alpha \notin \Z$.
\end{proof}

\begin{corollary} \label{CorAccChain2}
	Let $R$ be an $F$-finite $F$-split ring, $f \in R$ be a nonzerodivisor. The following are equivalent:
	\begin{enumerate}[(a)]
		\item The chain
		$$\cds \sq D_R \cdot f^{-1} \sq R \sq D_R \cdot f \sq \cds $$
		of $D_R$-submodules of $R_f$ stabilizes on both sides.
		
		\item We have $\BSR(f) \cap \Z \sq \{-1, 0 \}$.
		
		\item There is some $\epsilon > 0$ such that $(1-\epsilon, 1 + \epsilon)$ contains no differential thresholds of $f$. 
	\end{enumerate}
\end{corollary}
\begin{proof}
	The statement follows by applying Theorems \ref{ThmDiscretnessBelow} and \ref{thm:descthresh} to $\alpha = 0$, and by using the fact that $(0, \epsilon)$ contains no differential thresholds of $f$ if and only if $(1, 1 + \epsilon)$ contains no differential thresholds of $f$ (Proposition \ref{PropSkoda}).
\end{proof}

\begin{corollary}
	Let $R$ be an $F$-finite $F$-split ring, $f \in R$ be a nonzerodivisor, and $\lambda \in (\Z_{(p)})_{>0}$. If $R_f \boldsymbol{f^{-\lambda}}$ has finite length as a $D_R$-module then $\lambda$ is not an accumulation point of differential thresholds of $f$.  
\end{corollary}
\begin{proof}
	By Lemma \ref{lem:malpha} and Proposition \ref{PropSkoda} we may assume that $\lambda \in (0, 1]$. The statement then follows from Corollaries \ref{CorAccChain1} (applied to $\alpha = - \lambda$) and \ref{CorAccChain2}. 
\end{proof}

\begin{lemma} \label{LemmaSimplicityNalpha}
	Let $R$ be an $F$-finite ring strongly $F$-regular ring, $f \in R$ be a nonzerodivisor, and $\alpha \in (\Z_{(p)})_{\leq 0}$. Then $R_f \bfa$ is simple as a $D_{R_f}$-module.
\end{lemma}
\begin{proof}
Recall that $R$ is $F$-split and simple as a $D_R$-module \cite{Smi95}.
	Fix an integer $a  >0$ such that $\alpha(p^a - 1) \in \Z$. Take some nonzero element $g f^{-k} \bfa \in R_f \bfa$ where $g \in R$ and $k \in \Z_{\geq 0}$. Since $R$ is $D_R$-simple there is some differential operator $\delta \in D_R$ such that $\delta(g) = 1$; fix some integer $e$ large enough so that $\delta \in D^{(ea)}_R$. We then have
	$$\lp f^{-\alpha(p^{ea} - 1)} \delta f^{\alpha(p^{ea} - 1) + k} \rp \cdot g f^{-k} \bfa = \delta(g) \bfa = \bfa,$$
	which shows that every nonzero $D_{R_f}$-submodule of $R_f \bfa$ contains $\bfa$. Since $D_{R_f} \bfa = R_f \bfa$, the result follows.
\end{proof}

The following lemma extends a result of the second author and P\'erez \cite[Corollary 3.18]{NBP} to the singular case. 

\begin{lemma}  \label{LemmaIH}
	Let $R$ be an $F$-finite strongly $F$-regular ring. Let $f \in R$ be a Bernstein-Sato admissible nonzerodivisor and let $\alpha \in (\Z_{(p)})_{\leq 0}$. Then $D_R \cdot f^{-\lceil \alpha \rceil + 1} \bfa$ is contained in every nonzero $D_R$-submodule of $R_f \bfa$. In particular, $D_R \cdot f^{-\lceil \alpha \rceil + 1} \bfa$ is the unique simple nonzero $D_R$-submodule of $R_f \bfa$. 
\end{lemma}
\begin{proof}
Recall that $R$ is $F$-split and simple as a $D_R$-module \cite{Smi95}.
	Take some nonzero $v \in R_f \bfa$. By Lemma \ref{LemmaSimplicityNalpha}, $D_{R_f} \cdot v = R_f \bfa$ and therefore there is some differential operator $\delta \in D_R$ and some integer $k \geq 0$ such that $\delta \cdot v = f^k \bfa$; we may assume that $k \geq - \lceil \alpha \rceil + 1$. By Corollary \ref{cor-BSadm-chainconstant} we have
	\[D_R \cdot f^{- \lceil \alpha \rceil + 1} \bfa = D_R \cdot f^k \bfa \sq D_R \cdot v. \qedhere\]
\end{proof}

\begin{remark} \label{rmk-NBP-module}
	Note that in \cite{NBP} the submodule $D_R \cdot f^{\lceil - \alpha \rceil} \bfa$ is used instead of $D_R \cdot f^{- \lceil \alpha \rceil + 1} \bfa$. But these two submodules are equal: indeed, when $\alpha \notin \ZZ$ we have $\lceil - \alpha \rceil = - \lceil \alpha \rceil + 1$, and when $\alpha \in \ZZ$ the statement follows from Lemma \ref{lem:malpha} together with the fact that $D_R \cdot f = R = D_R \cdot 1$. 
\end{remark}

The following result recovers and extends a previous characterization of $F$-thresholds due to the second author and P\'erez \cite[Theorem 1.1]{NBP}.

\begin{theorem} \label{ThmSimpleRfaBSRDT}
	Let $R$ be an $F$-finite strongly $F$-regular ring. Let $f \in R$ be a Bernstein-Sato admissible nonzerodivisor and $\alpha \in \Z_{(p)} \cap [-1, 0)$. The following are equivalent:
	\begin{enumerate}[(a)]
		\item The module $R_f \bfa$ is not simple over $D_R$. 
		\item We have that $\alpha$ is a Bernstein-Sato root of $f$. 
		\item We have that $- \alpha$ is a differential threshold of $f$. 
	\end{enumerate}
\end{theorem}
\begin{proof}
	The equivalence between (b) and (c) is given in Corollary \ref{cor:principalrootsthresholds}. We show that (a) is equivalent to (b).
	
	By Corollary \ref{cor-BSadm-bfa-generates} we have $R_f \bfa = D_R \cdot \bfa$, and by Lemma \ref{LemmaIH} $D_R \cdot f \bfa$ is the unique simple $D_R$-submodule of $R_f \bfa$ (see Remark \ref{rmk-NBP-module} in the case $\alpha = -1$). Therefore $R_f \bfa$ is not simple if and only if $D_R \cdot \bfa \neq D_R \cdot f \bfa$, and the result follows from Proposition \ref{prop:BSrootsandalphajumps}.
\end{proof}

\subsection{The  length of $R_f\bfa $ } \

In this section we study the structure of $R_f\bfa $ for Bernstein algebras \cite{BernsteinAlgebras}. This is a class of rings whose $D_R$-modules satisfy Bernstein inequality, and as a consequence, there is a notion of holonomic $D_R$-modules.

\begin{setup}\label{SetupBA}
Let $\KK$ be a field and $R$ be a finitely generated graded $\KK$-algebra such that $R_0=\KK$.
Let $\fm$ denote the maximal homogeneous ideal and $w=\max\{j \; |\;  [\fm/\fm^2]_j\neq 0\}.$
\end{setup}

The generalized Bernstein filtration of $R$ with slope $w$, $\cB^\bullet_R$, is defined by 
$$
\cB^i_R=\{\delta\in D_{R} \; | \; \deg(\delta)+w\ord(\delta)\leq i \}.
$$
Since we have fixed the slope, we usually refer only to the Bernstein filtration and do not mention the slope.
The dimension of $\cB^\bullet_R$ is defined by
$$
\Dim(\cB^\bullet_R)=\inf\left\{s\in \RR_{\geq 0} \; | \; \lim\limits_{r\to\infty}\frac{\dim_\KK   \cB^i_R  }{i^r}\right\}.
$$
and  the multiplicity of  $\cB^\bullet_R$ is defined by 
$$
\e(\cB^\bullet_R)=\limsup\limits_{i\to\infty}  \frac{\dim_\KK   \cB^i_R  }{i^{\Dim(\cB^\bullet_R)}}.
$$

\begin{definition}
Let $R$ be as in Setup \ref{SetupBA}. We say that $R$ is a Bernstein algebra 
if it satisfies the following conditions.
	\begin{enumerate}[(i)]
\item There exists $C\in\ZZ_{>0}$ such that for every $\delta\in \cB^i_R$ we have that
$1\in \cB^{Ci}_R \delta \cB^{C i}_R$,
\item $\Dim(\cB^\bullet_R)=2 \dim (R)$, and 
\item $0<\e(\cB^\bullet_R)<\infty$.
\end{enumerate}
\end{definition}

Let $M$ be  a finitely generated $D_R$-module. 
We say a filtration $\cG^\bullet$  is a $B^{\bullet}_R$-filtration if
	\begin{enumerate}[(i)]\setcounter{enumi}{0}
\item Each $\cG^i$ is a finite dimensional $\KK$ vector space,
\item $M=\bigcup_{i\in\NN} \cG^i$, and
\item $\cB^i_R\cG^j\subseteq \cG^{i+j} $ for all $i,j\in \NN$.
\end{enumerate}
Given a filtration, $\cG^i$ of $M$,  
its dimension is defined by
$$
\Dim(\cG^\bullet)=\inf\left\{s\in \RR_{\geq 0} \; | \; \lim\limits_{r\to\infty}\frac{\dim_\KK   \cG^i  }{i^r}\right\},
$$
and its multiplicity by

$$
\e(\cG)=\limsup\limits_{i\to\infty}  \frac{\dim_\KK   \cG^i }{i^{\Dim(\cG^\bullet)}}.
$$

\begin{theorem}[{\cite[Theorem 3.1]{Bav09} $\&$ \cite[Theorem 3.4]{BernsteinAlgebras}}]
%Let $R$ be as in Setup \ref{SetupBA}.  
Suppose that $R$ is a Bernstein algebra.
Let $M$ be a finitely generated $D_R$-module, and 
 $\cG^\bullet$ a $\cB^\bullet_R$ filtration.
 Then, 
$$
\dim(R)\leq   \Dim(\cG^\bullet).
$$
\end{theorem}

\begin{definition}
Suppose that $R$ is a Bernstein algebra.
Let $M$ be a finitely generated nonzero $D_R$-module.
We say that  $M$ is holonomic if it admits a filtration of dimension $\dim(R)$ and finite multiplicity. 
\end{definition}

\begin{theorem}[{\cite[Theorem 3.8]{BernsteinAlgebras}}]\label{ThmBAFinLen}
Suppose that $R$ is a Bernstein algebra.
Let $M$ be a holonomic $D_R$-module.
Then, $M$ has finite length as $D_R$-module.
\end{theorem}

\begin{theorem}\label{ThmRfaBAlg} 
Suppose that $R$ is a Bernstein algebra.
Then, $R_f \bfa$ is holonomic for every nonzero element $f\in R$ and $\alpha\in \widehat{\ZZ}_{(p)}$. In particular, 
$R_f \bfa$ has  finite length as $D_R$-module.
\end{theorem}
\begin{proof}
There exists $C\in \NN$ such that $B^{i}_R\subseteq D^{Ci}_R$ \cite[Proposition 4.14]{BernsteinAlgebras}.
Let $a=\deg(f)$.
Let $\cG^i=\frac{1}{f^{Ci}} {\cB}^{i(Ca+1)}_R\subseteq R_f$.
Then, $R_f$ is holonomic $D_R$-module with the ${\cB}^{\bullet}_R$-filtration $\cG^\bullet$
\cite[Proof of Lemma 4.5]{BernsteinAlgebras}.
Let $A=C(2a+1)+1$ and  $\widetilde{\cG}^i=\cG^{Ai} \bfa\subset R_f\bfa$.

We now show that $\widetilde{\cG}^i$ is a ${\cB}^{\bullet}_R$-filtration.
We have that 
$\widetilde{\cG}^i$ is a finite dimensional $\KK$-vector space and  $R_f \bfa=\bigcup_{i\in\NN} \widetilde{\cG}^i$ because 
$\cG^\bullet$ is a  ${\cB}^{\bullet}_R$-filtration.
It remains to show that $\cB^i_R\widetilde{\cG}^j\subseteq \widetilde{G}^{i+j}$.
We set $e=\lfloor \log_p Ci\rfloor$.
We note that $\alpha_{<e}<p^e$.
Let $\delta \in \cB^i_R\subseteq D^{Ci}_R\subseteq D^{(e)}_R$ and $\frac{g}{f^{Aj}} \bfa\in \widetilde{\cG}^j$.
Then, $\frac{g}{f^{Aj}} \in \cG^{Aj}$. We have that
$$
\delta\left(    f^{\alpha_{<e}}   \frac{g}{f^{Aj}} \right)\in \cG^{Ci+Aj+a\alpha_{<e} }\subseteq \cG^{ Ci+Aj+p^e a}
\subseteq \cG^{ Ci+Ai+Ci a},
$$
  because $\cG^\bullet$ is a  ${\cB}^{\bullet}_R$-filtration.
By the way  $\cG^\bullet$ is defined, we have that 
$$
\frac{1}{f^{\alpha_{<e}}}\delta\left(    f^{\alpha_{<e}}   \frac{g}{f^{Aj}} \right)\in \cG^{ \alpha_{<e} a +Ci+Aj+Ci a}\subseteq \cG^{ p^e a +Ci+Aj+Ci a}\subseteq \cG^{ Ci a +Ci+Aj+Ci a}=\cG^{ Ai+Aj }.
$$
Hence,
$$
\delta\left(     \frac{g}{f^{Ai}} \bfa  \right)=
\frac{1}{f^{\alpha_{<e}}}\delta\left(    f^{\alpha_{<e}}   \frac{g}{f^{An}} \right)\bfa \in \widetilde{\cG}^{ i+j }.
$$
We conclude that $ \cB^i_R   \widetilde{\cG}^{ i}\subseteq  \widetilde{\cG}^{ i+j }.$ 
Then, $R_f \bfa$ is a holonomic $D_R$-module. We conclude that 
$R_f \bfa$ has  finite length as $D_R$-module by Theorem \ref{ThmBAFinLen}.
\end{proof}

%%%%%%%%%%%%%%%%%%%%%%%%%%%%%%%%%%%%%%%%%%%%%%%%%%%%%%%%%%%%%%%%%%%%%%%%%%%%%
\section{Examples} \label{Sec-examples}
%%%%%%%%%%%%%%%%%%%%%%%%%%%%%%%%%%%%%%%%%%%%%%%%%%%%%%%%%%%%%%%%%%%%%%%%%%%%%

\begin{example} \label{example-cone-ell-curve}
	Let $\mathbb{K}$ be a field of characteristic $p \equiv 1 \mod 3$, $\displaystyle R=\frac{\mathbb{K}[x,y,z]}{(x^3+y^3+z^3)}$, and $f\in (x,y,z)$. The ring $R$ has no differential operators of negative degree. It follows that $D^{(e)}_R \cdot f^n \neq D^{(e)}_R \cdot f^{n+1}$ for all $n\geq 0$, so $(f)$ is not Bernstein-Sato admissible. Every $p$-adic integer is a Bernstein-Sato root of $f$, and every nonnegative real number is a differential threshold of $f$.
\end{example}

\begin{example}
	Let $R=\FF_p[x,y,z]$, with $p$ odd, and $\fa=(x^2yz,xy^2z,xyz^2)$. We have that  $\alpha=\frac{-5}{4}$ is a Bernstein-Sato root of $\fa$ \cite[Example~3.5]{QG19b}. However, its negative, $\frac{5}{4}$, is not an $F$-jumping number of $\fa$. To see this, we claim that $\tau_R(\fa^\lambda)=(xyz)$ for $\lambda\in [1,\frac{3}{2})$. Since $\fa \subseteq(xyz)$, we have $\tau_R(\fa^\lambda)\subseteq \tau_R((xyz)^\lambda)=(xyz)$ for $\lambda\in [1,\frac{3}{2})$. It suffices to show that $xyz\in \tau_R(\fa^\lambda)$ for $\lambda<\frac{3}{2}$. For $e>0$, we have $(xyz)^{2p^e-2}=(x^2yz)^{(p^e-1)/2}(xy^2z)^{(p^e-1)/2}(xyz^2)^{(p^e-1)/2}\in \fa^{(3p^e-3)/2}$, so $ xyz\in \cC^e_R \cdot \fa^{(3p^e-3)/2}$. Since $\frac{(3p^e-3)/2}{p^e} \to \frac{3}{2}$, the claim follows.
	
	This example shows that the conclusion of Theorem~\ref{thm:threshsareroots} cannot be strengthened to say that the negative of a Bernstein-Sato root is a differential threshold. We have that $\frac{9}{4}=1-\alpha$ is an $F$-jumping number, and hence a differential threshold \cite[Example~3.5]{QG19b}.
\end{example}

\begin{example}\label{ExToric}
	Let $R=\mathbb{F}_p[x^2,xy,y^2]\subseteq S= \mathbb{F}_p[x,y]$, with $p$ odd. This inclusion is Cartier extensible, level differentially extensible, and is split as $R$-modules. Let $\mathfrak{m}=(x^2,xy,y^2)$ be the homogeneous maximal ideal of $R$, and $\mathfrak{n}=(x,y)$ be the homogeneous maximal ideal of $S$. 
	We have \[\tau_R(\mathfrak{m}^\lambda)=\tau_S((\mathfrak{m}S)^\lambda)\cap R= \tau_S(\mathfrak{n}^{2\lambda})\cap R\]
	for all $\lambda$ \cite[Proposition~5.9]{AMHJNBTW19}. We have $\tau_S(\mathfrak{n}^{\gamma})=\mathfrak{n}^{\lfloor \gamma -1 \rfloor}$ for $\gamma\geq 1$, and $\tau_S(\mathfrak{n}^{\gamma})=S$ for $0\leq \gamma<1$. We observe that $\mathfrak{n}^{\gamma} \cap R= \mathfrak{m}^{\lceil \gamma/2 \rceil}$. Thus, the $F$-jumping numbers of $\mathfrak{m}S$ are $\{1,\frac{3}{2},2,\frac{5}{2},3,\frac{7}{2},\dots\}$, and the $F$-jumping numbers of $\mathfrak{m}$ are $\{1,2,3,\dots\}$.
	
	By Proposition~\ref{PropDTFinv} (iv), the differential thresholds of $\fm S$ are $\{1,\frac{3}{2},2,\frac{5}{2},3,\frac{7}{2},\dots\}$. Then, by Theorem~\ref{thm:directsummand}~\ref{item:threshde}, the differential thresholds of $\fm$ are $\{1,\frac{3}{2},2,\frac{5}{2},3,\frac{7}{2},\dots\}$.
	
	To compute the Bernstein-Sato roots of $\fm$, we may equivalently compute the Bernstein-Sato roots of $\fm S=\fn^2$ by Theorem~\ref{thm:directsummand}~\ref{item:rootde}. We have that 
	\[D^{(e)}_S \cdot \fn^n=\begin{cases} S & \text{if } 0\leq n\leq 2p^e-2 \\
	(\fn^{a-2})^{[p^e]} & \text{if } (a-1)p^e-1\leq n\leq ap^e-2, \ a> 2.\end{cases}\]
	Thus, the differential jumps of level $e$ for $\fn^2$ are \[\{\lfloor a p^e-2 \rfloor \, | \, a\geq 2\} = \{ bp^e-1 \, | \, b\geq 1\} \cup \left\{ \frac{2c+1}{2}p^e-\frac{3}{2} \, | \, c\geq 1\right\}.\]
	Passing to $p$-adic limits, we conclude that the set of Bernstein-Sato roots of $\fm$ is $\{-1,-\frac{3}{2}\}$.

	This shows that an ideal may have Bernstein-Sato roots that are not congruent modulo $\mathbb{Z}$ to $F$-jumping numbers. This also shows that an ideal may have differential thresholds that are not $F$-jumping numbers, even in a strongly $F$-regular hypersurface.

\end{example}

\begin{example}
	Let $R$ and $S$ be as in the previous example. Take $f=x^4+y^6$, and $p\equiv 1 \mod 12$. We claim that $\tau_S(f^{\frac{7}{12} - \varepsilon})=(x,y)$ for $0<\varepsilon\ll 1$, and $\tau_S(f^{\frac{7}{12} })=(x,y^2)$. 
	
	To see this, first, we observe that for $n=\frac{7 p^e -7}{12}$, $y\in \cC^e_R \cdot f^n$. Note that the monomials in the binomial expansion are distinct, and that none of the exponents of different terms agree modulo $p^e$. In the binomial expansion of $f^n$, for $j=\frac{p^e-1}{4}$, we have $n-j=\frac{p^e-1}{3}$, and since $j=(\frac{p-1}{4})+(\frac{p-1}{4})p+\cdots+(\frac{p-1}{4})p^{e-1}$ and $n-j=(\frac{p-1}{3})+(\frac{p-1}{3})p+\cdots+(\frac{p-1}{3})p^{e-1}$, $j$ and $n-j$ add without carrying, so $\binom{n}{j}$ is nonzero modulo $p$. Then $\binom{n}{j} (x^4)^j (y^6)^{n-j}$ is a unit times $x^{p^e-1} y^{2p^e-2}$, justifying the first observation.
	
	Second, if $n\geq\frac{7 p^e}{12}$, for $0\leq j \leq n$, we must have either $4j \geq p^e$ or $6(n-j)\geq 2p^e$, so any monomial in the binomial expansion of $f^n$ lies in $(x,y^2)^{[p^e]}$.
	
	Third, if $n=\frac{2{(p^e-1)}}{3}$, taking $j=\frac{p^e-1}{2}$ we have $n-j=\frac{p^e-1}{6}$;  the integers $j$ and $n-j$ add without carrying as before, so there is a term in the expansion that is a unit times $x^{2 p^e -2} y^{p-1}$. We conclude that $x \in \cC^e_R\cdot f^n$. Similarly, one checks that $y^2 \in \cC^e_R\cdot f^n$ for $n=\frac{3(p^e-1)}{4}$.
	
	Put together, these justify the claims on $\tau_S$. Now, since $(x,y)\cap R = (x,y^2) \cap R$, $\frac{7}{12}$ is a jumping number of $fS$, but not of $fR$. Note that $\frac{7}{12}$ is a differential threshold of $fR$ by Theorem~\ref{thm:directsummand}~\ref{item:threshde} and \ref{PropDTFinv}.
	
	This shows that Bernstein-Sato roots and $F$-jumping numbers of principal ideals do not necessarily agree modulo $\ZZ$; likewise, differential thresholds and $F$-jumping numbers of principal ideals do not necessarily agree.
\end{example}

\begin{example}
	Let $\mathbb{K}$ be a field of characteristic $p>0$, and $\displaystyle R=\frac{\mathbb{K}[x,y]}{(xy)}$. The set of Bernstein-Sato roots of $x\in R$ is $\{-1,0\}$. To see this, consider the decomposition of $R$ as an $R^{p^e}$-module:
	\[ R = R^{p^e} \cdot 1 \ \oplus \ \bigoplus_{i=1}^{{p^e}-1} (R/yR)^{p^e} \cdot x^i \ \oplus \ \bigoplus_{j=1}^{{p^e}-1} (R/yR)^{p^e} \cdot y^j. \]
	From this, we compute that, for $0\leq j < p^e$,
	\[D^{(e)}_R \cdot x^{a p^e +j} = \begin{cases} 
	(x^{a p^e}) &\text{if } j=0 \\
	(x^{a p^e +1}) &\text{if } j\neq 0,\end{cases}\]
	so $\mathcal{B}^\bullet_x({p^e})\cap [0,{p^e}) = \{0,\,{p^e}-1\}$. 
	
	Passing to $p$-adic limits, we find that the Bernstein-Sato roots are $\{0,-1\}$ as claimed.	The differential thresholds of $x$ are $\{0,1,2,\dots\}$.
	
	Thus, for a Bernstein-Sato admissible ideal in an $F$-split ring, zero can occur as a Bernstein-Sato root. The occurrence of zero as a root here is explained by Proposition~\ref{prop:d-simple}.
\end{example}

\begin{example}
	Let $\mathbb{K}$ be a field of characteristic $p>2$, and $R=\mathbb{K}[x^2,x^3]$. The set of Bernstein-Sato roots of $x^2\in R$ is $\{-1,\frac{1}{2}\}$. To see this, consider the decomposition of $R$ as an $R^{p^e}$-module:
	\[ R = \overline{R}^{p^e} \cdot 1 \ \oplus \ \bigoplus_{i=2}^{{p^e}-1} \overline{R}^{p^e} \cdot x^i  \ \oplus \   \overline{R}^{p^e} \cdot x^{{p^e}+1},\]
	where $\overline{R}=\mathbb{K}[x]$ is the normalization of $R$. Then $D^{(e)}_R$ is a direct sum of copies of $E:=\mathrm{End}_{R^{p^e}}(\overline{R}^{p^e})$. We then have, for $0\leq j <p^e$,
		\[D^{(e)}_R \cdot (x^2)^{ap^e+j} = \begin{cases} 
	(x^{2 ap^e}) &\text{if } 0\leq j \leq \frac{{p^e}+1}{2}  \\
	(x^{2 ap^e})(E\cdot x^{p^e}) &\text{if } \frac{{p^e}+1}{2} < j <p^e\end{cases},\]
	so $\mathcal{B}^\bullet_x({p^e})\cap [0,{p^e}) = \{\frac{{p^e}+1}{2},\,{p^e}-1\}$. 
	
	Passing to $p$-adic limits, we find that the Bernstein-Sato roots are $\{ -1 , \frac{1}{2}\}$, as claimed. The positive root here is explained by Proposition~\ref{prop:seminormal}.
	The differential thresholds of $x^2$ are $\{\frac{1}{2},1,\frac{3}{2},2,\dots\}$.
	
	This illustrates that the $F$-split hypothesis in Theorems \ref{thm:rational} and \ref{thm:threshsareroots} is necessary.
\end{example}

\begin{example} Let $\mathbb{K}$ be a field of characteristic $2$, and $R=\mathbb{K}[x^2,x^3]$. The set of Bernstein-Sato roots of $x^2\in R$ is $\{-1\}$. We have the decomposition of $R$ as an $R^{p^e}$-module as in the previous example, and then, for $0\leq j < p^e$,
	\[D^{(e)} \cdot (x^2)^{ap^e+j} = \begin{cases} 
		(x^{2 ap^e}) &\text{if } 0\leq j \leq \frac{{p^e}}{2}-1  \\
	(x^{2 ap^e})(E\cdot x^{p^e}) &\text{if } \frac{{p^e}}{2}-1 < j <p^e,\end{cases}\]
	so $\mathcal{B}^\bullet_x({p^e})\cap [0,{p^e}) = \{\frac{{p^e}}{2}-1,\,{p^e}-1\}$. The only Bernstein-Sato root of $x^2$ is $-1$, while the set of differential thresholds is $\{\frac{1}{2},1,\frac{3}{2},2,\dots\}$.
\end{example}

\begin{example}
	Let $R=K[x]/(x^{n+1})$. For $e$ such that $p^e>n$, we have $D^{(e)}_R=\End_K(R)$, and hence $D^{(e)}_R \cdot x^j = R$ for $j\leq n$, and $D^{(e)}_R \cdot x^j=0$ for $j>n$. Thus, $\cB_{x}^{\bullet}(p^e)=\{n\}$ for all $e\gg 0$. We then have that $n$ is the unique Bernstein-Sato root, and the only differential threshold is zero.
\end{example}

	\bibliographystyle{alpha}
	\bibliography{biblio.bib}

\end{document}